\documentclass[reqno]{amsart}
\usepackage{amsfonts}
\usepackage{mathrsfs}
\usepackage{color}
\usepackage{cite}
\usepackage{amscd}
\usepackage{latexsym}
\usepackage{amsfonts}
\usepackage{graphicx}
\usepackage{CJK,indentfirst,amsmath,amsfonts,amssymb,amsthm,cite,cases,subeqnarray,setspace}

\begin{document}

\title[]
{Long wavelength limit\\ for the quantum Euler-Poisson equation}
\author{Huimin Liu and Xueke Pu}

\address{Huimin Liu \newline
Department of Mathematics, Chongqing University, Chongqing 401331, P.R.China}

\address{Xueke Pu \newline
Department of Mathematics, Chongqing University, Chongqing 401331, P.R.China} \email{ xuekepu@cqu.edu.cn}

\thanks{This work is supported in part by NSFC (11471057) and Natural Science Foundation Project of CQ CSTC (cstc2014jcyjA50020).}
 \keywords{}

\begin{abstract}
In this paper, we consider the long wavelength limit for the quantum Euler-Poisson equation. Under the Gardner-Morikawa transform, we derive the quantum Korteweg-de Vries (KdV) equation by a singular perturbation method. We show that the KdV dynamics can be seen at time interval of order $O(\epsilon^{-3/2})$. When the nondimensional quantum parameter $H=2$, it reduces to the inviscid Burgers equation.
\end{abstract}

\maketitle \numberwithin{equation}{section}
\newtheorem{proposition}{Proposition}[section]
\newtheorem{theorem}{Theorem}[section]
\newtheorem{lemma}[theorem]{Lemma}
\newtheorem{remark}[theorem]{Remark}
\newtheorem{hypothesis}[theorem]{Hypothesis}
\newtheorem{definition}{Definition}[section]
\newtheorem{corollary}{Corollary}[section]
\newtheorem{assumption}{Assumption}[section]
\section{Introduction}
\setcounter{section}{1}\setcounter{equation}{0}
In this paper, we consider a one-dimensional two species quantum plasma system made by one electronic and one ionic fluid, in the electrostatic approximation \cite{HGG3}. For simplicity, we only consider the continuity and momentum equations and ignore the energy transport equation, which are sufficient to describe the classical ion-acoustic waves \cite{N4}. The system is governed by the following equations
\begin{subequations}\label{mm2}
\begin{numcases}{}
\partial_{t}n_{e}+\partial_{x}(n_{e}u_{e})=0, \label{mm2-1}\\
\partial_{t}n_{i}+\partial_{x}(n_{i}u_{i})=0, \label{mm2-2}\\
\partial_{t}u_{e}+u_{e}\partial_{x}u_{e} =\frac{e}{m_{e}}\partial_{x}\phi -\frac{1}{m_{e}n_{e}}\partial_{x}P
+\frac{\hbar^{2}}{2m_{e}^{2}} \partial_{x} \left(\frac{\partial_{x}^{2}\sqrt{n_{e}}} {\sqrt{n_{e}}}\right),\label{mm2-3}\\
\partial_{t}u_{i}+u_{i}\partial_{x}u_{i}=-\frac{e}{m_{i}}\partial_{x}\phi,\label{mm2-4}\\
\partial_{x}^{2}\phi=\frac{e}{\epsilon_{0}}(n_{e}-n_{i}),\label{mm2-5}
\end{numcases}
\end{subequations}
where $n_{e,i}$ are the electronic and ionic number densities, $u_{e,i}$ the electronic and ionic velocities, $\phi$ the scalar potential, $m_{e,i}$ the electron and ion masses, $-e$ the electron charge, $\hbar=\frac{h}{2\pi}$, where $h$ is Planck's constant and $\epsilon_{0}$ the vacuum permittivity. The electron fluid pressure $P=P(n_{e})$, modeled by the equation of state for a one dimensional zero-temperature Fermi gas, is given by
\begin{equation}\label{equ1'''}
\begin{split}
P=\frac{m_{e}v_{F_{e}}^{2}}{3n_{0}^{2}}n_{e}^{3},
\end{split}
\end{equation}
where $n_{0}$ is the equilibrium density for both electrons and ions, and $v_{F_{e}}$ is the electrons Fermi velocity, related to the Fermi temperature $T_{F_{e}}$ by $m_{e}v_{F_{e}}^{2}=\kappa_{B}T_{F_{e}}$, where $\kappa_{B}$ is the Boltzmann constant. Throughout this paper, we assume such a cubic law for the electron fluid pressure, which is the most important significant physical case, as pointed out by Jackson \cite{Jackson62,GHZ15}.

Equations \eqref{mm2-1} and \eqref{mm2-2} represent conservation of charge and mass. Equations \eqref{mm2-3} and \eqref{mm2-4} account for momentum balance. The third order term in \eqref{mm2-3}, proportional to $\hbar^{2}$, takes into account the influence of quantum diffraction effects. However, the motion of ion can be taken as classical in view of the high ion mass in comparison to the electron mass. Accordingly, \eqref{mm2-4} contains no quantum terms. Finally, \eqref{mm2-5} is Poisson's equation, describing the self-consistent electrostatic potential.

Take the following rescaling,
\begin{equation}\label{equ11''}
\begin{split}
&\bar{x}=\frac{\omega_{p_{e}}x}{v_{F_{e}}}, \ \ \bar{t}=\omega_{p_{i}}t,\ \ \bar{n_{e}}=\frac{n_{e}}{n_{0}}, \ \ \bar{n_{i}}=\frac{n_{i}}{n_{0}}, \\
&\bar{u_{e}}=\frac{u_{e}}{c_{s}}, \ \ \bar{u_{i}}=\frac{u_{i}}{c_{s}}, \ \ \bar{\phi}=\frac{e\phi}{\kappa_{B}T_{F_{e}}},
\end{split}
\end{equation}
where $\omega_{p_{e}}$ and $\omega_{p_{i}}$ are the corresponding electron and ion plasma frequencies and $c_s$ is the quantum ion-acoustic velocity, given by
\begin{equation}\label{equ12''}
\begin{split}
\omega_{p_{e}}=\left(\frac{n_{0}e^{2}}{m_{e} \epsilon_{0}}\right)^{1/2}, \ \ \omega_{p_{i}}=\left(\frac{n_{0}e^{2}}{m_{i} \epsilon_{0}}\right)^{1/2}, \ \ c_{s}=\left(\frac{\kappa_{B}T_{F_{e}}}{m_{i}}\right)^{1/2}.
\end{split}
\end{equation}
In addition, consider nondimensional parameter $H={\hbar\omega_{p_{e}}}/{\kappa_{B}T_{F_{e}}}$. Physically, $H$ is the ratio between the electron plasmon energy and the electron Fermi energy. Using the new variables and dropping bars for simplifying natation, we obtain from \eqref{mm2-3}
\begin{equation}\label{equ14'}
\begin{split}
\frac{m_{e}}{m_{i}}(\partial_{t}u_{e}+u_{e}\partial_{x}u_{e})=\partial_{x}\phi-n_{e}\partial_{x}n_{e}
+\frac{H^{2}}{2}\partial_{x}\left(\frac{\partial_{x}^{2}\sqrt{n_{e}}}{\sqrt{n_{e}}}\right).
\end{split}
\end{equation}
Since ${m_{e}}/{m_{i}}\ll1$, we let the left-hand side of \eqref{equ14'} to be zero and then integrate about $x$ with the boundary conditions $n_{e}=1$, $\phi=0$ at infinity, to obtain
\begin{equation}\label{equ14''}
\begin{split}
\phi=-\frac{1}{2}+\frac{1}{2}n_{e}^{2}-\frac{H^{2}}{2\sqrt{n_{e}}}\partial_{x}^{2}\sqrt{n_{e}}.
\end{split}
\end{equation}
This last equation is the electrostatic potential in terms of the electron density and its derivatives. Even when the quantum diffraction effects are negligible $(H=0)$, the electron equilibrium is given by a Fermi-Dirac distribution and not by a Maxwell-Boltzmann one.

Applying the rescaling \eqref{equ11''} to \eqref{mm2-2}, \eqref{mm2-4} and \eqref{mm2-5}, we have by dropping the bars
\begin{subequations}\label{equ1}
\begin{numcases}{}
\partial_{t}n_{i}+\partial_{x}(n_{i}u_{i})=0,\label{equ1-1}\\
\partial_{t}u_{i}+u_{i}\partial_{x}u_{i}=-\partial_{x}\phi,\label{equ1-2}\\
\partial_{x}^{2}\phi=n_{e}-n_{i},\label{equ1-3}
\end{numcases}
\end{subequations}
Equations \eqref{equ1-1}-\eqref{equ1-3}, together with \eqref{equ14''}, provide a reduced model of four equations with four unknown quantities, $n_{i}$, $u_{i}$, $n_{e}$ and $\phi$. This reduced model is the basic model to be studied in the following, which will lead to the quantum Korteweg-de Vries (KdV) equation \eqref{kdv} under the Gardner-Morikawa transform \cite{GM,SG69}.

Obviously, the reduced system \eqref{equ14''}-\eqref{equ1} admits the homogeneous equilibrium solution $(n_e,n_i,u_i,\phi)=(1,1,0,0)$. Global existence of smooth solutions around the equilibrium is an outstanding difficult problem for the Euler-Poisson problem. Without quantum effects, Guo \cite{Guo98} firstly obtained global irrotational solutions with small velocity for the 3D electron fluid, based on the Klein-Gordon effect. Then, Jang, Li, Zhang and Wu \cite{Jang12,JLZ14,LW14} obtained global smooth small solutions for the 2D electron fluid in Euler-Poisson system. Very recently, Guo, Han and Zhang \cite{GHZ15} finally completely settled this problem and proved that no shocks form for the 1D Euler-Poisson system for electrons. For Euler-Poisson equation for ions, Guo and Pausader \cite{GP11} constructed global smooth irrotational solutions with small amplitude for ion dynamics. For the Euler-Poisson system \eqref{equ1} with quantum effects, there is no existence result, to the best knowledge of the authors.

To access weakly nonlinear solutions for the quantum ion-acoustic system \eqref{equ14''}-\eqref{equ1}, a singular perturbation method can be applied to the weakly nonlinear classical waves, which finally leads to the quantum KdV equation. For details, see Section 2. To this aspect, one may refer to the recent papers \cite{LLS,P,GP14}. In particular, Guo and Pu established rigorously the KdV limit for the ion Euler-Poisson system in 1D for both the cold and hot plasma case, where the electron density satisfies the classical Maxwell-Boltzmann law. This result was generalized to the higher dimensional case in \cite{P}, and the 2D Kadomtsev-Petviashvili-II (KP-II) equation and the 3D Zakharov-Kuznetsov equation are derived for well-prepared initial data under different scalings. Almost at the same time, \cite{LLS} also established the KdV limit in 1D and the Zakharov-Kuznetsov equation in 3D from the Euler-Poisson system. Han-Kwan \cite{HK} also introduced a long wave scaling for the Vlasov-Poisson equation and derived the KdV equation in 1D and the Zakharov-Kuznetsov equation in 3D using the modulated energy method. For other studies for the Euler-Poisson system or related models, the interested readers may refer to \cite{PG15,CG00,ELT01,G,LT02,LT03}, to list only a few. For derivation of the KdV equation from the water waves without surface tension, see \cite{SW00} and the references therein.

In the present paper, we will continue to study the long wavelength limit for the reduced system \eqref{equ1} for ions with quantum effects. Under the Gardner-Morikawa transform, the quantum KdV equation is derived when $H>0$ and $H\neq2$. But when $H=2$, the quantum KdV equation \eqref{kdv} reduces to the inviscid Burger's equation. The formal derivation of the quantum KdV equation can be found in \cite{HGG3} and is given in the next section. The main interest in this paper is to make such a formal derivation rigorous. To do so, we need to obtain uniform (in $\epsilon$) estimates for the remainders $\left(n_{eR}^{\epsilon},n_{iR}^{\epsilon}, u_{iR}^{\epsilon}\right)$ and then recover the uniform estimates of $\phi^{\epsilon}_R$ from the relation \eqref{equ14''}. To apply the Gronwall inequality to complete the proof, we define the triple norm
\begin{equation}\label{tri}
\begin{split}
|\!|\!|(N_i,N_{e},U)|\!|\!|_{\epsilon}^{2}= &\|(N_i,N_{e},U)\|_{H^{2}}^{2} +\epsilon\|(\partial_{x}^{3}N_{e}, \partial_{x}^{3}U)\|_{L^{2}}^{2}\\
&+\epsilon^{2}\|\partial_{x}^{4}N_{e}\|_{L^{2}}^{2} +\epsilon^{3}\|\partial_{x}^{5}N_{e}\|_{L^{2}}^{2}
+\epsilon^{4}\|\partial_{x}^{6}N_{e}\|_{L^{2}}^{2},
\end{split}
\end{equation}
which depends on the parameter $\epsilon$ in the Gardner-Morikawa transform. But we regard $H$ as a fixed constant. After careful computations, we finally close the estimates in this triple norm, which gives uniform (in $\epsilon$) estimates for the remainders $(N_i,N_{e},U)$ in $H^2$ and completes the proof. The main result is stated in Theorem \ref{thm1}. Furthermore, this implies that
\begin{equation}
\sup_{[0,\epsilon^{-3/2}\tau]}
\left\|\left(
\begin{array}{c}
(n_{i}-1)/\epsilon\\ (n_{e}-1)/\epsilon\\ u_{i}/\epsilon
\end{array}
\right)-KdV\right\|_{H^2}\leq C\epsilon,
\end{equation}
for some $C>0$ independent of $\epsilon>0$, for any fixed $\tau>0$ of order $O(1)$. Here the `KdV' stands for the first approximation of $(n_i,n_e,u_i)$ under the Gardner-Morikawa transform in \eqref{equ39}. It shows that the KdV dynamics can be seen at time interval of order $O(\epsilon^{-3/2})$. The result also applies to the case when $H=2$, where the inviscid Burger's equation is derived.

The results in this paper can be generalized to the following general cases. Firstly, for definiteness, we let the electron pressure satisfies the cubic law in \eqref{equ1'''}, but the result in this paper can be generalized to general $\gamma$-law, which will lead to a different relation between $\phi$ and $n_e$ in \eqref{equ14''}. Secondly, the ion momentum equation \eqref{mm2-4} does not contain ion pressure, which generally depends on ion density with the form $P_i(n_i)=T_i\ln n_i$. This paper corresponds to the cold ion case $T_i=0$. But the result in this paper can be generalized to general case $T_i>0$, and indeed, the proof will be slightly simpler since in this case, the system is Friedrich symmetrizable. The result in this paper can be also generalized to the general $\gamma$-law of the ion pressure, i.e., when $P_i(n_i)=T_in_{i}^{\gamma}$ for $\gamma\geq1$. For clarity, we will not mention these general cases in the rest of the paper and concentrate on the case $P(n_e)\sim n_e^3$ in \eqref{equ1'''} and zero ion temperature case $T_i=0$.

This paper is organized as follows. In Section 2, we present the formal derivation of the quantum KdV equation \eqref{kdv} and state the main result in Theorem \ref{thm1}. In Section 3, we present uniform estimates for the remainders in \eqref{cut}. The main estimates are stated in Proposition \ref{P1-P2} and \ref{P3}. Finally, we complete the proof in Section 4.

\section{Formal expansion and Main results}

\subsection{Formal KdV expansion}\label{1.1}
By the classical Gardner-Morikawa transformation \cite{GM,SG69}
\begin{equation}\label{equ39}
x\rightarrow\epsilon^{\frac{1}{2}}(x-t),\  \,t\rightarrow\epsilon^{\frac{3}{2}}t,
\end{equation}
we obtain from \eqref{equ1} the parameterized system
\begin{equation}\label{equ2}
\begin{cases}
\epsilon\partial_{t}n_{i}-\partial_{x}n_{i} +\partial_{x}(n_{i}u_{i})=0,\\
\epsilon\partial_{t}u_{i}-\partial_{x}u_{i} +u_{i}\partial_{x}u_{i}=-\partial_{x}\phi,\\
\epsilon\partial_{x}^{2}\phi=n_{e}-n_{i},
\end{cases}
\end{equation}
where $\epsilon$ is the amplitude of the initial disturbance and is assumed to be small compared with unity and \eqref{equ14''} is rescaled into the following relation
$$\phi=-\frac{1}{2}+\frac{1}{2}n_{e}^{2} -\frac{\epsilon H^{2}}{2\sqrt{n_{e}}}\partial_{x}^{2}\sqrt{n_{e}}.$$

We consider the following formal expansion around the equilibrium solution $(n_{i},n_{e},u_{i})=(1,1,0)$,
\begin{equation}\label{expan-formal}
\begin{cases}
n_{i}=1+\epsilon n_{i}^{(1)}+\epsilon^{2}n_{i}^{(2)}+\epsilon^{3}n_{i}^{(3)}+\epsilon^{4}n_{i}^{(4)}+\cdots,\\
n_{e}=1+\epsilon n_{e}^{(1)}+\epsilon^{2}n_{e}^{(2)}+\epsilon^{3}n_{e}^{(3)}+\epsilon^{4}n_{e}^{(4)}+\cdots,\\
u_{i}=\epsilon u_{i}^{(1)}+\epsilon^{2}u_{i}^{(2)}+\epsilon^{3}u_{i}^{(3)}+\epsilon^{4}u_{i}^{(4)}+\cdots.
\end{cases}
\end{equation}
Plugging \eqref{expan-formal} into \eqref{equ1}, we get a power series of $\epsilon$, whose coefficients depend on $(n_{i}^{(k)},n_{e}^{(k)},u_{i}^{(k)})$ for $k=1,2,\cdots$.

At the order $O(1)$, the coefficients are automatically balanced.

At the order $O(\epsilon)$, we obtain
\begin{subequations}\label{e1}
\begin{numcases}{(\mathcal S_0)}
-\partial_{x}n_{i}^{(1)}+\partial_{x}u_{i}^{(1)}=0,\\
-\partial_{x}u_{i}^{(1)}=-\partial_{x}n_{e}^{(1)},\\
0=n_{e}^{(1)}-n_{i}^{(1)}.
\end{numcases}
\end{subequations}
This enables us to assume the relation
\begin{equation}\label{equ3}
{(\mathcal L_1):\ \ \ \ }
n_{e}^{(1)}=n_{i}^{(1)}=u_{i}^{(1)},
\end{equation}
which makes \eqref{e1} valid and shows that the mode is quasi-neutral in a first approximation. Then only $n_{i}^{(1)}$ needs to be determined.

At the order $O(\epsilon^2)$, we obtain
\begin{subequations}\label{e2}
\begin{numcases}{(\mathcal S_1)}
\partial_{t}n_{i}^{(1)}-\partial_{x}n_{i}^{(2)}+\partial_{x}u_{i}^{(2)}+\partial_{x}(n_{i}^{(1)}u_{i}^{(1)})=0\label{e2-1},\\
\partial_{t}u_{i}^{(1)}-\partial_{x}u_{i}^{(2)}+u_{i}^{(1)}\partial_{x}u_{i}^{(1)}=-\partial_{x}n_{e}^{(2)}-n_{e}^{(1)}\partial_{x}n_{e}^{(1)}
+\frac{H^{2}}{4}\partial_{x}^{3}n_{e}^{(1)}\label{e2-2},\\
\partial_{x}^{2}n_{e}^{(1)}=n_{e}^{(2)}-n_{i}^{(2)}.\label{e2-3}
\end{numcases}
\end{subequations}
Differentiating \eqref{e2-3} with respect to $x$, and then adding the resultant and \eqref{e2-1} to \eqref{e2-2} together, we deduce that $n_{i}^{(1)}$ satisfies the quantum Kortweg-de Vries equation
\begin{equation}\label{kdv}
\partial_{t}n_{i}^{(1)}+2n_{i}^{(1)}\partial_{x}n_{i}^{(1)}+\frac{1}{2}(1-\frac{H^{2}}{4})\partial_{x}^{3}n_{i}^{(1)}=0,
\end{equation}
where we have used the relation \eqref{equ3}. We note that the system \eqref{equ3}, \eqref{kdv} are self-contained, which do not depend on $(n_{i}^{(j)}, n_{e}^{(j)}, u_{i}^{(j)})$ for $j\geq2$. We also note that \eqref{kdv} is different from the classical KdV equation due to the presence of the parameter $H$. When $H\neq2$, it can be transformed into the classical KdV equation, while when $H=2$, it reduces to the inviscid Burger's equation, which is drastically different from the KdV equation. For derivation of KdV from water waves, see \cite{KdV}.

Much of the properties of the KdV equation follow from the interplay between advection and dispersion. One can see that the quantum effects can even invert the sign of dispersion for large $H$. However, this sign is immaterial since we can apply the transform $t\to-t,x\to x, n_i^{(1)}\to-n_i^{(1)}$. This implies that for $H>2$, the localized solutions (bright solitons) with $n_i^{(1)}>0$ of the original equation correspond also to localized solutions, but with inverted polarization ($n_i^{(1)}<0$, dark solitons) and propagating backward in time. But when $H=2$, the dispersive term vanishes, which eventually yields the formation of a shock in the Burger's equation. For details of the solitons, one may refer to \cite{HGG3}.

When $H\neq2$, we have the following existence theorem \cite{KPV91,KPV93}.
\begin{theorem}\label{thm2}
Let $H\neq2$ and $\tilde s_1\geq2$ be a sufficiently large integer. Then for any given initial data $n_{i0}^{(1)}\in H^{\tilde s_1}(\mathbb R)$, there exists $\tau_*>0$ such that the initial value problem \eqref{kdv} has a unique solution
$n_{i}^{(1)}\in L^{\infty}\big(-\tau_*,\tau_*; H^{\tilde s_1}(\Bbb R)\big)$. Furthermore, by using the conservation laws of the KdV equation, we can extend the solution to any time interval $[-\tau,\tau]$.
\end{theorem}


There is also an existence theorem for $H=2$, see \cite{SD,Majda84}.
\begin{theorem}\label{thm2''}
Let $H=2$ and $\tilde s_2\geq2$ be a sufficiently large integer. Then for any given initial data $n_{i0}^{(1)}\in H^{\tilde s_2}(\Bbb R)$, there exists $\tilde{\tau}_*>0$ such that the initial value problem \eqref{kdv} with $H=2$ has a unique solution $n_{i}^{(1)}\in L^{\infty}\big(0,\tilde{\tau_*}; H^{\tilde s_2}(\Bbb R)\big)$ with initial data $n_{i0}^{(1)}$.
\end{theorem}

To find out the equation satisfied by $(n_{i}^{(2)},n_{e}^{(2)},u_{i}^{(2)})$ assuming $(n_{i}^{(1)},n_{e}^{(1)},u_{i}^{(1)})$ is known form \eqref{equ3} and \eqref{kdv}, we express $(n_{i}^{(2)},n_{e}^{(2)},u_{i}^{(2)})$ in terms of $(n_{i}^{(1)},n_{e}^{(1)},u_{i}^{(1)})$ from \eqref{e2},
\begin{subequations}\label{equ3'}
\begin{numcases}{(\mathcal L_2):\ \ \ }
n_{e}^{(2)}=n_{i}^{(2)}+\partial_{x}^{2}n_{e}^{(1)},\\
u_{i}^{(2)}=n_{i}^{(2)}+g^{(1)}, g^{(1)}=-\int_{0}^{x}\mathfrak{g}^{(1)}(t,\xi)d\xi,\nonumber\\ \ \ \ \ \ \ \ \ \ \ \ \ \ \ \ \ \ \ \ \ \ \ \ \ \mathfrak{g}^{(1)}=-\partial_{t}u_{i}^{(1)}+\partial_{\xi}(n_{i}^{(1)}u_{i}^{(1)}),
\end{numcases}
\end{subequations}
which makes \eqref{e2} valid. Thus only $n_{i}^{(2)}$ needs to be determined.

At the order $O(\epsilon^3)$, we obtain
\begin{subequations}\label{e3}
\begin{numcases}{(\mathcal S_2)}
\partial_{t}n_{i}^{(2)}-\partial_{x}n_{i}^{(3)}+\partial_{x}u_{i}^{(3)}+\partial_{x}(n_{i}^{(1)}u_{i}^{(2)}+n_{i}^{(2)}u_{i}^{(1)})=0,\label{e3-1}\\
\partial_{t}u_{i}^{(2)}-\partial_{x}u_{i}^{(3)}+\partial_{x}(u_{i}^{(1)}u_{i}^{(2)})
=-\partial_{x}n_{e}^{(3)}-\partial_{x}(n_{e}^{(1)}n_{e}^{(2)})\nonumber\\ \ \ \ \ \ \ \ \ \
+\frac{H^{2}}{4}(\partial_{x}^{3}n_{e}^{(2)}-2\partial_{x}n_{e}^{(1)}\partial_{x}^{2}n_{e}^{(1)}-n_{e}^{(1)}\partial_{x}^{3}n_{e}^{(1)}),\label{e3-2}\\
\partial_{x}^{2}n_{e}^{(2)}+n_{e}^{(1)}\partial_{x}^{2}n_{e}^{(1)}+(\partial_{x}n_{e}^{(1)})^{2}-\frac{H^{2}}{4}\partial_{x}^{4}n_{e}^{(1)}
=n_{e}^{(3)}-n_{i}^{(3)}.\label{e3-3}
\end{numcases}
\end{subequations}
Differentiating \eqref{e3-3} with respect to $x$, and then adding the resultant and \eqref{e3-2} to \eqref{e3-1} together, we deduce that $n_{i}^{(2)}$ satisfies the linearized inhomogeneous quantum KdV equation
\begin{equation}\label{kdv1}
\partial_{t}n_{i}^{(2)}+2\partial_{x}(n_{i}^{(1)}n_{i}^{(2)})+\frac{1}{2}(1-\frac{H^{2}}{4})\partial_{x}^{3}n_{i}^{(2)}=G^{(1)},
\end{equation}
where we have used \eqref{equ3'} and $G^{(1)}$ depending on only $n_{i}^{(1)}$. Again, the system \eqref{kdv1} and \eqref{equ3'} for $(n_{i}^{(2)},n_{e}^{(2)},u_{i}^{(2)})$ are self contained, which do not depend on $(n_{i}^{(j)},n_{e}^{(j)},u_{i}^{(j)})$ for $j\geq3$.\\

Inductively, at the order $O(\epsilon^{k})$, we obtain a system $(\mathcal S_{k-1})$ for $(n_{i}^{(k-1)},n_{e}^{(k-1)},u_{i}^{(k-1)})$, from which we obtain
\begin{equation}\label{relation}
{(\mathcal L_k):}\ \ \ \
n_{e}^{(k)}=n_{i}^{(k)}+h^{(k-1)},\ \ \ u_{i}^{(k)}=n_{i}^{k}+g^{(k-1)},
\end{equation}
where $h^{(k-1)}$ and $g^{(k-1)}$ depend only on $(n_{e}^{(j)})$ for $1\leq j\leq k-1$. Thus we need only to determine $n_{i}^{(k)}$. At the order $O(\epsilon^{k+1})$, we obtain a system $(\mathcal{S}_{k})$ for $(n_{i}^{(k)},n_{e}^{(k)},u_{i}^{(k)})$, from which we obtain the linearized inhomogeneous KdV equation for $n_{i}^{(k)}$,
\begin{equation}\label{linearized}
\begin{split}
\partial_{t}n_{i}^{(k)}+2\partial_{x}(n_{i}^{(1)}n_{i}^{(k)})+\frac{1}{2}(1-\frac{H^{2}}{4})\partial_{x}^{3}n_{i}^{(k)}=G^{(k-1)}, \ \ \ k\geq3,
\end{split}
\end{equation}
where $G^{(k-1)}$ depends only on $n_{i}^{(1)},n_{i}^{(2)},\cdots,n_{i}^{(k-1)}$, which are ``known" from the first $(k-1)^{th}$ steps. Also, it is important to note that the system \eqref{relation} and \eqref{linearized} for $n_{i}^{(k)},n_{e}^{(k)},u_{i}^{(k)}$ are self contained, which do not depend on $(n_{i}^{(j)},n_{e}^{(j)},u_{i}^{(j)})$ for $j\geq k+1$.

For the solvability of $(n_{i}^{(k)},n_{e}^{(k)},u_{i}^{(k)})$ for $k\geq2$, we state the following
\begin{theorem}\label{thm3}
Let $k\geq 2$, $\tilde s_k\leq \tilde s_1-3(k-1)$ be sufficiently large integers and $n_{i0}^{(k)}\in H^{\tilde s_k}(\Bbb R)$. Then when $H\neq 2$, the initial value problem \eqref{linearized} with initial data $n_{i0}^{(k)}$ has a unique solution $n_{i}^{(k)}\in L^{\infty}(-\tau,\tau; H^{\tilde s_k}(\Bbb R))$ for any $\tau>0$. When $H=2$, the initial value problem \eqref{linearized} has a unique solution $n_{i}^{(k)}\in L^{\infty}(0,\tilde\tau_*; H^{\tilde s_k}(\Bbb R))$, where $\tilde\tau_*$ is given in Theorem \ref{thm2''}.
\end{theorem}

The proof of Theorem \ref{thm3} is standard. Based on this theorem, we will assume that these solutions $(n_{i}^{(k)},n_{e}^{(k)},u_{i}^{(k)})$ for $1\leq k\leq 4$ are as smooth as we want. The optimality of $\tilde s_k$ will not be addressed in this paper.

\subsection{Main result}\label{sect-rem} To show that $n_{i}^{(1)}$ converges to a solution of the KdV equation as $\epsilon\rightarrow0$, we must make the above procedure rigorous. Let $(n_{e},n_{i},u_{i})$ be the solution of the scaled system (1.3) of the following expansion

\begin{equation}\label{cut}
\begin{cases}
n_{i}&=1+\epsilon n_{i}^{(1)}+\epsilon^{2}n_{i}^{(2)}+\epsilon^{3}n_{i}^{(3)}+\epsilon^{4}n_{i}^{(4)} +\epsilon^{3}N_i,\\
n_{e}&=1+\epsilon n_{e}^{(1)}+\epsilon^{2}n_{e}^{(2)}+\epsilon^{3}n_{e}^{(3)} +\epsilon^{4}n_{e}^{(4)}+\epsilon^{3}N_e,\\
u_{i}&=\epsilon u_{i}^{(1)}+\epsilon^{2}u_{i}^{(2)}+\epsilon^{3}u_{i}^{(3)} +\epsilon^{4}u_{i}^{(4)}+\epsilon^{3}U,
\end{cases}
\end{equation}
where $(n_{i}^{(1)}$,$n_{e}^{(1)}$,$u_{i}^{(1)})$ satisfies \eqref{e1} and \eqref{equ3}, ($n_{i}^{(k)}$,$n_{e}^{(k)}$,$u_{i}^{(k)}$) satisfies \eqref{relation} and \eqref{linearized} for $2\leq k\leq4$, and $(N_{i},N_{e},U)$ is the remainder. To simplify the notation slightly, we set
\begin{equation}
\begin{cases}
\tilde{n_i}= n_i^{(1)}+\epsilon n_i^{(2)}+\epsilon^{2}n_i^{(3)}+\epsilon^{3}n_i^{(4)},\\
\tilde{n_e}= n_e^{(1)}+\epsilon n_e^{(2)}+\epsilon^{2}n_e^{(3)}+\epsilon^{3}n_e^{(4)},\\
\tilde{u_i}= u_i^{(1)}+\epsilon u_i^{(2)}+\epsilon^{2}u_i^{(3)}+\epsilon^{3}u_i^{(4)}.
\end{cases}
\end{equation}
After careful computations, we obtain the following remainder system for $(N_{i},N_{e},U)$,
\begin{subequations}\label{rem}
\begin{numcases}{}
\partial_{t}N_{i}-\frac{1-u_{i}}{\epsilon}\partial_{x}N_{i}+\frac{n_{i}}{\epsilon}\partial_{x}U
+\partial_{x}\tilde{n_{i}}U+\partial_{x}\tilde{u_{i}}N_{i}+\epsilon \mathcal{R}_{1}=0,\label{rem-1}\\
\partial_{t}U-\frac{1-u_{i}}{\epsilon}\partial_{x}U+\partial_{x}\tilde{u_{i}}U=
-\frac{n_{e}}{\epsilon}\partial_{x}N_{e}-\partial_{x}\tilde{n_{e}}N_{e}\nonumber\\ \ \ \ \ \ \
+\frac{H^{2}}{4}\Big[\frac{3\epsilon^{2}(\partial_{x}\tilde{n_{e}})^{2}}{n_{e}^{3}}\partial_{x}N_{e}
-\frac{2\epsilon\partial_{x}^{2}\tilde{n_{e}}}{n_{e}^{2}}\partial_{x}N_{e}
+\frac{3\epsilon^{4}\partial_{x}\tilde{n_{e}}}{n_{e}^{3}}(\partial_{x}N_{e})^{2}\nonumber\\ \ \ \ \ \ \
-\frac{2\epsilon\partial_{x}\tilde{n_{e}}}{n_{e}^{2}}\partial_{x}^{2}N_{e}
+\frac{\epsilon^{6}}{n_{e}^{3}}(\partial_{x}N_{e})^{3}
-\frac{2\epsilon^{3}}{n_{e}^{2}}\partial_{x}N_{e}\partial_{x}^{2}N_{e}
+\frac{1}{n_{e}}\partial_{x}^{3}N_{e}\nonumber\\ \ \ \ \ \ \
+\frac{\epsilon}{n_{e}^{3}}(\mathcal{R}_{2}^{3}+\mathcal{R}_{2}^{4})\Big]+\epsilon \mathcal{R}_{2}^{1,2},\label{rem-2}\\
2\partial_{x}\tilde{n_{e}}\partial_{x}N_{e}+\epsilon^{2}(\partial_{x}N_{e})^{2}+\frac{n_{e}}{\epsilon}\partial_{x}^{2}N_{e}
+\partial_{x}^{2}\tilde{n_{e}}N_{e}+\mathcal{R}_{3}^{1}\nonumber\\ \ \ \ \ \ \
-\frac{H^{2}}{4}\Big[-\frac{12\epsilon^{3}(\partial_{x}\tilde{n_{e}})^{3}}{n_{e}^{4}}\partial_{x}N_{e}
+\frac{14\epsilon^{2}\partial_{x}\tilde{n_{e}}\partial_{x}^{2}\tilde{n_{e}}}{n_{e}^{3}}\partial_{x}N_{e}
-\frac{3\epsilon\partial_{x}^{3}\tilde{n_{e}}}{n_{e}^{2}}\partial_{x}N_{e}\nonumber\\ \ \ \ \ \ \
-\frac{18\epsilon^{5}(\partial_{x}\tilde{n_{e}})^{2}}{n_{e}^{4}}(\partial_{x}N_{e})^{2}
+\frac{7\epsilon^{2}(\partial_{x}\tilde{n_{e}})^{2}}{n_{e}^{3}}\partial_{x}^{2}N_{e}
+\frac{7\epsilon^{4}\partial_{x}^{2}\tilde{n_{e}}}{n_{e}^{3}}(\partial_{x}N_{e})^{2}\nonumber\\ \ \ \ \ \ \
-\frac{4\epsilon\partial_{x}^{2}\tilde{n_{e}}}{n_{e}^{2}}\partial_{x}^{2}N_{e}
-\frac{12\epsilon^{7}\partial_{x}\tilde{n_{e}}}{n_{e}^{4}}(\partial_{x}N_{e})^{3}
+\frac{14\epsilon^{4}\partial_{x}\tilde{n_{e}}}{n_{e}^{3}}\partial_{x}N_{e}\partial_{x}^{2}N_{e}\nonumber\\ \ \ \ \ \ \
-\frac{3\epsilon\partial_{x}\tilde{n_{e}}}{n_{e}^{2}}\partial_{x}^{3}N_{e}
-\frac{3\epsilon^{9}}{n_{e}^{4}}(\partial_{x}N_{e})^{4}
+\frac{7\epsilon^{6}}{n_{e}^{3}}(\partial_{x}N_{e})^{2}\partial_{x}^{2}N_{e}
-\frac{2\epsilon^{3}}{n_{e}^{2}}(\partial_{x}^{2}N_{e})^{2}\nonumber\\ \ \ \ \ \ \
-\frac{3\epsilon^{3}}{n_{e}^{2}}\partial_{x}N_{e}\partial_{x}^{3}N_{e}
+\frac{\partial_{x}^{4}N_{e}}{n_{e}}
+\frac{\mathcal R_{3}^{2}+\mathcal{R}_{3}^{3}}{n_{e}^{4}}\Big]
=\frac{N_{e}-N_{i}}{\epsilon^{2}},\label{rem-3}
\end{numcases}
\end{subequations}
where
\begin{eqnarray}
\left\{
\begin{array}{llll}
\mathcal{R}_{1}=\partial_{x}(n_{i}^{(1)}u_{i}^{(4)})+\partial_{x}[n_{i}^{(2)}(u_{i}^{(3)}
+\epsilon u_{i}^{(4)})]+\partial_{x}[n_{i}^{(3)}(u_{i}^{(2)}\\ \ \ \ \ \ \ \ \
+\epsilon u_{i}^{(3)}+\epsilon^{2} u_{i}^{(4)})]+\partial_{x}(u_{i}^{(4)}\tilde{u_{i}})-\epsilon\partial_{t}n_{i}^{(4)},\\
\mathcal{R}_{2}^{1}=-\partial_{t}u_{i}^{(4)}+u_{i}^{(1)}\partial_{x}u_{i}^{(4)}+u_{i}^{(2)}(\partial_{x}u_{i}^{(3)}
+\epsilon\partial_{x}u_{i}^{(4)})\\ \ \ \ \ \ \ \ \ \ \
+u_{i}^{(3)}(\partial_{x}u_{i}^{(2)}+\epsilon\partial_{x}u_{i}^{(3)}
+\epsilon^{2}u_{i}^{(4)})+u_{i}^{(4)}\partial_{x}\tilde{u_{i}},\\
\mathcal{R}_{2}^{2}=n_{e}^{(1)}\partial_{x}n_{e}^{(4)}+n_{e}^{(2)}(\partial_{x}n_{e}^{(3)}+\epsilon\partial_{x}n_{e}^{(4)})\\ \ \ \ \ \ \ \ \ \ \
+n_{e}^{(3)}(\partial_{x}n_{e}^{(2)}+\epsilon\partial_{x}n_{e}^{(3)}+\epsilon^{2} n_{e}^{(4)})+n_{e}^{(4)}\tilde{n_{e}},\\
\mathcal{R}_{2}^{3}=\mathcal{R}_{2}^{3}(n_{e}^{(1)},n_{e}^{(2)},n_{e}^{(3)},n_{e}^{(4)}),\\
\mathcal{R}_{2}^{4}=\mathcal{R}_{2}^{4}(\epsilon N_{e}),\\
\mathcal{R}_{3}^{1}=\mathcal{R}_{3}^{1}(n_{e}^{(1)},n_{e}^{(2)},n_{e}^{(3)},n_{e}^{(4)}),\\
\mathcal{R}_{3}^{2}=\mathcal{R}_{3}^{2}(n_{e}^{(1)},n_{e}^{(2)},n_{e}^{(3)},n_{e}^{(4)}),\\
\mathcal{R}_{3}^{3}=\mathcal{R}_{3}^{3}(N_{e}),
\end{array}
\right.
\end{eqnarray}
where $\mathcal{R}_{2}^{4}$ and $\mathcal{R}_{3}^{3}$ are smooth functions of $N_{e}$, and do not involve any derivatives of $N_{e}$. The mathematical key difficulty is to derive uniform in $\epsilon$ estimates for the remainder $(N_{e},N_{i},U)$.

For convenient usage, we give the following
\begin{lemma}\label{L8}
For $\alpha=0,1,\cdots$ integers, there exists some constant $C=C(\|n_{e}^{(i)}\|_{H^{\tilde s_i}})$  such that
\begin{equation}\label{equ28}
\begin{split}
\|\mathcal R_1,\mathcal R_2^{1,2,3},\mathcal R_3^{1,2}\|_{H^{\alpha}}\leq C ,\ \ \ \alpha=0,1,\cdots,
\end{split}
\end{equation}
\begin{equation}\label{equ29}
\begin{split}
&\|\mathcal R_2^{4}\|_{H^{\alpha}}\leq C \epsilon\|N_{e}\|_{H^{\alpha}}, \ \ \ \alpha=0,1,\cdots,\\
&\|\mathcal R_3^{3}\|_{H^{\alpha}}\leq C \|N_{e}\|_{H^{\alpha}}, \ \ \ \ \alpha=0,1,\cdots,
\end{split}
\end{equation}
and
\begin{equation}\label{equ34}
\begin{split}
&\|\partial_tR_2^{4}\|_{H^{\alpha}}\leq C \epsilon\|\partial_tN_{e}\|_{H^{\alpha}}, \ \ \ \alpha=0,1,\cdots,\\
&\|\partial_t R_3^{3}\|_{H^{\alpha}}\leq C \|\partial_tN_{e}\|_{H^{\alpha}}, \ \ \ \alpha=0,1,\cdots.
\end{split}
\end{equation}
\end{lemma}
Recalling the fact that $H^1$ is an algebra, the estimate for Lemma \ref{L8} is straightforward. The details are hence omitted.
Our main result of this paper is the following\\
\begin{theorem}\label{thm1}
Let $\tilde{s_{i}}$ be sufficiently large and $(n_{i}^{(1)},n_{e}^{(1)},u_{i}^{(1)})\in H^{\tilde{s_{1}}}$ be a solution constructed in Theorem \ref{thm2} for the quantum KdV equation with initial data $(n_{i0},n_{e0},u_{i0})\in H^{\tilde{s_{1}}}$ satisfying \eqref{equ3}. Let $(n_{i}^{(j)},n_{e}^{(j)},u_{i}^{(j)})\in H^{\tilde{s_{j}}}$ (i=2,3,4) be solutions of \eqref{relation} and \eqref{linearized} constructed in Theorem \ref{thm3} with initial data
$(n_{i0}^{j},n_{e0}^{j},u_{i0}^{j})\in H^{\tilde{s_{j}}}$ satisfying \eqref{relation}. Let
$(N_{i0},N_{e0},U_{0})$ satisfy \eqref{rem} and assume
\begin{equation*}
\begin{split}
n_{i0}&=1+\epsilon n_{i0}^{(1)}+\epsilon^{2}n_{i0}^{(2)}+\epsilon^{3}n_{i0}^{(3)}+\epsilon^{4}n_{i0}^{(4)}+\epsilon^{3}N_{i0},\\
n_{e0}&=1+\epsilon n_{e0}^{(1)}+\epsilon^{2}n_{e0}^{(2)}+\epsilon^{3}n_{e0}^{(3)}+\epsilon^{4}n_{e0}^{(4)}+\epsilon^{3}N_{e0},\\
u_{i0}&=\epsilon u_{i0}^{(1)}+\epsilon^{2}u_{i0}^{(2)}+\epsilon^{3}u_{i0}^{(3)}+\epsilon^{4}u_{i0}^{(4)}+\epsilon^{3}U_{0}.\\
\end{split}
\end{equation*}
Then for any $\tau>0$, there exists $\epsilon_{0}>0$ such that if $0<\epsilon<\epsilon_{0}$, the solution of the EP system \eqref{equ2} with initial data $(n_{i0},n_{e0},u_{i0})$ can be expressed as
\begin{equation*}
\begin{split}
n_{i}&=1+\epsilon n_{i}^{(1)}+\epsilon^{2}n_{i}^{(2)}+\epsilon^{3}n_{i}^{(3)}+\epsilon^{4}n_{i}^{(4)}+\epsilon^{3}N_{i},\\
n_{e}&=1+\epsilon n_{e}^{(1)}+\epsilon^{2}n_{e}^{(2)}+\epsilon^{3}n_{e}^{(3)}+\epsilon^{4}n_{e}^{(4)}+\epsilon^{3}N_{e},\\
u_{i}&=\epsilon u_{i}^{(1)}+\epsilon^{2}u_{i}^{(2)}+\epsilon^{3}u_{i}^{(3)}+\epsilon^{4}u_{i}^{(4)}+\epsilon^{3}U,
\end{split}
\end{equation*}
such that for all $0<\epsilon<\epsilon_{0}$,
\begin{equation}\label{ineq}
\begin{split}
\sup_{[0,\tau]}\Big\{\|(N_{i},N_{e},U)\|_{H^{2}}^{2} &+\epsilon\|(\partial_{x}^{3}N_{e}, \partial_{x}^{3}U)\|_{L^{2}}^{2} +\epsilon^{2}\|\partial_{x}^{4}N_{e}\|_{L^{2}}^{2} +\epsilon^{3}\|\partial_{x}^{5}N_{e}\|_{L^{2}}^{2}\\
+\epsilon^{4}\|\partial_{x}^{6}N_{e}\|_{L^{2}}^{2} \Big\}\leq &C_{\tau}\Big(1+\|(N_{i0},N_{e0},U_{0})\|_{H^{2}}^{2} +\epsilon\|(\partial_{x}^{3}N_{e0},\partial_{x}^{3}U_{0})\|_{L^{2}}^{2}\\
&+\epsilon^{2}\|\partial_{x}^{4}N_{e0}\|_{L^{2}}^{2} +\epsilon^{3}\|\partial_{x}^{5}N_{e0}\|_{L^{2}}^{2} +\epsilon^{4}\|\partial_{x}^{6}N_{e0}\|_{L^{2}}^{2}\Big).
\end{split}
\end{equation}
\end{theorem}
From \eqref{ineq}, we see that the $H^2$-norm of the remainder $(N_i,N_e,U)$ is bounded uniformly in $\epsilon$. Note also the Gardner-Morikawa transform \eqref{equ39}, we see that
\begin{equation}
\sup_{[0,\epsilon^{-3/2}\tau]}
\left\|\left(
\begin{array}{c}
(n_{i}-1)/\epsilon\\ (n_{e}-1)/\epsilon\\ u_{i}/\epsilon
\end{array}
\right)-KdV\right\|_{H^2}\leq C\epsilon,
\end{equation}
for some $C>0$ independent of $\epsilon>0$. Here `KdV' is the equation satisfied by the first approximation $(n_i^{(1)},n_e^{(1)},u_i^{(1)})$.

The basic plan is to first estimate some uniform bound for $(N_{e},U)$ and then recover the estimate for $N_{i}$ from the estimate of $N_{e}$ by the equation \eqref{rem}. We want to apply the Gronwall lemma to complete the proof. To state clearly, we first introduce
\begin{equation}\label{|||}
\begin{split}
|\!|\!|(N_{e},U)|\!|\!|_{\epsilon}^{2}=&\|(N_{e},U)\|_{H^{2}}^{2} +\epsilon\|(\partial_{x}^{3}N_{e},\partial_{x}^{3}U)\|_{L^{2}}^{2} +\epsilon^{2}\|\partial_{x}^{4}N_{e}\|_{L^{2}}^{2} +\epsilon^{3}\|\partial_{x}^{5}N_{e}\|_{L^{2}}^{2} +\epsilon^{4}\|\partial_{x}^{6}N_{e}\|_{L^{2}}^{2}.\
\end{split}
\end{equation}
As we will see, the zeroth order, the first order to the second order estimates for $(N_{e},U)$ and the third order estimates for $\epsilon(N_{e},U)$ all can be controlled in terms of $|\!|\!|(N_{e},U)|\!|\!|_{\epsilon}^{2}$.

For convenience, we introduce the following lemma
\begin{lemma}[Commutator Estimate]\label{L9}
Let $m\geq1$ be an integer, and then the commutator which is defined by the following
\begin{equation}
\begin{split}\label{w}
[\nabla^{m},f]g:=\nabla^{m}(fg)-f\nabla^{m}g,
\end{split}
\end{equation}
can be bounded by
\begin{equation}
\begin{split}\label{w11}
\|[\nabla^{m},f]g\|_{L^{p}}\leq \|\nabla f\|_{L^{p_{1}}}\|\nabla^{m-1}g\|_{L^{p_{2}}}+\|\nabla^{m} f\|_{L^{p_{3}}}\|g\|_{L^{p_{4}}},
\end{split}
\end{equation}
where $p,p_{2},p_{3}\in(1,\infty)$  and
\begin{equation*}
\begin{split}
\frac{1}{p}=\frac{1}{p_{1}}+\frac{1}{p_{2}}=\frac{1}{p_{3}}+\frac{1}{p_{_{_{4}}}}.
\end{split}
\end{equation*}
\end{lemma}
\begin{proof}
The proof can be found in \cite{CM86,Kato90,KP88,KPV91,KPV93}, for example.
\end{proof}

\section{Uniform energy estimates}\label{sect-energy}
In this section, we give the energy estimates uniformly in $\epsilon$ for the remainder $(N_{e},N_{i},U)$, which requires a combination of energy method and analysis of the remainder equation \eqref{rem}. To simplify the proof slightly, we assume that \eqref{rem} has smooth solutions in $[0,\tau_{\epsilon}]$ for                    $\tau_{\epsilon}>0$ depending on $\epsilon$.
Let $\tilde{C}$ be a constant independent of $\epsilon$, which will be determined later, much larger than the bound $|\!|\!|(N_{e},U)(0)|\!|\!|_{\epsilon}^{2}$ of the initial data. It is classical that there exists $\tau_{\epsilon}>0$ such that on $[0,\tau_{\epsilon}]$,
\begin{equation}
\begin{split}\label{priori}
\|N_{i}\|_{H^{2}}^{2},\ \ |\!|\!|(N_{e},U)|\!|\!|_{\epsilon}^{2}\leq\tilde{C}.
\end{split}
\end{equation}
As a direct corollary, there exists some $\epsilon_{1}>0$ such that $n_{e} \ and  \ n_{i}$ are bounded from above and below, say $\frac{1}{2}<n_{i},n_{e}<\frac{3}{2}$ and $u_{i}$ is bounded by $|u_{i}|<\frac{1}{2}$ when $\epsilon<\epsilon_{1}$. Since $\mathcal{R}_{2}^{4},\mathcal{R}_{3}^{3}$ are smooth functions of $N_{e}$, there exists some constant $C_{1}=C_{1}(\epsilon\tilde{C})$ for any $\alpha,\beta\geq0$ such that
\begin{equation*}
\begin{split}
\left|\partial_{n_{e}^{(j)}}^{\alpha} \partial_{N_{e}}^{\beta}(\mathcal{R}_{2}^{4}, \mathcal{R}_{3}^{3})\right|\leq C_{1}=C_{1}(\epsilon\tilde{C}),
\end{split}
\end{equation*}
where $C_{1}(\cdot)$ can be chosen to be nondecreasing in its argument. We will show that for any given $\tau>0$ there is some $\epsilon_{0}>0$, such that the existence time $\tau_{\epsilon}>\tau$ for any $0<\epsilon<\epsilon_{0}$.

The purpose of this section is to prove Proposition \ref{P1-P2} and \ref{P3}. Since the proof of Proposition \ref{P1-P2} will be almost the same to that of Proposition \ref{P3}, we will omit the proof of Proposition \ref{P1-P2}. In Subsection \ref{2.1}, we first show three lemmas that will be frequently used later. In Subsection \ref{2.2} and Subsection \ref{2.3}, we present and prove the two main propositions, while estimates of some crucial terms are postponed to Subsection \ref{2.4} and Subsection \ref{2.5}.

\subsection{Basic estimates}\label{2.1}
We first prove the following Lemma \ref{L1}-\ref{L3}, in which we bound $N_{i}$ and $\partial_{t}N_{e}$ in terms of $N_{e}$.
\begin{lemma}\label{L1}
Let $(N_{i},N_{e},U)$ be a solution to \eqref{rem} and $\alpha\geq0$ be an integer. There exist some constants $0<\epsilon_{1}<1$ and $C_{1}=C_{1}(\epsilon\tilde{C})$ such that for every $0<\epsilon<\epsilon_{1}$,
\begin{equation}\label{eqaL1}
\begin{split}
C_{1}^{-1}\|\partial_{x}^{\alpha}N_{i}\|^{2} \leq &\|\partial_{x}^{\alpha}N_{e}\|^{2}
+\epsilon\|\partial_{x}^{\alpha+1}N_{e}\|^{2}+\epsilon^{2}\|\partial_{x}^{\alpha+2}N_{e}\|^{2}\\
&+\epsilon^{3}\|\partial_{x}^{\alpha+3}N_{e}\|^{2}+\epsilon^{4}\|\partial_{x}^{\alpha+4}N_{e}\|^{2}
\leq C_{1}\|\partial_{x}^{\alpha}N_{i}\|^{2}.
\end{split}
\end{equation}
\end{lemma}
\begin{proof}
When $\alpha=0$, taking inner product of (1.17c) with $N_{e}$ and integration by parts, we have
\begin{equation}
\begin{split}\label{3.1}
\|N_{e}&\|^{2}+\epsilon\|\partial_{x}N_{e}\|^{2}+\frac{\epsilon^{2}H^{2}}{4}\int\frac{1}{n_{e}}(\partial_{x}^{2}N_{e})^{2}\\
=&\epsilon^{2}\int(\tilde{n_{e}}+\epsilon^{2}N_{e})(\partial_{x}^{2}N_{e})N_{e}
-\frac{\epsilon^{2}H^{2}}{2}\int(\partial_{x}\frac{1}{n_{e}})\partial_{x}^{2}N_{e}\partial_{x}N_{e}
-\frac{\epsilon^{2}H^{2}}{4}\int(\partial_{x}^{2}\frac{1}{n_{e}})\partial_{x}^{2}N_{e}N_{e}\\
&+2\epsilon^{2}\int\partial_{x}\tilde{n_{e}}(\partial_{x}N_{e})N_{e}
+\epsilon^{4}\int(\partial_{x}N_{e})^{2}N_{e}
+\epsilon^{2}\int\partial_{x}^{2}\tilde{n_{e}}N_{e}^{2}
+\epsilon^{2}\int \mathcal{R}_{3}^{1}N_{e}\\
&-\frac{H^{2}}{4}\Big\{-12\epsilon^{5}\int\frac{(\partial_{x}\tilde{n_{e}})^{3}}{n_{e}^{4}}(\partial_{x}N_{e})N_{e}
+14\epsilon^{4}\int\frac{\partial_{x}\tilde{n_{e}}\partial_{x}^{2}\tilde{n_{e}}}{n_{e}^{3}}(\partial_{x}N_{e})N_{e}\\
&-3\epsilon^{3}\int\frac{\partial_{x}^{3}\tilde{n_{e}}}{n_{e}^{2}}(\partial_{x}N_{e})N_{e}
-18\epsilon^{7}\int\frac{(\partial_{x}\tilde{n_{e}})^{2}}{n_{e}^{4}}(\partial_{x}N_{e})^{2}N_{e}
+7\epsilon^{4}\int\frac{(\partial_{x}\tilde{n_{e}})^{2}}{n_{e}^{3}}(\partial_{x}^{2}N_{e})N_{e}\\
&+7\epsilon^{6}\int\frac{\partial_{x}^{2}\tilde{n_{e}}}{n_{e}^{3}}(\partial_{x}N_{e})^{2}N_{e}
-4\epsilon^{3}\int\frac{\partial_{x}^{2}\tilde{n_{e}}}{n_{e}^{2}}(\partial_{x}^{2}N_{e})N_{e}
-12\epsilon^{9}\int\frac{\partial_{x}\tilde{n_{e}}}{n_{e}^{4}}(\partial_{x}N_{e})^{3}N_{e}\\
&+14\epsilon^{6}\int\frac{\partial_{x}\tilde{n_{e}}}{n_{e}^{3}}\partial_{x}N_{e}(\partial_{x}^{2}N_{e})N_{e}
-3\epsilon^{3}\int\frac{\partial_{x}\tilde{n_{e}}}{n_{e}^{2}}(\partial_{x}^{3}N_{e})N_{e}
-3\epsilon^{11}\int\frac{1}{n_{e}^{4}}(\partial_{x}N_{e})^{4}N_{e}\\
&+7\epsilon^{8}\int\frac{1}{n_{e}^{3}}(\partial_{x}N_{e})^{2}(\partial_{x}^{2}N_{e})N_{e}
-2\epsilon^{5}\int\frac{1}{n_{e}^{2}}(\partial_{x}^{2}N_{e})^{2}N_{e}
-3\epsilon^{5}\int\frac{1}{n_{e}^{2}}\partial_{x}N_{e}(\partial_{x}^{3}N_{e})N_{e}\\
&+\epsilon^{2}\int\frac{\mathcal{R}_{3}^{2}+\mathcal{R}_{3}^{3}}{n_{e}^{4}}N_{e}\Big\}
+\int{N_{e}N_{i}}\\
=&:\sum_{i=1}^{23}A_{i} \ .
\end{split}
\end{equation}
Since $\frac{1}{2}<n_{e}<\frac{3}{2}$ and $H$ is a fixed constant, there exists a fixed constant $C$ such that
\begin{equation*}
\begin{split}
\frac{\epsilon^{2}H^{2}}{4}\int\frac{1}{n_{e}}(\partial_{x}^{2}N_{e})^{2}
\geq C\epsilon^{2}\|\partial_{x}^{2}N_{e}\|^{2}.
\end{split}
\end{equation*}
Thus the LHS of \eqref{3.1} is equal or greater than $C(\|N_{e}\|^{2}+\epsilon\|\partial_{x}N_{e}\|^{2} +\epsilon^{2}\|\partial_{x}^{2}N_{e}\|^{2})$. Next, we estimate the RHS of \eqref{3.1}. For $A_{1}$, since $\tilde{n_{e}}$ is known and bounded in $L^{\infty}$, there exists some constant $C$ such that
\begin{equation*}
\begin{split}
A_{1}&=\epsilon^{2}\int(\tilde{n_{e}}+\epsilon^{2}N_{e})(\partial_{x}^{2}N_{e})N_{e}\\
&\leq C(1+\epsilon^{2}\|N_{e}\|_{L^{\infty}})(\epsilon\|N_{e}\|^{2}+\epsilon^{3}\|\partial_{x}^{2}N_{e}\|^{2})\\
&\leq C(1+\epsilon^{2}\|N_{e}\|_{H^{1}})(\epsilon\|N_{e}\|^{2}+\epsilon^{3}\|\partial_{x}^{2}N_{e}\|^{2})\\
&\leq C(1+\epsilon^{2}\tilde{C})(\epsilon\|N_{e}\|^{2}+\epsilon^{3}\|\partial_{x}^{2}N_{e}\|^{2})\\
&\leq C(\epsilon\|N_{e}\|^{2}+\epsilon^{3}\|\partial_{x}^{2}N_{e}\|^{2}),
\end{split}
\end{equation*}
where we have used H\"older's inequality, Sobolev embedding $H^1\hookrightarrow L^{\infty}$, the priori assumption \eqref{priori} and Cauchy inequality.

Note that
\begin{equation}
\begin{split}\label{one order}
\left|\partial_{x}\left(\frac{1}{n_{e}}\right)\right|
\leq C\left(\epsilon|\partial_{x}\tilde{n_{e}}| +\epsilon^{3}|\partial_{x}N_{e}|\right),
\end{split}
\end{equation}
and
\begin{equation}
\begin{split}\label{two orders}
\left|\partial_{x}^{2}\left(\frac{1}{n_{e}}\right)\right|
\leq C\big(\epsilon+\epsilon^{3}(|\partial_{x}N_{e}|+|\partial_{x}^{2}N_{e}|)+\epsilon^{6}|\partial_{x}N_{e}|^{2}\big).
\end{split}
\end{equation}
Since $\partial_{x}\tilde{n_{e}},\ \partial_{x}^{2}\tilde{n_{e}}$ are bounded in $L^{\infty}$, similar to $A_{1}$, we have
\begin{equation*}
\begin{split}
A_{2\sim16,18\sim20}\leq C_{1}(\epsilon\|N_{e}\|^{2}+\epsilon^{2}\|\partial_{x}N_{e}\|^{2}+\epsilon^{3}\|\partial_{x}^{2}N_{e}\|^{2}).
\end{split}
\end{equation*}
Now we estimate $A_{17}$.
By integration by parts, we obtain
\begin{equation*}
\begin{split}
A_{17}&=-\frac{3\epsilon^{3}H^{2}}{4}\int(\partial_{x}\frac{\partial_{x}\tilde{n_{e}}}{n_{e}^{2}})\partial_{x}^{2}N_{e}N_{e}
-\frac{3\epsilon^{3}H^{2}}{4}\int\frac{\partial_{x}\tilde{n_{e}}}{n_{e}^{2}}\partial_{x}^{2}N_{e}\partial_{x}N_{e}.
\end{split}
\end{equation*}
Similar to \eqref{one order}, we have
\begin{equation}
\begin{split}\label{g1}
\left|\partial_{x}\left(\frac{\partial_{x}\tilde{n_{e}}}{n_{e}^{2}}\right)\right|
\leq C(1+\epsilon^{3}|\partial_{x}N_{e}|).
\end{split}
\end{equation}
Similar to $A_{1}$, by applying H\"older's inequality, Sobolev embedding $H^1\hookrightarrow L^{\infty}$, the priori assumption \eqref{priori} and Cauchy inequality again, we have
\begin{equation*}
\begin{split}
A_{17}\leq C_{1}(\epsilon\|N_{e}\|^{2}+\epsilon^{2}\|\partial_{x}N_{e}\|^{2}+\epsilon^{3}\|\partial_{x}^{2}N_{e}\|^{2}).
\end{split}
\end{equation*}
The term $A_{21}$ can be similarly bounded by
\begin{equation*}
\begin{split}
A_{21}\leq C_{1}(\epsilon^{2}\|\partial_{x}N_{e}\|^{2}+\epsilon^{3}\|\partial_{x}^{2}N_{e}\|^{2}).
\end{split}
\end{equation*}
According to the form of $\mathcal{R}_{3}^{2} \ and  \ \mathcal{R}_{3}^{3}$ in (1.18), by applying Cauchy inequality, we have
\[
A_{22}\leq C_{1}\| N_{e}\|^{2}.
\]
By Young inequality, we have
\[
\int N_{e}N_{i}\leq\delta\|N_{e}\|^{2}+C_{\delta}\|N_{i}\|^{2},
\]
for arbitrary $\delta>0$.
Hence, there exists some $\epsilon_1>0$ such that for $0<\epsilon<\epsilon_1$,
\begin{equation}
\begin{split}\label{31}
\|N_{e}\|^{2}+\epsilon\|\partial_{x}N_{e}\|^{2}+\epsilon^{2}\|\partial_{x}^{2}N_{e}\|^{2}\leq C_{1}\|N_{i}\|^{2}.
\end{split}
\end{equation}
Taking inner product of (1.17c) with $\epsilon \partial_{x}^{2}N_{e}$ and $\epsilon^{2}\partial_{x}^{4}N_{e}$, applying H\"older inequality and integration by parts, we have similarly
\begin{equation}
\begin{split}\label{32}
\epsilon\|\partial_{x}N_{e}\|^{2}+\epsilon^{2}\|\partial_{x}^{2}N_{e}\|^{2}+\epsilon^{3}\|\partial_{x}^{3}N_{e}\|^{2}
\leq C_{1}\|N_{i}\|^{2},
\end{split}
\end{equation}
and
\begin{equation}
\begin{split}\label{33}
\epsilon^{2}\|\partial_{x}^{2}N_{e}\|^{2}+\epsilon^{3}\|\partial_{x}^{3}N_{e}\|^{2}+\epsilon^{4}\|\partial_{x}^{4}N_{e}\|^{2}
\leq C_{1}\|N_{i}\|^{2}.
\end{split}
\end{equation}
By the estimates \eqref{31},\ \eqref{32}\ and\ \eqref{33}, we obtain
\begin{equation}
\begin{split}\label{40}
\|N_{e}\|^{2}+\epsilon\|\partial_{x}N_{e}\|^{2}+\epsilon^{2}\|\partial_{x}^{2}N_{e}\|^{2}+\epsilon^{3}\|\partial_{x}^{3}N_{e}\|^{2}
+\epsilon^{4}\|\partial_{x}^{4}N_{e}\|^{2}
\leq C_{1}\|N_{i}\|^{2}.
\end{split}
\end{equation}
On the other hand, from the equation \eqref{rem-3}, there exist some $C$ such that
\begin{equation}
\begin{split}\label{34}
\|N_{i}\|^{2}\leq C_{1}(\|N_{e}\|^{2}+\epsilon\|\partial_{x}N_{e}\|^{2}
+\epsilon^{2}\|\partial_{x}^{2}N_{e}\|^{2}+\epsilon^{3}\|\partial_{x}^{3}N_{e}\|^{2}
+\epsilon^{4}\|\partial_{x}^{4}N_{e}\|^{2}).
\end{split}
\end{equation}
Putting \eqref{31}-\eqref{34} together, we deduce the inequality for $\alpha=0$.

For higher order inequalities, we differentiate \eqref{rem-3} with $\partial_{x}^{\alpha}$ and then take inner product with $\partial_{x}^{\alpha}N_{e}$, $\epsilon\partial_{x}^{\alpha+2}N_{e}$\  and\  $\epsilon^{2}\partial_{x}^{\alpha+4}N_{e}$ separately. The Lemma then follows by the same procedure of the case $\alpha=0$.
\end{proof}

Recall $|\!|\!|(N_{e},U)|\!|\!|_{\epsilon}$ in \eqref{|||}. We remark that only $\|N_{i}\|_{H^{2}}$ can be bounded in terms of $|\!|\!|(N_{e},U)|\!|\!|_{\epsilon}$ and no higher order derivatives of $N_{i}$ are allowed in Lemma \ref{L1}. In fact, we only need $0\leq\alpha\leq2$ in Lemma \ref{L1}.\\

\begin{lemma}\label{L2}
 Let $(N_{i},N_{e},U)$ be a solution to \eqref{rem}. There exist some constants $C$ and $C_{1}=C_{1}(\epsilon\tilde{C})$ such that
\begin{equation}\label{eqaL21}
\begin{split}
\|\epsilon\partial_{t}N_{i}\|^{2}\leq & C\big(\|N_{e}\|_{H^{1}}^{2}+\|U\|_{H^{1}}^{2} +\epsilon\|\partial_{x}^{2}N_{e}\|^{2} +\epsilon^{2}\|\partial_{x}^{3}N_{e}\|^{2}\\
&+\epsilon^{3}\|\partial_{x}^{4}N_{e}\|^{2} +\epsilon^{4}\|\partial_{x}^{5}N_{e}\|^{2}\big) +C\epsilon,
\end{split}
\end{equation}
and
\begin{equation}\label{eqaL22}
\begin{split}
\|\epsilon\partial_{tx}N_{i}\|^{2}\leq & C_{1}(\|N_{e}\|_{H^{2}}^{2} +\|U\|_{H^{2}}^{2} +\epsilon\|\partial_{x}^{3}N_{e}\|^{2} +\epsilon^{2}\|\partial_{x}^{4}N_{e}\|^{2}\\
&+\epsilon^{3}\|\partial_{x}^{5}N_{e}\|^{2} +\epsilon^{4}\|\partial_{x}^{6}N_{e}\|^{2}) +C\epsilon.
\end{split}
\end{equation}
In terms of $|\!|\!|(N_{e},U)|\!|\!|_{\epsilon}$, we can rewrite \eqref{eqaL21} and \eqref{eqaL22} as
\[
\|\epsilon\partial_{t}N_{i}\|_{H^{1}}^{2}\leq C_{1}|\!|\!|(N_{e},U)|\!|\!|_{\epsilon}^{2}+C\epsilon.
\]
\end{lemma}
\begin{proof}
From \eqref{rem-1}, we have
\[
\epsilon\partial_{t}N_{i}=(1-u_{i})\partial_{t}N_{i}-n_{i}\partial_{x}U
-\epsilon\partial_{x}\tilde{n_{i}}U-\epsilon\partial_{x}\tilde{u_{i}}N_{i}-\epsilon^{2}\mathcal{R}_{1}.
\]
Since $\frac{1}{2}<n_{i}<\frac{3}{2}$ and $|u_{i}|<\frac{1}{2}$, taking $L^{2}$-norm yields
\begin{equation*}
\begin{split}
\|\epsilon\partial_tN_{i}\|^2\leq&\|(1-u_{i})\partial_xN_{i}\|^2 +\|n_{i}\partial_xU\|^2 +\epsilon^2\|\partial_x\tilde u_{i}N_{i}\|^2+ \epsilon^2\|\partial_x\tilde n_{i}U\|^2 +\epsilon^4\|\mathcal R_1\|^2\\
\leq & C(\|\partial_xN_{i}\|^2+\|\partial_xU\|^2) +C\epsilon^2(\epsilon^2+\|N_{i}\|^2+\|U\|^2).
\end{split}
\end{equation*}
Applying Lemma \ref{L1} with $\alpha=1$, we deduce \eqref{eqaL21}.
To prove \eqref{eqaL22}, we take $\partial_{x}$ of \eqref{rem-1} to obtain
\begin{equation*}
\begin{split}
\|\epsilon&\partial_{tx}N_{i}\|^{2}
\leq C(\|U\|_{H^{2}}^{2}+\|N_{i}\|_{H^{2}}^{2})+C\epsilon^{6}\int|\partial_{x}{N_{i}}|^{2}|\partial_{x}U|^{2}+C\epsilon^{4}.
\end{split}
\end{equation*}
We note that
\begin{equation*}
\begin{split}
C\epsilon^{6}\|\partial_{x}U\|_{L^{\infty}}^{2}\|\partial_{x}{N_{i}}\|^{2}
\leq C\epsilon^{6}\|U\|_{H^{2}}^{2}\|N_{i}\|_{H^{1}}^{2}
\leq C(\epsilon\tilde{C})\|U\|_{H^{2}}^{2}.
\end{split}
\end{equation*}
Applying Lemma \ref{L1} with $\alpha=2$, we deduce \eqref{eqaL22}.
The Lemma then follows from Lemma \ref{L1}.
\end{proof}

\begin{lemma}\label{L3}
 Let $(N_{i},N_{e},U)$ be a solution to \eqref{rem} and $\alpha\geq0$ be an integer. There exist some constants $C_{1}=C_{1}(\epsilon\tilde{C})$ and $\epsilon_{1}>0$ such that for every $0<\epsilon<\epsilon_{1}$,
\begin{equation}\label{eqaL3}
\begin{split}
\epsilon^{4}\|\partial_{t}\partial_{x}^{\alpha+4}N_{e}\|^{2}&+\epsilon^{3}\|\partial_{t}\partial_{x}^{\alpha+3}N_{e}\|^{2}
+\epsilon^{2}\|\partial_{t}\partial_{x}^{\alpha+2}N_{e}\|^{2}\\ \ \ \ \ \ \ \ \
&+\epsilon\|\partial_{t}\partial_{x}^{\alpha+1}N_{e}\|^{2}
+\|\partial_{t}\partial_{x}^{\alpha}N_{e}\|^{2}
\leq C\|\partial_{t}\partial_{x}^{\alpha}N_{i}\|^{2}+C_{1}.
\end{split}
\end{equation}
\end{lemma}
\begin{proof}
The proof is similar to that of Lemma \ref{L1}. When $\alpha=0$, by first taking $\partial_{t}$ of \eqref{rem-3} and then taking inner product with $\partial_{t}N_{e}$ and integration by parts, we have
\begin{equation}
\begin{split}\label{m1}
\|\partial_{t}&N_{e}\|^{2}
+\epsilon\int n_{e}(\partial_{tx}N_{e})^{2}+\frac{\epsilon^{2}H^{2}}{4}\int\frac{1}{n_{e}}(\partial_{t}\partial_{x}^{2}N_{e})^{2}\\
=&-\epsilon\int\partial_{x}n_{e}\partial_{tx}N_{e}\partial_{t}N_{e}
+\epsilon\int\partial_{t}n_{e}\partial_{x}^{2}N_{e}\partial_{t}N_{e}
-\frac{\epsilon^{2}H^{2}}{2}\int(\partial_{x}\frac{1}{n_{e}})\partial_{t}\partial_{x}^{2}N_{e}\partial_{tx}N_{e}\\
&-\frac{\epsilon^{2}H^{2}}{4}\int(\partial_{x}^{2}\frac{1}{n_{e}})\partial_{t}\partial_{x}^{2}N_{e}\partial_{t}N_{e}
-\frac{\epsilon^{2}H^{2}}{4}\int(\partial_{t}\frac{1}{n_{e}})\partial_{x}^{4}N_{e}\partial_{t}N_{e}\\
&+2\epsilon^{2}\int\partial_{t}(\partial_{x}\tilde{n_{e}}\partial_{x}N_{e})\partial_{t}N_{e}
+\epsilon^{4}\int\partial_{t}\big((\partial_{x}N_{e})^{2}\big)\partial_{t}N_{e}
+\epsilon^{2}\int\partial_{t}(\partial_{x}^{2}\tilde{n_{e}}N_{e})\partial_{t}N_{e}\\
&+\epsilon^{2}\int\partial_{t}\mathcal{R}_{3}^{1}\partial_{t}N_{e}
-\frac{H^{2}}{4}\Big\{-12\epsilon^{5}\int\partial_{t}\big[\frac{(\partial_{x}\tilde{n_{e}})^{3}}{n_{e}^{4}}\partial_{x}N_{e}\big]\partial_{t}N_{e}\\
&+14\epsilon^{4}\int\partial_{t}\big[\frac{\partial_{x}\tilde{n_{e}}\partial_{x}^{2}\tilde{n_{e}}}{n_{e}^{3}}\partial_{x}N_{e}\big]\partial_{t}N_{e}
-3\epsilon^{3}\int\partial_{t}\big[\frac{\partial_{x}^{3}\tilde{n_{e}}}{n_{e}^{2}}\partial_{x}N_{e}\big]\partial_{t}N_{e}\\
&-18\epsilon^{7}\int\partial_{t}\big[\frac{(\partial_{x}\tilde{n_{e}})^{2}}{n_{e}^{4}}(\partial_{x}N_{e})^{2}\big]\partial_{t}N_{e}
+7\epsilon^{4}\int\partial_{t}\big[\frac{(\partial_{x}\tilde{n_{e}})^{2}}{n_{e}^{3}}\partial_{x}^{2}N_{e}\big]\partial_{t}N_{e}\\
&+7\epsilon^{6}\int\partial_{t}\big[\frac{\partial_{x}^{2}\tilde{n_{e}}}{n_{e}^{3}}(\partial_{x}N_{e})^{2}\big]\partial_{t}N_{e}
-4\epsilon^{3}\int\partial_{t}\big[\frac{\partial_{x}^{2}\tilde{n_{e}}}{n_{e}^{2}}\partial_{x}^{2}N_{e}\big]\partial_{t}N_{e}\\
&-12\epsilon^{9}\int\partial_{t}\big[\frac{\partial_{x}\tilde{n_{e}}}{n_{e}^{4}}(\partial_{x}N_{e})^{3}\big]\partial_{t}N_{e}
+14\epsilon^{6}\int\partial_{t}\big[\frac{\partial_{x}\tilde{n_{e}}}{n_{e}^{3}}\partial_{x}N_{e}\partial_{x}^{2}N_{e}\big]\partial_{t}N_{e}\\
&-3\epsilon^{3}\int\partial_{t}\big[\frac{\partial_{x}\tilde{n_{e}}}{n_{e}^{2}}\partial_{x}^{3}N_{e}\big]\partial_{t}N_{e}
-3\epsilon^{11}\int\partial_{t}\big[\frac{1}{n_{e}^{4}}(\partial_{x}N_{e})^{4}\big]\partial_{t}N_{e}\\
&+7\epsilon^{8}\int\partial_{t}\big[\frac{1}{n_{e}^{3}}(\partial_{x}N_{e})^{2}\partial_{x}^{2}N_{e}\big]\partial_{t}N_{e}
-2\epsilon^{5}\int\partial_{t}\big[\frac{1}{n_{e}^{2}}(\partial_{x}^{2}N_{e})^{2}\big]\partial_{t}N_{e}\\
&-3\epsilon^{5}\int\partial_{t}\big[\frac{1}{n_{e}^{2}}\partial_{x}N_{e}\partial_{x}^{3}N_{e}\big]\partial_{t}N_{e}
+\epsilon^{2}\int\partial_{t}\big[\frac{1}{n_{e}^{4}}(\mathcal{R}_{3}^{2}+\mathcal{R}_{3}^{3})\big]\partial_{t}N_{e}\Big\}\\
&+\int\partial_{t}N_{i}\partial_{t}N_{e}\\
=&:\sum_{i=1}^{25}B_{i} \ .
\end{split}
\end{equation}

\emph{Estimate of the LHS of \eqref{m1}.}
Since $\frac{1}{2}<n_{e}<\frac{3}{2}$ and $H$ is a fixed constant, there exists a fixed constant $C$ such that
\begin{equation*}
\begin{split}
\epsilon\int n_{e}(\partial_{tx}N_{e})^{2} +\frac{\epsilon^{2}H^{2}}{4}\int\frac{1}{n_{e}} (\partial_{t}\partial_{x}^{2}N_{e})^{2}
\geq C(\epsilon\|\partial_{tx}N_{e}\|^{2} +\epsilon^{2}\|\partial_{t} \partial_{x}^{2}N_{e}\|^{2}).
\end{split}
\end{equation*}
Thus the LHS of \eqref{m1} is equal or greater than $C(\|\partial_{t}N_{e}\|^{2} +\epsilon\|\partial_{tx}N_{e}\|^{2} +\epsilon^{2} \|\partial_{t}\partial_{x}^{2}N_{e}\|^{2})$. Next, we estimate the righthand side terms. For $B_{1}$, by applying H\"older's inequality, Cauchy inequality and Sobolev embedding $H^1\hookrightarrow L^{\infty}$, we have
\begin{equation*}
\begin{split}
B_{1}&=\epsilon^{2}\int(\partial_{x}\tilde{n_{e}} +\epsilon^{2}\partial_{x}N_{e})\partial_{tx}N_{e} \partial_{t}N_{e}\\
&\leq C\epsilon(1+\epsilon^{2} \|\partial_{x}N_{e}\|_{L^{\infty}}) (\epsilon\|\partial_{t}N_{e}\|^{2} +\epsilon^{2}\|\partial_{tx}N_{e}\|^{2})\\
&\leq C(\epsilon\tilde{C}) (\epsilon\|\partial_{t}N_{e}\|^{2} +\epsilon^{2}\|\partial_{tx}N_{e}\|^{2})\\
&\leq C_{1}(\epsilon\|\partial_{t}N_{e}\|^{2} +\epsilon^{2}\|\partial_{tx}N_{e}\|^{2}),
\end{split}
\end{equation*}
where we have used \eqref{priori}. Similarly,
\begin{equation*}
\begin{split}
B_{2}\leq C_{1}(\epsilon\|\partial_{t}N_{e}\|^{2} +\epsilon^{2}\|\partial_{tx}N_{e}\|^{2})+C_{1}.
\end{split}
\end{equation*}
By \eqref{priori}, Sobolev embedding theorem and Cauchy inequality, we have
\begin{equation*}
\begin{split}
B_{3}+B_{4}\leq C_{1}(\epsilon\|\partial_{t}N_{e}\|^{2}+\epsilon^{2}\|\partial_{tx}N_{e}\|^{2}
+\epsilon^{3}\|\partial_{t}\partial_{x}^{2}N_{e}\|^{2}),
\end{split}
\end{equation*}
where we have used \eqref{one order} and \eqref{two orders}.

\emph{Estimate of $B_{5}$.} Similar to \eqref{one order}, we note that
\begin{equation}
\begin{split}\label{one}
|\partial_{t}\frac{1}{n_{e}}|
\leq C(\epsilon|\partial_{t}\tilde{n_{e}}|+\epsilon^{3}|\partial_{t}N_{e}|).
\end{split}
\end{equation}
Therefore, we have
\begin{equation*}
\begin{split}
B_{5}&\leq\epsilon^{2} \|\partial_{x}^{4}N_{e}\|
(\epsilon\|\partial_{t}\tilde{n_{e}}\|_{L^{\infty}}+\epsilon^{3}\|\partial_{t}{N_{e}}\|_{L^{\infty}})\|\partial_{t}N_{e}\|\\
&\leq \epsilon^{2}\tilde{C}(\epsilon C+\epsilon^{3}\|\partial_{t}{N_{e}}\|_{H^{1}})\|\partial_{t}N_{e}\|\\
&\leq C_{1}(\epsilon\|\partial_{t}N_{e}\|^{2}+\epsilon^{2}\|\partial_{tx}N_{e}\|^{2})+C_{1}.
\end{split}
\end{equation*}

\emph{Estimate of $B_{6}$.} By direct computation, we have
\begin{equation*}
\begin{split}
\partial_{t}(\partial_{x}\tilde{n_{e}}\partial_{x}N_{e})
=\partial_{tx}\tilde{n_{e}}\partial_{x}N_{e}+\partial_{x}\tilde{n_{e}}\partial_{tx}N_{e},
\end{split}
\end{equation*}
which yields that
\begin{equation*}
\begin{split}
\|\partial_{t}(\partial_{x}\tilde{n_{e}}\partial_{x}N_{e})\|
\leq C(\|\partial_{x}N_{e}\|+\|\partial_{tx}N_{e}\|),
\end{split}
\end{equation*}
where $C$ is a fixed constant. By applying H\"older inequality and Young inequality, we have
\begin{equation*}
\begin{split}
B_{6}\leq C_{1}(\epsilon\|\partial_{t}N_{e}\|^{2}+\epsilon^{2}\|\partial_{tx}N_{e}\|^{2})+C_{1}.
\end{split}
\end{equation*}
$B_{8}$ is similar to $B_{6}$.

\emph{Estimate of $B_{7}$.} We note that
\begin{equation*}
\begin{split}
\|\partial_{t}[(\partial_{x}N_{e})^{2}]\|
\leq C(\|\partial_{x}N_{e}\|_{L^{\infty}}\|\partial_{tx}N_{e}\|).
\end{split}
\end{equation*}
Thus, we have
\begin{equation*}
\begin{split}
B_{7}\leq C_{1}(\epsilon\|\partial_{t}N_{e}\|^{2}+\epsilon^{2}\|\partial_{tx}N_{e}\|^{2}),
\end{split}
\end{equation*}
thanks to H\"older inequality, Cauchy inequality and Sobolev embedding $H^1\hookrightarrow L^{\infty}$ and \eqref{priori}.

\emph{Estimate of $B_{9}$.} Since $\mathcal{R}_{3}^{1}$ is known, thus by Cauchy inequality, we have
\begin{equation*}
\begin{split}
B_{9}\leq C_{1}\epsilon\|\partial_{t}N_{e}\|^{2}+C_{1}.
\end{split}
\end{equation*}

Estimate of $B_{20}$. By direct computation, we have
\begin{equation*}
\begin{split}
\partial_{t}\left[\frac{1}{n_{e}^{4}} (\partial_{x}N_{e})^{4}\right]
=\partial_{t}\left(\frac{1}{n_{e}^{4}}\right)(\partial_{x}N_{e})^{4}+\frac{4}{n_{e}^{4}}(\partial_{x}N_{e})^{3}\partial_{tx}N_{e}.
\end{split}
\end{equation*}
Similar to \eqref{one}, we have
\begin{equation}
\begin{split}\label{repeat}
\left|\partial_{t}\left(\frac{1}{n_{e}^{4}}\right)\right|
\leq C(\epsilon+\epsilon^{3}|\partial_{t}N_{e}|).
\end{split}
\end{equation}
Thus by applying H\"older inequality, Sobolev embedding $H^1\hookrightarrow L^{\infty}$ and \eqref{priori} again, we have
\begin{equation*}
\begin{split}
B_{20}&\leq C\epsilon^{11}\|\partial_{t}[\frac{1}{n_{e}^{4}}(\partial_{x}N_{e})^{4}]\|\|\partial_{t}N_{e}\|\\
&\leq C\epsilon^{11}\|\partial_{x}N_{e}\|_{L^{\infty}}^{4}(\epsilon+\epsilon^{3}\|\partial_{t}N_{e}\|)\|\partial_{t}N_{e}\|
+C\epsilon^{11}\|\partial_{x}N_{e}\|_{L^{\infty}}^{3}(\|\partial_{t}N_{e}\|\|\partial_{tx}N_{e}\|)\\
&\leq C_{1}(\epsilon\|\partial_{t}N_{e}\|^{2}+\epsilon^{2}\|\partial_{tx}N_{e}\|^{2})+C_{1}.
\end{split}
\end{equation*}
The estimates of $B_{10\sim13},B_{15} \ and \ B_{17}$ are similar to that for $B_{20}$.

\emph{Estimate of $B_{21}$.} By direct computation, we have
\begin{equation*}
\begin{split}
\partial_{t}\left[\frac{1}{n_{e}^{3}} (\partial_{x}N_{e})^{2}\partial_{x}^{2}N_{e}\right]
=\partial_{t}\left(\frac{1}{n_{e}^{3}}\right)(\partial_{x}N_{e})^{2}\partial_{x}^{2}N_{e}
+\frac{2}{n_{e}^{3}}\partial_{x}N_{e}\partial_{x}^{2}N_{e}\partial_{tx}N_{e}
+\frac{1}{n_{e}^{3}}(\partial_{x}N_{e})^{2}\partial_{t}\partial_{x}^{2}N_{e}.
\end{split}
\end{equation*}
Thus similarly, we have
\begin{equation*}
\begin{split}
B_{21}\leq &C\epsilon^{8}\|\partial_{t}\big[\frac{1}{n_{e}^{3}}(\partial_{x}N_{e})^{2}\partial_{x}^{2}N_{e}\big]\|\|\partial_{t}N_{e}\|\\
\leq &C\epsilon^{8}[\|\partial_{x}N_{e}\|_{L^{\infty}}^{2}\|\partial_{x}^{2}N_{e}\|_{L^{\infty}}
(\epsilon+\epsilon^{3}\|\partial_{t}N_{e}\|)
+\|\partial_{x}N_{e}\|_{L^{\infty}}\|\partial_{x}^{2}N_{e}\|_{L^{\infty}}\|\partial_{tx}N_{e}\|\\
&+\|\partial_{x}N_{e}\|_{L^{\infty}}^{2}\|\partial_{t}\partial_{x}^{2}N_{e}\|]\|\partial_{t}N_{e}\|\\
\leq &C_{1}\|N_{e}\|_{H^{2}}^{2}(1+\epsilon\|\partial_{x}^{3}N_{e}\|^{2})\epsilon\|\partial_{t}N_{e}\|^{2}\\
&+C_{1}(\|N_{e}\|_{H^{2}}^{2}+\epsilon\|\partial_{x}^{3}N_{e}\|^{2})(\epsilon\|\partial_{t}N_{e}\|^{2}+\epsilon^{2}\|\partial_{tx}N_{e}\|^{2})\\
&+C_{1}\|N_{e}\|_{H^{2}}^{2}((\epsilon\|\partial_{t}N_{e}\|^{2}+\epsilon^{3}\|\partial_{t}\partial_{x}^{2}N_{e}\|^{2}))\\
\leq &C_{1}(\epsilon\|\partial_{t}N_{e}\|^{2}+\epsilon^{2}\|\partial_{tx}N_{e}\|^{2}+\epsilon^{3}\|\partial_{t}\partial_{x}^{2}N_{e}\|^{2})+C_{1}.
\end{split}
\end{equation*}
The estimates of $B_{14},B_{16},B_{18}\ and \ B_{19}$ are similar to that for $B_{21}$.

\emph{Estimate of $B_{23}$.} By direct computation, we have
\begin{equation*}
\begin{split}
\partial_{t}\left[\frac{1}{n_{e}^{2}} \partial_{x}N_{e}\partial_{x}^{3}N_{e}\right]
&=\partial_{t}\left(\frac{1}{n_{e}^{2}}\right)\partial_{x}N_{e}\partial_{x}^{3}N_{e}
+\frac{1}{n_{e}^{2}}\partial_{tx}N_{e}\partial_{x}^{3}N_{e}
+\frac{1}{n_{e}^{2}}\partial_{x}N_{e}\partial_{t}\partial_{x}^{3}N_{e}\\
&=:G_{1}+G_{2}+G_{3}.
\end{split}
\end{equation*}
Thus $B_{23}$ is divided three terms
\begin{equation*}
\begin{split}
B_{23}=\frac{3\epsilon^{5}H^{2}}{4}\sum_{i=1}^{3}\int G_{i}\partial_{t}N_{e}=:B_{231}+B_{232}+B_{233}.
\end{split}
\end{equation*}
The first two terms $B_{231}$ and $B_{232}$ can be easily estimated by $C_{1}(\epsilon\|\partial_{t}N_{e}\|^{2} +\epsilon^{2}\|\partial_{tx}N_{e}\|^{2})$. For the last term $B_{233}$, we integrate by parts and use H\"older inequality, Cauchy inequality and \eqref{priori} again to obtain
\begin{equation*}
\begin{split}
B_{233}=&-\frac{3H^{2}}{4}\epsilon^{5}\int\left((\partial_{x}\frac{1}{n_{e}^{2}})\partial_{x}N_{e}
+\frac{1}{n_{e}^{2}}\partial_{x}^{2}N_{e}\right)\partial_{t}\partial_{x}^{2}N_{e}\partial_{t}N_{e}\\
&-\frac{3H^{2}}{4}\epsilon^{5}\int\frac{1}{n_{e}^{2}}\partial_{x}N_{e}\partial_{t}\partial_{x}^{2}N_{e}\partial_{tx}N_{e}\\
\leq &C_{1}(\epsilon\|\partial_{t}N_{e}\|^{2}+\epsilon^{2}\|\partial_{tx}N_{e}\|^{2}+\epsilon^{3}\|\partial_{t}\partial_{x}^{2}N_{e}\|^{2}),
\end{split}
\end{equation*}
where we also have used  \eqref{one order}.
Thus we have
\begin{equation*}
\begin{split}
B_{23}\leq C_{1}(\epsilon\|\partial_{t}N_{e}\|^{2}+\epsilon^{2}\|\partial_{tx}N_{e}\|^{2}
+\epsilon^{3}\|\partial_{t}\partial_{x}^{2}N_{e}\|^{2}).
\end{split}
\end{equation*}
$B_{22}$ is similar to $B_{23}$.

\emph{Estimate of $B_{24}$.} Since $\mathcal{R}_{3}^{2}$ is known, by
using \eqref{equ34} in Lemma \ref{L8}, we have
\[
B_{24}\leq C_{1}(1+\epsilon\|\partial_{t}N_{e}\|^{2}).
\]

\emph{Estimate of $B_{25}$.} Applying Young inequality, we have
\[
B_{25}=\int\partial_{t}N_{i}\partial_{t}N_{e}\leq\gamma\|\partial_{t}N_{e}\|^{2}+C_{\gamma}\|\partial_{t}N_{i}\|^{2},
\]
where for arbitrary small $\gamma>0$. Hence, we have shown that there exists some $\epsilon_1>0$ such that for $0<\epsilon<\epsilon_1$, we have
\begin{equation}\label{lemma31}
\begin{split}
\|\partial_{t}N_{e}\|^{2} +\epsilon\|\partial_{tx}N_{e}\|^{2} +\epsilon^{2}\|\partial_{t} \partial_{x}^{2}N_{e}\|^{2}\leq C\|\partial_{t}N_{i}\|^{2}+C_{1}.
\end{split}
\end{equation}
Similarly, taking $\partial_{tx}$ of \eqref{rem-3} and then taking inner product with $\epsilon\partial_{tx}N_{e}$, we have
\begin{equation}\label{lemma32}
\begin{split}
\epsilon\|\partial_{tx}N_{e}\|^{2}+\epsilon^{2}\|\partial_{t}\partial_{x}^{2}N_{e}\|^{2}
+\epsilon^{3}\|\partial_{t}\partial_{x}^{3}N_{e}\|^{2}
\leq C_{\alpha_{2}}\|\partial_{t}N_{i}\|^{2} +\epsilon\|\partial_{t}N_{e}\|^{2}+C_{1}.
\end{split}
\end{equation}
Taking $\partial_{t}\partial_{x}^{2}$ of \eqref{rem-3} and then taking inner product with $\epsilon^{2}\partial_{t}\partial_{x}^{2}N_{e}$, we have
\begin{equation}\label{lemma33}
\begin{split}
\epsilon^{2}\|\partial_{t}\partial_{x}^{2}N_{e}\|^{2}+\epsilon^{3}\|\partial_{t}\partial_{x}^{3}N_{e}\|^{2}
+\epsilon^{4}\|\partial_{t}\partial_{x}^{4}N_{e}\|^{2}
\leq &C_{\alpha_{3}}\|\partial_{t}N_{i}\|^{2}+\epsilon^{2}\|\partial_{tx}N_{e}\|^{2}\\
&+\epsilon\|\partial_{t}N_{e}\|^{2}+C_{1}.
\end{split}
\end{equation}
Putting \eqref{lemma31}, \eqref{lemma32} and \eqref{lemma33} together, let $C=\max\{C_{\alpha_{1}},C_{\alpha_{2}},C_{\alpha_{3}}\}$, we obtain
\begin{equation*}
\begin{split}
\|\partial_{t}N_{e}\|^{2}&+\epsilon\|\partial_{tx}N_{e}\|^{2}+\epsilon^{2}\|\partial_{t}\partial_{x}^{2}N_{e}\|^{2}
+\epsilon^{3}\|\partial_{t}\partial_{x}^{3}N_{e}\|^{2}
+\epsilon^{4}\|\partial_{t}\partial_{x}^{4}N_{e}\|^{2}
\leq C\|\partial_{t}N_{i}\|^{2}+C_{1}.
\end{split}
\end{equation*}
Thus we have proven \eqref{eqaL3} for $\alpha=0$. The case of $\alpha\geq1$ can be proved similarly.
\end{proof}

\subsection{Zeroth, first and second order estimates}\label{2.2}
The zeroth, first and second order estimates can be summarized in the following
\begin{proposition}\label{P1-P2}
Let $(N_{i},N_{e},U)$ be a solution to \eqref{rem} and $\gamma=0,1,2$, then
\begin{equation}\label{equ10}
\begin{split}
\frac{1}{2}\frac{d}{dt}&\|\partial_{x}^{\gamma}U\|^{2} +\frac{1}{2}\frac{d}{dt} \left[\int\frac{n_{e}}{n_{i}}(\partial_{x}^{\gamma} N_{e})^{2} +\epsilon\int\frac{n_{e}^{2}}{n_{i}} (\partial_{x}^{\gamma+1}N_{e})^{2} +\frac{\epsilon^{2}H^{2}}{4} \int\frac{1}{n_{i}}(\partial_{x}^{\gamma+2}N_{e})^{2}\right]\\
&+\frac{1}{2}\frac{\epsilon H^{2}}{4}\frac{d}{dt}\left[\int \frac{(\partial_{x}^{\gamma+1}N_{e})^{2}}{n_{e}n_{i}} +\epsilon\int\frac{1}{n_{i}} (\partial_{x}^{\gamma+2}N_{e})^{2} +\frac{\epsilon^{2}H^{2}}{4}\int\frac{1}{n_{e}^{2} n_{i}}(\partial_{x}^{\gamma+3}N_{e})^{2}\right]\\
\leq & C_{1}(1+\epsilon|\!|\!|(N_{e},U)|\!|\!|_{\epsilon}^{4})(1+|\!|\!|(N_{e},U)|\!|\!|_{\epsilon}^{2}).
\end{split}
\end{equation}
\end{proposition}
This proposition can be proved after long tedious calculations, which can be done by the same procedure that used in the proof of Proposition \ref{P3}. Hence we omit the details here for simplicity.

\subsection{Third order estimates}\label{2.3}
\begin{proposition}\label{P3}
 Let $(N_{i},N_{e},U)$ be a solution to \eqref{rem}£¬ then
\begin{equation}\label{equ111}
\begin{split}
\frac{\epsilon}{2}\frac{d}{dt}&\|\partial_{x}^{3}U\|^{2}+\frac{\epsilon}{2}\frac{d}{dt}\left[\int\frac{n_{e}}{n_{i}}(\partial_{x}^{3}N_{e})^{2}
+\epsilon\int\frac{n_{e}^{2}}{n_{i}} (\partial_{x}^{4}N_{e})^{2} +\frac{\epsilon^{2}H^{2}}{4} \int\frac{1}{n_{i}}(\partial_{x}^{5}N_{e})^{2} \right]\\
&+\frac{\epsilon}{2}\frac{\epsilon H^{2}}{4}\frac{d}{dt}\left[\int\frac{1}{n_{e}n_{i}}(\partial_{x}^{4}N_{e})^{2}
+\epsilon\int\frac{1}{n_{i}}(\partial_{x}^{5}N_{e})^{2}
+\frac{\epsilon^{2}H^{2}}{4}\int\frac{1}{n_{e}^{2}n_{i}}(\partial_{x}^{6}N_{e})^{2}\right]\\
\leq & C_{1}(1+\epsilon^{2}|\!|\!|(N_{e},U)|\!|\!|_{\epsilon}^{6})(1+|\!|\!|(N_{e},U)|\!|\!|_{\epsilon}^{2}).
\end{split}
\end{equation}
\end{proposition}

\begin{proof}
We take $\partial_{x}^{3}$ of \eqref{rem-2} and then take inner product of $\epsilon\partial_{x}^{3}U$. We obtain
\begin{equation}\label{equ11}
\begin{split}
\frac{\epsilon}{2}\frac{d}{dt}&\|\partial_{x}^{3}U\|^{2}-\int\partial_{x}^{3}\big((1-u_{i})\partial_{x}U\big)\partial_{x}^{3}U
+\epsilon\int\partial_{x}^{3}\big(\partial_{x}\tilde u_{i}U\big)\partial_{x}^{3}U\\
=&-\int\partial_{x}^{3}\big(n_{e}\partial_{x}N_{e}\big)\partial_{x}^{3}U
+\frac{\epsilon H^{2}}{4}\int\partial_{x}^{3}\big(\frac{\partial_{x}^{3}N_{e}}{n_{e}}\big)\partial_{x}^{3}U
-\epsilon\int\partial_{x}^{3}\big(\partial_{x}\tilde n_{e}N_{e}\big)\partial_{x}^{3}U\\
&+\frac{H^{2}}{4}\Big\{3\epsilon^{3}\int\partial_{x}^{3}\big(\frac{(\partial_{x}\tilde{n_{e}})^{2}}{n_{e}^{3}}\partial_{x}N_{e}\big)\partial_{x}^{3}U
-2\epsilon^{2}\int\partial_{x}^{3}\big(\frac{\partial_{x}^{2}\tilde{n_{e}}}{n_{e}^{2}}\partial_{x}N_{e}\big)\partial_{x}^{3}U\\
&+3\epsilon^{5}\int\partial_{x}^{3}\big(\frac{\partial_{x}\tilde{n_{e}}}{n_{e}^{3}}(\partial_{x}N_{e})^{2}\big)\partial_{x}^{3}U
-2\epsilon^{2}\int\partial_{x}^{3}\big(\frac{\partial_{x}\tilde{n_{e}}}{n_{e}^{2}}\partial_{x}^{2}N_{e}\big)\partial_{x}^{3}U\\
&+\epsilon^{7}\int\partial_{x}^{3}\big(\frac{(\partial_{x}N_{e})^{3}}{n_{e}^{3}}\big)\partial_{x}^{3}U
-2\epsilon^{4}\int\partial_{x}^{3}\big(\frac{\partial_{x}N_{e}\partial_{x}^{2}N_{e}}{n_{e}^{2}}\big)\partial_{x}^{3}U\\
&+\epsilon^{2}\int\partial_{x}^{3}\big(\frac{\mathcal{R}_{2}^{3}+\mathcal{R}_{2}^{4}}{n_{e}^{3}}\big)\partial_{x}^{3}U\Big\}
+\epsilon^{2}\int(\partial_{x}^{3}\mathcal{R}_{2}^{1,2})\partial_{x}^{3}U\\
=&:\sum_{i=1}^{11}F_{i} \ .
\end{split}
\end{equation}

\emph{Estimate of the LHS of \eqref{equ11}.} First, we estimate the second term on the LHS of \eqref{equ11}. Using commutator notation \eqref{w} to rewrite it as
\begin{equation*}
\begin{split}
-\int\partial_{x}^{3}\big((1-u_{i}) \partial_{x}U\big)\partial_{x}^{3}U
&=-\int\big([\partial_{x}^{3},1-u_{i}]\partial_{x}U +(1-u_{i})\partial_{x}^{4}U\big)\partial_{x}^{3}U =:M_{1}+M_{2}
\end{split}
\end{equation*}
We first estimate $M_1$. By commutator estimate of Lemma \ref{L9}, we have
\begin{equation*}
\begin{split}
\|[\partial_{x}^{3},1-u_{i}]\partial_{x}U\|
\leq\|\partial_{x}(1-u_{i})\|_{L^{\infty}}\|\partial_{x}^{3}U\|
+\|\partial_{x}^{3}(1-u_{i})\|\|\partial_{x}U\|_{L^{\infty}}.
\end{split}
\end{equation*}
Thus by H\"older inequality, Cauchy inequality and Sobolev embedding theorem $H^1\hookrightarrow L^{\infty}$, we have
\begin{equation}
\begin{split}\label{G2}
| M_{1}|
\leq &\|[\partial_{x}^{3},1-u_{i}]\partial_{x}U\|\|\partial_{x}^{3}U\|\\
\leq &C(1+\epsilon^{2}\|\partial_{x}U\|_{L^{\infty}}^{2})(\epsilon\|\partial_{x}^{3}U\|^{2})
+C\epsilon(1+\epsilon^{2}\|\partial_{x}^{3}U\|^{2})(\|\partial_{x}U\|_{L^{\infty}}^{2}+\epsilon\|\partial_{x}^{3}U\|^{2})\\
\leq &C_{1}\big(1+\epsilon^{2}(\|U\|_{H^{2}}^{2}+\epsilon\|\partial_{x}^{3}U\|^{2})\big)(\|U\|_{H^{2}}^{2}+\epsilon\|\partial_{x}^{3}U\|^{2})\\
\leq &C_{1}(1+\epsilon^{2}|\!|\!|(N_{e},U)|\!|\!|_{\epsilon}^{2})|\!|\!|(N_{e},U)|\!|\!|_{\epsilon}^{2},
\end{split}
\end{equation}
where $|\!|\!|(N_{e},U)|\!|\!|_{\epsilon}^{2}$ is given in \eqref{|||}. Next, we estimate $M_{2}$. By integration by parts, we have
\begin{equation}
\begin{split}\label{G3}
|M_{2}|&=|\frac{1}{2}\int\partial_{x}(1-u_{i})(\partial_{x}^{3}U)^{2}|\\
&\leq C(1+\epsilon^{2}\|\partial_{x}U\|_{L^{\infty}})(\epsilon\|\partial_{x}^{3}U\|^{2})\\
&\leq C(1+\epsilon^{2}\|U\|_{H^{2}})(\epsilon\|\partial_{x}^{3}U\|^{2}),
\end{split}
\end{equation}
where we have used Sobolev embedding theorem $H^1\hookrightarrow L^{\infty}$. In light of \eqref{G2} and \eqref{G3}, we find the second term on the LHS of \eqref{equ11} can be bounded by $C(1+\epsilon^{2}|\!|\!|(N_{e}, U)|\!|\!|_{\epsilon}^{2})|\!|\!|(N_{e}, U)|\!|\!|_{\epsilon}^{2}$. The third term on the LHS of \eqref{equ11} is bilinear in the unknowns and can be bounded by
\begin{equation*}
\begin{split}
\epsilon\int\partial_{x}^{3}(\partial_{x}\tilde u_{i}U)\partial_{x}^{3}U
\leq C(\|U\|_{H^{2}}^{2}+\epsilon\|\partial_{x}^{3}U\|^{2})
\leq C|\!|\!|(N_{e},U)|\!|\!|_{\epsilon}^{2}.
\end{split}
\end{equation*}
Next, we estimate of the RHS of \eqref{equ11}. We first estimate the terms $F_{i}$ for $3\leq i\leq11$.

\emph{Estimate of $F_{3}$.} Since $F_{3}$ is bilinear in the unknowns, it can be bounded by
\begin{equation*}
\begin{split}
F_{3}\le C(\|N_{e}\|_{H^{2}}^{2}+\epsilon\|\partial_{x}^{3}N_{e}\|^{2}+\epsilon\|\partial_{x}^{3}U\|^{2}),
\end{split}
\end{equation*}
where we have used Cauchy inequality.

\emph{Estimate of $F_{8}$.} Using commutator notation \eqref{w}, we write
\begin{equation*}
\begin{split}
F_{8}=\frac{\epsilon^{7}H^{2}}{4}\int\Big\{[\partial_{x}^{3},\frac{1}{n_{e}^{3}}](\partial_{x}N_{e})^{3}
+\frac{1}{n_{e}^{3}}\partial_{x}^{3}\left((\partial_{x}N_{e})^{3}\right)\Big\}\partial_{x}^{3}U
=:F_{81}+F_{82}.
\end{split}
\end{equation*}
By commutator estimates \eqref{w11} in Lemma \ref{L9}, we have
\begin{equation*}
\begin{split}
\|[\partial_{x}^{3},\frac{1}{n_{e}^{3}}](\partial_{x}N_{e})^{3}\|
\leq\|\partial_{x}(\frac{1}{n_{e}^{3}})\|_{L^{\infty}}\|\partial_{x}^{2}\left((\partial_{x}^{3}N_{e})^{3}\right)\|
+\|\partial_{x}^{3}(\frac{1}{n_{e}^{3}})\|\|\partial_{x}N_{e}\|_{L^{\infty}}^{3}.
\end{split}
\end{equation*}
By direct computation and Sobolev embedding theorem, we note that
\begin{equation}
\begin{split}\label{use 1}
\|\partial_{x}(\frac{1}{n_{e}^{3}})\|_{L^{\infty}}
\leq C(\epsilon+\epsilon^{3}\|\partial_{x}N_{e}\|_{L^{\infty}})
\leq C(\epsilon+\epsilon^{3}\|N_{e}\|_{H^{2}}),
\end{split}
\end{equation}
and
\begin{equation}
\begin{split}\label{use 2}
|\partial_{x}^{3}(\frac{1}{n_{e}^{3}})|
\leq &C(\epsilon+\epsilon^{3}(|\partial_{x}N_{e}|+|\partial_{x}^{2}N_{e}|+|\partial_{x}^{3}N_{e}|)\\
&+\epsilon^{6}(|\partial_{x}N_{e}|^{2} +|\partial_{x}N_{e}||\partial_{x}^{2}N_{e}|) +\epsilon^{9}|\partial_{x}N_{e}|^{3}),
\end{split}
\end{equation}
which yields that
\begin{equation}
\begin{split}\label{use 21}
\|\partial_{x}^{3}(\frac{1}{n_{e}^{3}})\|
\leq &C\big(\epsilon+\epsilon^{3}(\|\partial_{x}N_{e}\|+\|\partial_{x}^{2}N_{e}\|+\|\partial_{x}^{3}N_{e}\|)\\
&+\epsilon^{6}(\|\partial_{x}N_{e}\|_{L^{\infty}}\|\partial_{x}N_{e}\|+\|\partial_{x}N_{e}\|_{L^{\infty}}\|\partial_{x}^{2}N_{e}\|)\\
&+\epsilon^{9}\|\partial_{x}N_{e}\|_{L^{\infty}}^{2}\|\partial_{x}N_{e}\|\big).
\end{split}
\end{equation}
By direct computation, we have
\begin{equation}
\begin{split}\label{use 9}
\|\partial_{x}^{2}\left[(\partial_{x}N_{e})^{3}\right]\|
\leq C(\|\partial_{x}^{2}N_{e}\|_{L^{\infty}}^{2}\|\partial_{x}N_{e}\|+\|\partial_{x}N_{e}\|_{L^{\infty}}^{2}\|\partial_{x}^{3}N_{e}\|).
\end{split}
\end{equation}
Therefore, by \eqref{use 1},\ \eqref{use 21}\ and \eqref{use 9}, and using H\"older inequality and Sobolev embedding $H^1\hookrightarrow L^{\infty}$, we can obtain
\begin{equation}
\begin{split}\label{s1}
F_{81}
\leq C_{1}(1+\epsilon^{2}|\!|\!|(N_{e},U)|\!|\!|_{\epsilon}^{6})(1+|\!|\!|(N_{e},U)|\!|\!|_{\epsilon}^{2}).
\end{split}
\end{equation}
On the other hand, by direct computation, we have
\begin{equation}
\begin{split}\label{use 10}
\|\partial_{x}^{3}\left[(\partial_{x}N_{e})^{3}\right]\|
\leq &C(\|\partial_{x}^{2}N_{e}\|_{L^{\infty}}^{2} \|\partial_{x}^{2}N_{e}\| \\
&+\|\partial_{x}N_{e}\|_{L^{\infty}} \|\partial_{x}^{2}N_{e}\|_{L^{\infty}} \|\partial_{x}^{3}N_{e}\| +\|\partial_{x}N_{e}\|_{L^{\infty}}^{2} \|\partial_{x}^{4}N_{e}\|).
\end{split}
\end{equation}
By applying H\"older inequality and Sobolev embedding theorem again, we have
\begin{equation}
\begin{split}\label{s2}
F_{82}
\leq C_{1}(1+\epsilon^{2}|\!|\!|(N_{e},U)|\!|\!|_{\epsilon}^{2})|\!|\!|(N_{e},U)|\!|\!|_{\epsilon}^{2}.
\end{split}
\end{equation}
Adding the estimates \eqref{s1} and \eqref{s2}, we have
\begin{equation*}
\begin{split}
F_{8}\leq C_{1}(1+\epsilon^{2}|\!|\!|(N_{e},U)|\!|\!|_{\epsilon}^{5})|\!|\!|(N_{e},U)|\!|\!|_{\epsilon}^{2}.
\end{split}
\end{equation*}
The estimates of $F_{4}\sim F_{6}$ are similar to $F_{8}$ and can be bounded by
\begin{equation*}
\begin{split}
F_{4\sim6}\leq C_{1}(1+\epsilon^{2}|\!|\!|(N_{e},U)|\!|\!|_{\epsilon}^{4})|\!|\!|(N_{e},U)|\!|\!|_{\epsilon}^{2}.
\end{split}
\end{equation*}

\emph{Estimate of $F_{9}$.} Using commutator notation \eqref{w}, we have
\begin{equation*}
\begin{split}
F_{9}=\frac{\epsilon^{2}H^{2}}{2}\int
\Big\{[\partial_{x}^{3},\frac{1}{n_{e}^{2}}]\partial_{x}N_{e}\partial_{x}^{2}N_{e}
+\frac{1}{n_{e}^{2}}\partial_{x}^{3}(\partial_{x}N_{e}\partial_{x}^{2}N_{e})\Big\}\partial_{x}^{3}U
=:F_{91}+F_{92}.
\end{split}
\end{equation*}
By commutator estimates \eqref{w11} of Lemma \ref{L9}, we have
\begin{equation*}
\begin{split}
\|[\partial_{x}^{3},\frac{1}{n_{e}^{2}}]\partial_{x}N_{e}\partial_{x}^{2}N_{e}\|
\leq\|\partial_{x}(\frac{1}{n_{e}^{2}})\|_{L^{\infty}}\|\partial_{x}^{2}\left(\partial_{x}N_{e}\partial_{x}^{2}N_{e}\right)\|
+\|\partial_{x}^{3}(\frac{1}{n_{e}^{2}})\|\|\partial_{x}N_{e}\partial_{x}^{2}N_{e}\|_{L^{\infty}}.
\end{split}
\end{equation*}
By direct computation, we note that $\|\partial_{x}(\frac{1}{n_{e}^{2}})\|_{L^{\infty}}$, $|\partial_{x}^{3}(\frac{1}{n_{e}^{2}})|$ and $\|\partial_{x}^{3}(\frac{1}{n_{e}^{2}})\|$ have similar estimates to \eqref{use 1}, \eqref{use 2} and \eqref{use 21}. Hence we have
\begin{equation}
\begin{split}\label{use 9'}
\|\partial_{x}^{2}(\partial_{x}N_{e}\partial_{x}^{2}N_{e})\|
\leq C(\|\partial_{x}N_{e}\|_{L^{\infty}}\|\partial_{x}^{4}N_{e}\|+\|\partial_{x}^{2}N_{e}\|_{L^{\infty}}\|\partial_{x}^{3}N_{e}\|).
\end{split}
\end{equation}
Therefore, by H\"older inequality and Sobolev embedding theorem, we obtain
\begin{equation}
\begin{split}\label{s3}
F_{91}
\leq C_{1}(1+\epsilon^{2}|\!|\!|(N_{e},U)|\!|\!|_{\epsilon}^{4})|\!|\!|(N_{e},U)|\!|\!|_{\epsilon}^{2},
\end{split}
\end{equation}
where we have used \eqref{use 1}, \eqref{use 2} and \eqref{use 21}. On the other hand,
\begin{equation*}
\begin{split}\label{use 10'}
\|\partial_{x}^{3}(\partial_{x}N_{e} \partial_{x}^{2}N_{e})\| \leq C(\|\partial_{x}N_{e}\|_{L^{\infty}} \|\partial_{x}^{5}N_{e}\| +\|\partial_{x}^{2}N_{e}\|_{L^{\infty}} \|\partial_{x}^{4}N_{e}\| +\|\partial_{x}^{3}N_{e}\|_{L^{\infty}} \|\partial_{x}^{3}N_{e}\|).
\end{split}
\end{equation*}
Therefore, by applying H\"older inequality again, we have
\begin{equation}
\begin{split}\label{s4}
F_{92}
\leq C_{1}(1+\epsilon^{2}|\!|\!|(N_{e},U)|\!|\!|_{\epsilon}^{2})|\!|\!|(N_{e},U)|\!|\!|_{\epsilon}^{2}.
\end{split}
\end{equation}
Adding the estimates \eqref{s3} and \eqref{s4}, we have
\begin{equation*}
\begin{split}
F_{9}\leq C_{1}(1+\epsilon^{2}|\!|\!|(N_{e},U)|\!|\!|_{\epsilon}^{4})|\!|\!|(N_{e},U)|\!|\!|_{\epsilon}^{2}.
\end{split}
\end{equation*}
$F_{7}$ is similar to $F_{9}$.
From equation \eqref{equ28} and \eqref{equ29} in Lemma \ref{L8}, we can obtain
\begin{equation*}
\begin{split}
&F_{10}\leq C_{1}(1+\epsilon^{2}|\!|\!|(N_{e},U)|\!|\!|_{\epsilon}^{2})|\!|\!|(N_{e},U)|\!|\!|_{\epsilon}^{2},\\
&F_{11}\leq C_{1}(1+\epsilon|\!|\!|(N_{e},U)|\!|\!|).
\end{split}
\end{equation*}

\emph{Estimate of $F_{1}+F_{2}$.} By direct computation, we have

\begin{equation*}
\begin{split}
F_{1}+F_{2}=&\int\left\{\partial_{x}^{2}(n_{e}\partial_{x}N_{e})-\frac{\epsilon H^{2}}{4}\partial_{x}^{2}\left(\frac{\partial_{x}^{3}N_{e}}{n_{e}}\right)\right\}\partial_{x}^{4}U\\
=&\int\left(n_{e}\partial_{x}^{3}N_{e}-\frac{\epsilon H^{2}}{4}\frac{\partial_{x}^{5}N_{e}}{n_{e}}\right)\partial_{x}^{4}U
+\int\left(\sum_{\alpha=1}^{2}C_{2}^{\alpha}\partial_{x}^{\alpha}n_{e}\partial_{x}^{3-\alpha}N_{e}\right)\partial_{x}^{4}U\\
&-\frac{\epsilon H^{2}}{4}\int\left(\sum_{\beta=1}^{2}C_{2}^{\beta} \partial_{x}^{\beta}\left(\frac{1}{n_{e}}\right)\partial_{x}^{5-\beta}N_{e}\right)\partial_{x}^{4}U\\
=&:\sum_{i=1}^{3}K_{i}.
\end{split}
\end{equation*}

\emph{Estimates of $K_{2}$ and $K_{3}$.} By integration by parts, we have
\begin{equation*}
\begin{split}
K_{2}=-\int\big(\sum_{\alpha=1}^{2}C_{2}^{\alpha}\partial_{x}^{\alpha+1}n_{e}\partial_{x}^{3-\alpha}N_{e}
+\sum_{\alpha=1}^{2}C_{2}^{\alpha}\partial_{x}^{\alpha}n_{e}\partial_{x}^{4-\alpha}N_{e}\big)\partial_{x}^{3}U.
\end{split}
\end{equation*}

\begin{equation*}
\begin{split}
K_{3}=\frac{\epsilon H^{2}}{4}\int\big(\sum_{\beta=1}^{2}C_{2}^{\beta}\partial_{x}^{\beta+1}(\frac{1}{n_{e}})\partial_{x}^{5-\beta}N_{e}
+\sum_{\beta=1}^{2}C_{2}^{\beta}\partial_{x}^{\beta}\frac{1}{n_{e}}\partial_{x}^{6-\beta}N_{e}\big)\partial_{x}^{3}U.
\end{split}
\end{equation*}
For $K_{2}$, we have
\begin{equation*}
\begin{split}
K_{2}\leq C_{1}(1+\epsilon^{2}|\!|\!|(N_{e},U)|\!|\!|_{\epsilon}^{2})|\!|\!|(N_{e},U)|\!|\!|_{\epsilon}^{2}.
\end{split}
\end{equation*}
Combining \eqref{one order},\ \eqref{two orders}\ and \eqref{use 2}, we obtain
\begin{equation*}
\begin{split}
K_{3}\leq C_{1}(1+\epsilon^{2}|\!|\!|(N_{e}, U)|\!|\!|_{\epsilon}^{3})|\!|\!|(N_{e}, U)|\!|\!|_{\epsilon}^{2},
\end{split}
\end{equation*}
where we have used H\"older inequality and Sobolev embedding theorem.

\emph{Estimates of $K_{1}$.} By $\eqref{rem-1}$, we have
\begin{equation*}
\begin{split}
\partial_{x}^{4}U =&\frac{1}{n_{i}}\Big\{\partial_{x}^{3}((1-u_{i})\partial_{x}N_{i})-\epsilon\partial_{t}\partial_{x}^{3}N_{i}\\
&-\sum_{\beta=1}^{3}C_{3}^{\beta}\partial_{x}^{\beta}n_{i}\partial_{x}^{4-\beta}U
-\epsilon\partial_{x}^{3}(\partial_{x}\tilde{n_{i}}U)-\epsilon\partial_{x}^{3}(\partial_{x}\tilde{u_{i}}N_{i})
-\epsilon^{2}\partial_{x}^{3}\mathcal{R}_{1}\Big\}\\
=&:\sum_{i=1}^{6}E_{i} \ .
\end{split}
\end{equation*}
Accordingly, $K_{1}$ is decomposed into
\begin{equation}\label{biaohao}
\begin{split}
K_{1}=\sum_{i=1}^{6}\int(n_{e}\partial_{x}^{3}N_{e}-\frac{\epsilon H^{2}}{4}\frac{1}{n_{e}}\partial_{x}^{5}N_{e})E_{i}=:\sum_{i=1}^{6}K_{1i}.
\end{split}
\end{equation}
We first estimate the terms $K_{1i}$ for $3\leq i\leq6$ and leave $K_{11}$ and $K_{12}$ in the next two subsections. By Lemma \ref{Lem-u1} and Lemma \ref{Lem-u2} in the next two subsections, we have
\begin{equation*}
\begin{split}
K_{11}\leq C_{1}(1+\epsilon^{2}|\!|\!|(N_{e},U)|\!|\!|_{\epsilon}^{6})(1+|\!|\!|(N_{e},U)|\!|\!|_{\epsilon}^{2}),\\
\end{split}
\end{equation*}
\begin{equation*}
\begin{split}
K_{12}\leq &-\frac\epsilon2\frac{d}{dt}\left[\int\frac{n_{e}}{n_{i}}(\partial_{x}^{3}N_{e})^{2}
+\epsilon\int\frac{n_{e}^{2}}{n_{i}}(\partial_{x}^{4}N_{e})^{2}
+\frac{\epsilon^{2}H^{2}}{4}\int\frac{1}{n_{i}}(\partial_{x}^{5}N_{e})^{2}\right]\\
&-\frac\epsilon2\frac{\epsilon H^{2}}{4}\frac{d}{dt}\left[\int\frac{1}{n_{e}n_{i}}(\partial_{x}^{4}N_{e})^{2}
+\epsilon\int\frac{1}{n_{i}}(\partial_{x}^{5}N_{e})^{2}
+\frac{\epsilon^{2}H^{2}}{4}\int\frac{1}{(n_{e})^{2}n_{i}}(\partial_{x}^{6}N_{e})^{2}\right]\\
&+C_1(1+\epsilon^2|\!|\!|(N_{e},U)|\!|\!|^{6}_{\epsilon})(1+|\!|\!|(N_{e},U)|\!|\!|^2_{\epsilon}).
\end{split}
\end{equation*}

\emph{Estimate of $K_{13}$.} It can be decomposed that
\begin{equation*}
\begin{split}
K_{13}=\int\left(\frac{n_{e}}{n_{i}}\partial_{x}^{3}N_{e}-\frac{\epsilon H^{2}}{4}\frac{1}{n_{e}n_{i}}\partial_{x}^{5}N_{e}\right)
\sum_{\beta=1}^{3}C_{3}^{\beta}\partial_{x}^{\beta}n_{i}\partial_{x}^{4-\beta}U.
\end{split}
\end{equation*}
When $\beta=1,2$, $K_{13}$ can be easily bounded by $C_{1}(1+\epsilon^{2}|\!|\!|(N_{e},U)|\!|\!|_{\epsilon}^{2})|\!|\!|(N_{e},U)|\!|\!|_{\epsilon}^{2}$ by H\"older inequality, Cauchy inequality and Lemma \ref{L1}.
When $\beta=3$, by integration by parts, we have
\begin{equation*}
\begin{split}
K_{13}=&-\int\big(\partial_{x}(\frac{n_{e}}{n_{i}})\partial_{x}^{3}N_{e}
-\frac{\epsilon H^{2}}{4}\partial_{x}(\frac{1}{n_{e}n_{i}})\partial_{x}^{5}N_{e}\big)
\partial_{x}^{2}n_{i}\partial_{x}U\\
&-\int\big(\frac{n_{e}}{n_{i}}\partial_{x}^{4}N_{e}
-\frac{\epsilon H^{2}}{4}\frac{1}{n_{e}n_{i}}\partial_{x}^{6}N_{e}\big)
\partial_{x}^{2}n_{i}\partial_{x}U\\
&-\int\big(\frac{n_{e}}{n_{i}}\partial_{x}^{3}N_{e}
-\frac{\epsilon H^{2}}{4}\frac{1}{n_{e}n_{i}}\partial_{x}^{5}N_{e}\big)
\partial_{x}^{2}n_{i}\partial_{x}^{2}U\\
=&:K_{131}+K_{132}+K_{133}.
\end{split}
\end{equation*}

By direct computation, we have
\begin{equation}
\begin{split}\label{use 11}
\left|\partial_{x}\left(\frac{1}{n_{e}n_{i}}\right)\right|, \ \  \left|\partial_{x}\left(\frac{n_{e}}{n_{i}}\right)\right|\leq C\left(\epsilon+\epsilon^{3}(|\partial_{x}N_{e}| +|\partial_{x}N_{i}|)\right).
\end{split}
\end{equation}

Therefore, by H\"older inequality, Sobolev embedding $H^1\hookrightarrow L^{\infty}$ and Lemma \ref{L1}, we have
\begin{equation*}
\begin{split}
K_{131}\leq C_{1}&\left(1+\epsilon^{2}(\|\partial_{x}N_{e}\|_{L^{\infty}}^{2}+\|\partial_{x}N_{i}\|_{L^{\infty}}^{2}
+\|\partial_{x}U\|_{L^{\infty}}^{2})\right) \left(\epsilon\|\partial_{x}^{3}N_{e}\|^{2} +\|\partial_{x}^{2}N_{i}\|^{2} +\epsilon^{3}\|\partial_{x}^{5}N_{e}\|^{2}\right)\\
\leq C_{1}&\big(1+\epsilon^{2}(\|N_{e}\|_{H^{2}}^{2}+\|N_{i}\|_{H^{2}}^{2}
+\|U\|_{H^{2}}^{2})\big)(\epsilon\|\partial_{x}^{3}N_{e}\|^{2}+\|\partial_{x}^{2}N_{i}\|^{2}+\epsilon^{3}\|\partial_{x}^{5}N_{e}\|^{2})\\
\leq C_{1}&(1+\epsilon^{2}|\!|\!|(N_{e},U)|\!|\!|_{\epsilon}^{2})|\!|\!|(N_{e},U)|\!|\!|_{\epsilon}^{2}.
\end{split}
\end{equation*}
On the other hand, by H\"older inequality and Lemma \ref{L1}, $K_{132}\ and\ K_{133}$ can be bounded by $C_{1}(1+\epsilon^{2}|\!|\!|(N_{e}, U)|\!|\!|_{\epsilon}^{2})|\!|\!|(N_{e},U)|\!|\!|_{\epsilon}^{2}$. Thus $K_{13}$ is bounded by $C_{1}(1+\epsilon^{2}|\!|\!|(N_{e}, U)|\!|\!|_{\epsilon}^{2})|\!|\!|(N_{e}, U)|\!|\!|_{\epsilon}^{2}$.

\emph{Estimate of $K_{14}$.} By H\"older inequality and Lemma \ref{L1}, $K_{14}$ can be bounded easily by $C_{1}|\!|\!|(N_{e},U)|\!|\!|_{\epsilon}^{2}$.

\emph{Estimate of $K_{15}$.} By applying integration by parts and \eqref{use 11}, we have
\begin{equation*}
\begin{split}
K_{15}=&-\epsilon\int\big(\partial_{x}(\frac{n_{e}}{n_{i}})\partial_{x}^{3}N_{e}
-\frac{\epsilon H^{2}}{4}\partial_{x}(\frac{1}{n_{e}n_{i}})\partial_{x}^{5}N_{e}\big)
\partial_{x}^{2}(\partial_{x}\tilde{u_{i}}N_{i})\\
&-\epsilon\int\big(\frac{n_{e}}{n_{i}}\partial_{x}^{4}N_{e}
-\frac{\epsilon H^{2}}{4}\frac{1}{n_{e}n_{i}}\partial_{x}^{6}N_{e}\big)
\partial_{x}^{2}(\partial_{x}\tilde{u_{i}}N_{i})\\
\leq &C_{1}(1+\epsilon^{2}|\!|\!|(N_{e}, U)|\!|\!|_{\epsilon}^{2})|\!|\!|(N_{e}, U)|\!|\!|_{\epsilon}^{2},
\end{split}
\end{equation*}
where we have used H\"older inequality and the Lemma \ref{L1}.

\emph{Estimate of $K_{16}$.} Since $\mathcal{R}_{1}$ is known, thus we have
\begin{equation*}
\begin{split}
K_{16}\leq C_{1}(\epsilon\|\partial_{x}^{3}N_{e}\|+\epsilon^{3}\|\partial_{x}^{5}N_{e}\|).
\end{split}
\end{equation*}

Summarizing all the estimates, we complete the proof of Proposition \ref{P3}.
\end{proof}

\subsection{Estimate of $K_{11}$}\label{2.4}
Next we estimate $K_{11}$ in \eqref{biaohao}.
\begin{lemma}[\textbf{\emph{Estimate of $K_{11}$}}]\label{Lem-u1}
Let $(N_{i},N_{e},U)$ be a solution to \eqref{rem}£¬ then
\begin{equation*}
\begin{split}
K_{11}
\leq C_{1}(1+\epsilon^{2}|\!|\!|(N_{e},U)|\!|\!|_{\epsilon}^{6})(1+|\!|\!|(N_{e},U)|\!|\!|_{\epsilon}^{2}).
\end{split}
\end{equation*}
\end{lemma}

\begin{proof}
Recall that in \eqref{biaohao},
\begin{equation*}
\begin{split}
K_{11}&=\int(\frac{n_{e}}{n_{i}}\partial_{x}^{3}N_{e}-\frac{\epsilon H^{2}}{4}\frac{1}{n_{e}n_{i}}\partial_{x}^{5}N_{e})\partial_{x}^{3}\big((1-u_{i})\partial_{x}N_{i})\big)\\
&=\int(\frac{n_{e}}{n_{i}}\partial_{x}^{3}N_{e}-\frac{\epsilon H^{2}}{4}\frac{1}{n_{e}n_{i}}\partial_{x}^{5}N_{e})
\sum_{\gamma=0}^{3}C_{3}^{\gamma}\partial_{x}^{3-\gamma}(1-u_{i})\partial_{x}^{\gamma+1}N_{i}.
\end{split}
\end{equation*}
When $\gamma=0,1$, by H\"older inequality, Sobolev embedding $H^1\hookrightarrow L^{\infty}$ and Lemma \ref{L1},
\begin{equation*}
\begin{split}
K_{11}|_{\gamma=0,1}&\leq C_{1}(1+\epsilon^{2}(\epsilon\|\partial_{x}^{3}U\|^{2}))
(\|N_{i}\|_{H^{2}}^{2}+\epsilon\|\partial_{x}^{3}N_{e}\|^{2}+\epsilon^{3}\|\partial_{x}^{5}N_{e}\|^{2})\\
&\leq C_{1}(1+\epsilon^{2}|\!|\!|(N_{e},U)|\!|\!|_{\epsilon}^{2})(1+|\!|\!|(N_{e},U)|\!|\!|_{\epsilon}^{2}).
\end{split}
\end{equation*}

By integration by parts for $\gamma=2$ and \eqref{use 11}, we have
\begin{equation*}
\begin{split}
K_{11}|_{\gamma=2}=&3\int(\partial_{x}(\frac{n_{e}}{n_{i}})\partial_{x}^{3}N_{e}-\frac{\epsilon H^{2}}{4}\partial_{x}(\frac{1}{n_{e}n_{i}})\partial_{x}^{5}N_{e})\partial_{x}u_{i}\partial_{x}^{2}N_{i}\\
&+3\int(\frac{n_{e}}{n_{i}}\partial_{x}^{4}N_{e}-\frac{\epsilon H^{2}}{4}\frac{1}{n_{e}n_{i}}\partial_{x}^{6}N_{e})\partial_{x}u_{i}\partial_{x}^{2}N_{i}\\
&+3\int(\frac{n_{e}}{n_{i}}\partial_{x}^{3}N_{e}-\frac{\epsilon H^{2}}{4}\frac{1}{n_{e}n_{i}}\partial_{x}^{5}N_{e})\partial_{x}^{2}u_{i}\partial_{x}^{2}N_{i}\\
\leq &C(1+\epsilon^{2}|\!|\!|(N_{e}, U)|\!|\!|_{\epsilon}^{2})|\!|\!|(N_{e}, U)|\!|\!|_{\epsilon}^{2},
\end{split}
\end{equation*}
where we have used H\"older inequality and Lemma \ref{L1}.

In the following we estimate $K_{11}$ for $\gamma=3$, by \eqref{rem-3}, we have
\begin{equation*}
\begin{split}
\partial_{x}^{4}N_{i}=&\partial_{x}^{4}N_{e}-\epsilon\partial_{x}^{4}(n_{e}\partial_{x}^{2}N_{e})
-2\epsilon^{2}\partial_{x}^{4}\big(\partial_{x}\tilde{n_{e}}\partial_{x}N_{e}\big)
-\epsilon^{4}\partial_{x}^{4}\big((\partial_{x}N_{e})^{2}\big)
-\epsilon^{2}\partial_{x}^{4}\big(\partial_{x}^{2}\tilde{n_{e}}N_{e}\big)\\
&-\epsilon^{2}\partial_{x}^{4}R_{3}^{1}
+\frac{H^{2}}{4}\Big[-12\epsilon^{5}\partial_{x}^{4}\big(\frac{(\partial_{x}\tilde{n_{e}})^{3}}{n_{e}^{4}}\partial_{x}N_{e}\big)
+14\epsilon^{4}\partial_{x}^{4}\big(\frac{\partial_{x}\tilde{n_{e}}\partial_{x}^{2}\tilde{n_{e}}}{n_{e}^{3}}\partial_{x}N_{e}\big)\\
&-3\epsilon^{3}\partial_{x}^{4}\big(\frac{\partial_{x}^{3}\tilde{n_{e}}}{n_{e}^{2}}\partial_{x}N_{e}\big)
-18\epsilon^{7}\partial_{x}^{4}\big(\frac{(\partial_{x}\tilde{n_{e}})^{2}}{n_{e}^{4}}(\partial_{x}N_{e})^{2}\big)
+7\epsilon^{4}\partial_{x}^{4}\big(\frac{(\partial_{x}\tilde{n_{e}})^{2}}{n_{e}^{3}}\partial_{x}^{2}N_{e}\big)\\
&+7\epsilon^{6}\partial_{x}^{4}\big(\frac{\partial_{x}^{2}\tilde{n_{e}}}{n_{e}^{3}}(\partial_{x}N_{e})^{2}\big)
-4\epsilon^{3}\partial_{x}^{4}\big(\frac{\partial_{x}^{2}\tilde{n_{e}}}{n_{e}^{2}}\partial_{x}^{2}N_{e}\big)
-12\epsilon^{9}\partial_{x}^{4}\big(\frac{\partial_{x}\tilde{n_{e}}}{n_{e}^{4}}(\partial_{x}N_{e})^{3}\big)\\
&+14\epsilon^{6}\partial_{x}^{4}\big(\frac{\partial_{x}\tilde{n_{e}}}{n_{e}^{3}}\partial_{x}N_{e}\partial_{x}^{2}N_{e}\big)
-3\epsilon^{3}\partial_{x}^{4}\big(\frac{\partial_{x}\tilde{n_{e}}}{n_{e}^{2}}\partial_{x}^{3}N_{e}\big)
-3\epsilon^{11}\partial_{x}^{4}\big(\frac{1}{n_{e}^{4}}(\partial_{x}N_{e})^{4}\big)\\
&+7\epsilon^{8}\partial_{x}^{4}\big(\frac{1}{n_{e}^{3}}(\partial_{x}N_{e})^{2}\partial_{x}^{2}N_{e}\big)
-2\epsilon^{5}\partial_{x}^{4}\big(\frac{1}{n_{e}^{2}}(\partial_{x}^{2}N_{e})^{2}\big)
-3\epsilon^{5}\partial_{x}^{4}\big(\frac{1}{n_{e}^{2}}\partial_{x}N_{e}\partial_{x}^{3}N_{e}\big)\\
&+\epsilon^{2}\partial_{x}^{4}\big(\frac{\partial_{x}^{4}N_{e}}{n_{e}}\big)
+\epsilon^{2}\partial_{x}^{4}\big(\frac{R_{3}^{2}+R_{3}^{3}}{n_{e}^{4}}\big)\Big]\\
=:&\sum_{i=1}^{22}F_{i}\ .
\end{split}
\end{equation*}
Accordingly, $K_{11}|_{\gamma=3}$ is decomposed into
\begin{equation*}
\begin{split}
K_{11}|_{\gamma=3}=\sum_{i=1}^{22}\int\big(\frac{n_{e}}{n_{i}}\partial_{x}^{3}N_{e}
-\frac{\epsilon H^{2}}{4}\frac{1}{n_{e}n_{i}}\partial_{x}^{5}N_{e}\big)(1-u_{i})F_{i}
=:\sum_{i=1}^{22}J_{i}\ .
\end{split}
\end{equation*}

\emph{Estimate of $J_{1}$.} By integration by parts, we have
\begin{equation*}
\begin{split}
J_{1}=-\frac{1}{2}\int(\partial_{x}^{3}N_{e})^{2}\partial_{x}\big(\frac{n_{e}(1-u_{i})}{n_{i}}\big)
+\frac{\epsilon H^{2}}{8}\int(\partial_{x}^{4}N_{e})^{2}\partial_{x}(\frac{1-u_{i}}{n_{e}n_{i}}).
\end{split}
\end{equation*}
By direct computation, we have
\begin{equation}
\begin{split}\label{use 12}
\left|\partial_{x}\frac{1-u_{i}}{n_{e}n_{i}}\right|,\ \  \left|\partial_{x}\frac{n_{e}(1-u_{i})}{n_{i}}\right|\leq C\left(\epsilon+\epsilon^{3}(|\partial_{x}N_{e}| +|\partial_{x}N_{i}|+|\partial_{x}U|)\right).
\end{split}
\end{equation}
Hence by H\"older inequality, Sobolev embedding $H^1\hookrightarrow L^{\infty}$ and Lemma \ref{L1}, we have
\begin{equation*}
\begin{split}
J_{1}\leq &C(1+\epsilon^{2}(\|N_{e}\|_{H^{2}}^{2} +\|N_{i}\|_{H^{2}}^{2}+\|U\|_{H^{2}}^{2})) (\epsilon\|\partial_{x}^{3}N_{e}\|^{2} +\epsilon^{2}\|\partial_{x}^{4}N_{e}\|^{2})\\
\leq &C((1+\epsilon^{2}|\!|\!|(N_{e}, U)|\!|\!|_{\epsilon}^{2})|\!|\!|(N_{e}, U)|\!|\!|_{\epsilon}^{2}).
\end{split}
\end{equation*}

\emph{Estimate of $J_{2}$.} By integration by parts, we have
\begin{equation*}
\begin{split}
J_{2}=&-\epsilon\int\Big[\big(\partial_{x}\frac{n_{e}(1-u_{i})}{n_{i}}\big)\partial_{x}^{3}N_{e}
-\frac{\epsilon H^{2}}{4}\big(\partial_{x}\frac{1-u_{i}}{n_{e}n_{i}}\big)
\partial_{x}^{5}N_{e}\Big]\sum_{\alpha=0}^{3}C_{\alpha}^{3}\partial_{x}^{\alpha}n_{e}\partial_{x}^{5-\alpha}N_{e}\\
&-\epsilon\int\Big[\frac{n_{e}(1-u_{i})}{n_{i}}\partial_{x}^{4}N_{e}
-\frac{\epsilon H^{2}}{4}\frac{1-u_{i}}{n_{e}n_{i}}
\partial_{x}^{6}N_{e}\Big]\sum_{\beta=0}^{3}C_{\beta}^{3}\partial_{x}^{\beta}n_{e}\partial_{x}^{5-\beta}N_{e}\\
=&:J_{21}+J_{22}.
\end{split}
\end{equation*}
Using the equation \eqref{use 12}, and by H\"older inequality and Sobolev embedding theorem,
\begin{equation*}
\begin{split}
J_{21}
\leq C((1+\epsilon^{2}|\!|\!|(N_{e},U)|\!|\!|_{\epsilon}^{2})|\!|\!|(N_{e},U)|\!|\!|_{\epsilon}^{2}).
\end{split}
\end{equation*}
When $\beta=0,1,2$, $J_{22}$ can be easily estimated by $C(1+\epsilon^{2}|\!|\!|(N_{e},U)|\!|\!|_{\epsilon}^{2})|\!|\!|(N_{e},U)|\!|\!|_{\epsilon}^{2}$. When $\beta=3$, by integration by parts, we have
\begin{equation*}
\begin{split}
J_{22}=&\frac{\epsilon}{2}\int\Big[\big(\partial_{x}\frac{n_{e}^{2}(1-u_{i})}{n_{i}}\big)(\partial_{x}^{4}N_{e})^{2}
-\frac{\epsilon H^{2}}{4}\big(\partial_{x}\frac{1-u_{i}}{n_{i}}\big)
(\partial_{x}^{5}N_{e})^{2}\Big].
\end{split}
\end{equation*}
Similar to \eqref{use 12}, we have
\begin{equation}
\begin{split}\label{use131}
\left|\partial_{x}\frac{n_{e}^{2}(1-u_{i})}{n_{i}}\right|\leq & C\left(\epsilon+\epsilon^{3}(|\partial_{x}N_{e}|+|\partial_{x}N_{i}|+|\partial_{x}U|)\right),\\
\left|\partial_{x}\frac{1-u_{i}}{n_{i}}\right|\leq & C\left(\epsilon+\epsilon^{3}(|\partial_{x}N_{i}|+|\partial_{x}U|)\right).
\end{split}
\end{equation}
Therefore, $J_{22}$ can be estimated by $C(1+\epsilon^{2}|\!|\!|(N_{e}, U)|\!|\!|_{\epsilon}^{2})|\!|\!|(N_{e}, U)|\!|\!|_{\epsilon}^{2}$.
As a result, $J_{2}$ can be estimated by $C(1+\epsilon^{2}|\!|\!|(N_{e}, U)|\!|\!|_{\epsilon}^{2})|\!|\!|(N_{e}, U)|\!|\!|_{\epsilon}^{2}$.
$J_{3}\sim J_{5}$ can be also estimated by
$C(1+\epsilon^{2}|\!|\!|(N_{e}, U)|\!|\!|_{\epsilon}^{2})|\!|\!|(N_{e}, U)|\!|\!|_{\epsilon}^{2}$.

\emph{Estimate of $J_{17}$.} Using commutator notation \eqref{w}, we have
\begin{equation*}
\begin{split}
J_{17}=&-\frac{3\epsilon^{11}H^{2}}{4}\int\big(\frac{n_{e}(1-u_{i})}{n_{i}}\partial_{x}^{3}N_{e}-\frac{\epsilon H^{2}}{4}\frac{1-u_{i}}{n_{e}n_{i}}\partial_{x}^{5}N_{e}\big)\\
&\times\left\{[\partial_{x}^{4},\frac{1}{n_{e}^{4}}](\partial_{x}N_{e})^{4}
+\frac{1}{n_{e}^{4}}\partial_{x}^{4}(\partial_{x}N_{e})^{4}\right\}=:J_{171}+J_{172}.
\end{split}
\end{equation*}
By commutator estimates \eqref{w11}, we have
\begin{equation*}
\begin{split}
\left\|\left[\partial_{x}^{4}, \frac{1}{n_{e}^{4}}\right](\partial_{x}N_{e})^{4} \right\|
\leq \left\|\partial_{x}\left(\frac{1}{n_{e}^{4}}\right)\right\|_{L^{\infty}}\|\partial_{x}^{3}(\partial_{x}N_{e})^{4}\|
+\left\|\partial_{x}^4\left(\frac{1}{n_{e}^{4}}\right)\right\|\|\partial_{x}N_{e}\|_{L^{\infty}}^{4}.
\end{split}
\end{equation*}
By direct computation, we have
\begin{equation}
\begin{split}\label{use211}
\left\|\partial_{x}^{4}\left(\frac{1}{n_{e}^{4}}\right)\right\|
\leq &C\Big(1+\epsilon^{3}(\|\partial_{x}N_{e}\| +\|\partial_{x}^{2}N_{e}\|+\|\partial_{x}^{3}N_{e}\| +\|\partial_{x}^{4}N_{e}\|)\\
&+\epsilon^{6}(\|\partial_{x}N_{e}\|_{L^{\infty}}\|\partial_{x}N_{e}\|_{H^2} +\|\partial_{x}^{2}N_{e}\|_{L^{\infty}}\|\partial_{x}^{2}N_{e}\|)\\
&+\epsilon^{9}(\|\partial_{x}N_{e}\|_{L^{\infty}}^{2}\|\partial_{x}^{2}N_{e}\|
+\|\partial_{x}N_{e}\|_{L^{\infty}}^{2}\|\partial_{x}N_{e}\|)\\
&+\epsilon^{12}\|\partial_{x}N_{e}\|\|\partial_{x}N_{e}\|_{L^{\infty}}^{3}\Big),
\end{split}
\end{equation}
and
\begin{equation}
\begin{split}\label{use 17}
\|\partial_{x}^{3}(\partial_{x}N_{e})^{4}\|
\leq &C(\|\partial_{x}N_{e}\|_{L^{\infty}}^{2}\|\partial_{x}^{2}N_{e}\|_{L^{\infty}}\|\partial_{x}^{3}N_{e}\|\\
&+\|\partial_{x}N_{e}\|_{L^{\infty}}\|\partial_{x}^{2}N_{e}\|_{L^{\infty}}^{2}\|\partial_{x}^{2}N_{e}\|
+\|\partial_{x}N_{e}\|_{L^{\infty}}^{3}\|\partial_{x}^{4}N_{e}\|).
\end{split}
\end{equation}
Therefore, using \eqref{use 1}, \eqref{use211} and \eqref{use 17}, by H\"older inequality and Sobolev embedding theorem, we obtain
\begin{equation}
\begin{split}\label{s5}
J_{171}
\leq C_{1}(1+\epsilon^{2}|\!|\!|(N_{e},U)|\!|\!|_{\epsilon}^{8})|\!|\!|(N_{e},U)|\!|\!|_{\epsilon}^{2}.
\end{split}
\end{equation}
On the other hand, by direct computation, we have
\begin{equation}
\begin{split}\label{use 18}
\|\partial_{x}^{4}(\partial_{x}N_{e})^{4}\|
\leq &C\big(\|\partial_{x}^{2}N_{e}\|_{L^{\infty}}^{3}\|\partial_{x}^{2}N_{e}\|
+\|\partial_{x}N_{e}\|_{L^{\infty}}\|\partial_{x}^{2}N_{e}\|_{L^{\infty}}^{2}\|\partial_{x}^{3}N_{e}\|\\
&+\|\partial_{x}N_{e}\|\|\partial_{x}^{3}N_{e}\|_{L^{\infty}}^{2}
+\|\partial_{x}N_{e}\|_{L^{\infty}}^{2}\|\partial_{x}^{2}N_{e}\|_{L^{\infty}}\|\partial_{x}^{4}N_{e}\|\\
&+\|\partial_{x}N_{e}\|_{L^{\infty}}^{3}\|\partial_{x}^{5}N_{e}\|\big).
\end{split}
\end{equation}
Therefore, by applying H\"older inequality again, we have
\begin{equation}
\begin{split}\label{s6}
J_{172}
\leq C_{1}(1+\epsilon^{2}|\!|\!|(N_{e},U)|\!|\!|_{\epsilon}^{4})|\!|\!|(N_{e},U)|\!|\!|_{\epsilon}^{2}.
\end{split}
\end{equation}
Adding the estimates \eqref{s5} and \eqref{s6}, we have
\begin{equation*}
\begin{split}
J_{17}\leq C_{1}(1+\epsilon^{2}|\!|\!|(N_{e},U)|\!|\!|_{\epsilon}^{8})|\!|\!|(N_{e},U)|\!|\!|_{\epsilon}^{2}.
\end{split}
\end{equation*}
$J_{7}\sim J_{10},\ J_{12},\ J_{14}$ are similar to $J_{17}$.

\emph{Estimate of $J_{19}$.} Using commutator notation \eqref{w}, we have
\begin{equation*}
\begin{split}
J_{19}=&\frac{1}{2}\epsilon^{5}H^{2}\int\left(\frac{n_{e}(1-u_{i})}{n_{i}}\partial_{x}^{3}N_{e}-\frac{\epsilon H^{2}}{4}\frac{1-u_{i}}{n_{e}n_{i}}\partial_{x}^{5}N_{e}\right)\\
&\times\left\{[\partial_{x}^{4},\frac{1}{n_{e}^{2}}](\partial_{x}^{2}N_{e})^{2}
+\frac{1}{n_{e}^{2}}\partial_{x}^{4}\big((\partial_{x}^{2}N_{e})^{2}\big)\right\} =:J_{191}+J_{192}.
\end{split}
\end{equation*}
By commutator estimates \eqref{w11}, we have
\begin{equation*}
\begin{split}
\|[\partial_{x}^{4},\frac{1}{n_{e}^{2}}](\partial_{x}^{2}N_{e})^{2}\|
\leq \|\partial_{x}\frac{1}{n_{e}^{2}}\|_{L^{\infty}}\|\partial_{x}^{3}(\partial_{x}^{2}N_{e})^{2}\|
+\|\partial_{x}^{4}\frac{1}{n_{e}^{2}}\|\|\partial_{x}^{2}N_{e}\|_{L^{\infty}}^{2}.
\end{split}
\end{equation*}
By direct computation, we have
\begin{equation}
\begin{split}\label{use222}
\|\partial_{x}^{3}(\partial_{x}^{2}N_{e})^{2}\|
\leq C(\|\partial_{x}^{3}N_{e}\|_{L^{\infty}}\|\partial_{x}^{4}N_{e}\|
+\|\partial_{x}^{2}N_{e}\|_{L^{\infty}}\|\partial_{x}^{5}N_{e}\|).
\end{split}
\end{equation}
Therefore, using \eqref{use 1}, \eqref{use211} and \eqref{use222}, by H\"older inequality and Sobolev embedding theorem, we can obtain
\begin{equation}
\begin{split}\label{s7}
J_{191}
\leq C_{1}(1+\epsilon^{2}|\!|\!|(N_{e},U)|\!|\!|_{\epsilon}^{6})|\!|\!|(N_{e},U)|\!|\!|_{\epsilon}^{2}.
\end{split}
\end{equation}
On the other hand, by direct computation, we have
\begin{equation}
\begin{split}\label{use 22}
\|\partial_{x}^{4}(\partial_{x}^{2}N_{e})^{2}\|
\leq C(\|\partial_{x}^{4}N_{e}\|_{L^{\infty}}\|\partial_{x}^{4}N_{e}\|
+\|\partial_{x}^{3}N_{e}\|_{L^{\infty}}\|\partial_{x}^{5}N_{e}\|+\|\partial_{x}^{2}N_{e}\|_{L^{\infty}}\|\partial_{x}^{6}N_{e}\|).
\end{split}
\end{equation}
Therefore, by applying H\"older inequality again, we have
\begin{equation}
\begin{split}\label{s8}
J_{192}
\leq C_{1}(1+\epsilon^{2}|\!|\!|(N_{e},U)|\!|\!|_{\epsilon}^{2})|\!|\!|(N_{e},U)|\!|\!|_{\epsilon}^{2}.
\end{split}
\end{equation}
Adding the estimates \eqref{s7} and \eqref{s8}, we have
\begin{equation*}
\begin{split}
J_{19}\leq C_{1}(1+\epsilon^{2}|\!|\!|(N_{e},U)|\!|\!|_{\epsilon}^{6})|\!|\!|(N_{e},U)|\!|\!|_{\epsilon}^{2}.
\end{split}
\end{equation*}
$J_{11},\ J_{13},\ J_{15},\ J_{18}$ are similar to $J_{19}$.

\emph{Estimate of $J_{21}$.}
\begin{equation*}
\begin{split}
J_{21}=\frac{\epsilon^{2} H^{2}}{4}\int(\frac{n_{e}}{n_{i}}\partial_{x}^{3}N_{e}
-\frac{\epsilon H^{2}}{4}\frac{1}{n_{e}n_{i}}\partial_{x}^{5}N_{e})(1-u_{i})\partial_{x}^{4}(\frac{\partial_{x}^{4}N_{e}}{n_{e}})
=:J_{211}+J_{212}\ .
\end{split}
\end{equation*}
By integration of parts twice and commutator notation \eqref{w}, we have

\begin{equation*}
\begin{split}
J_{211}=&\frac{\epsilon^{2} H^{2}}{4}\int\left(\frac{n_{e}(1-u_{i})}{n_{i}}\partial_{x}^{5}N_{e}
+2\partial_{x}(\frac{n_{e}(1-u_{i})}{n_{i}})\partial_{x}^{4}N_{e}
+\partial_{x}^{2}(\frac{n_{e}(1-u_{i})}{n_{i}})\partial_{x}^{3}N_{e}\right)\\
&\times\big([\partial_{x}^{2},\frac{1}{n_{e}}]\partial_{x}^{4}N_{e}+\frac{1}{n_{e}}\partial_{x}^{6}N_{e}\big).
\end{split}
\end{equation*}
By commutator estimates \eqref{w11} of Lemma \ref{L9}, we have
\begin{equation*}
\begin{split}
\|[\partial_{x}^{2},\frac{1}{n_{e}}]\partial_{x}^{4}N_{e}\|
\leq \|\partial_{x}\frac{1}{n_{e}}\|_{L^{\infty}}\|\partial_{x}^{5}N_{e}\|
+\|\partial_{x}^{2}\frac{1}{n_{e}}\|\|\partial_{x}^{4}N_{e}\|_{L^{\infty}}.
\end{split}
\end{equation*}
By direct computation, we have
\begin{equation}
\begin{split}\label{use223}
\|\partial_{x}^{2}(\frac{n_{e}(1-u_{i})}{n_{i}})\|
\leq &C\big(\epsilon+\epsilon^{3}(\|\partial_{x}N_{e}\|+\|\partial_{x}N_{i}\|+\|\partial_{x}U\|)\\
&+\epsilon^{6}(\|\partial_{x}N_{e}\|_{L^{\infty}}\|\partial_{x}N_{i}\|+\|\partial_{x}N_{i}\|_{L^{\infty}}\|\partial_{x}U\|
+\|\partial_{x}N_{e}\|_{L^{\infty}}\|\partial_{x}U\|\\
&+\|\partial_{x}N_{e}\|_{L^{\infty}}\|\partial_{x}N_{e}\|+\|\partial_{x}N_{i}\|\|\partial_{x}N_{i}\|_{L^{\infty}}
+\|\partial_{x}U\|_{L^{\infty}}\|\partial_{x}U\|)\big).
\end{split}
\end{equation}
Therefore, by applying H\"older inequality, Sobolev embedding theorem and Lemma \ref{L1},
\begin{equation*}
\begin{split}
J_{211}
\leq C_{1}(1+\epsilon^{2}|\!|\!|(N_{e},U)|\!|\!|_{\epsilon}^{4})|\!|\!|(N_{e},U)|\!|\!|_{\epsilon}^{2}.
\end{split}
\end{equation*}
By integration by parts and commutator notation \eqref{w}, we have
\begin{equation*}
\begin{split}
J_{212}
&=\frac{\epsilon^{3} H^{4}}{16}\int(\frac{1-u_{i}}{n_{e}n_{i}}
\partial_{x}^{6}N_{e}+\partial_{x}(\frac{1-u_{i}}{n_{e}n_{i}})\partial_{x}^{5}N_{e})
([\partial_{x}^{3},\frac{1}{n_{e}}]\partial_{x}^{4}N_{e}+\frac{1}{n_{e}}\partial_{x}^{7}N_{e})\\
&=:J_{2121}+J_{2122}.
\end{split}
\end{equation*}
Similar to $J_{211}$, using \eqref{use 1}\ and\ \eqref{use 21}, Sobolev embedding, Cauchy inequality and Lemma \ref{L1}, we have
\begin{equation}
\begin{split}\label{s9}
J_{2121}
\leq C_{1}(1+\epsilon^{2}|\!|\!|(N_{e},U)|\!|\!|_{\epsilon}^{4})|\!|\!|(N_{e},U)|\!|\!|_{\epsilon}^{2}.
\end{split}
\end{equation}
By integration by parts, we have
\begin{equation*}
\begin{split}
J_{2122}=&-\frac{\epsilon^{3} H^{4}}{16}\int\left(\frac{1}{2}\partial_{x}(\frac{1-u_{i}}{n_{e}^{2}n_{i}})
+\frac{1}{n_{e}}\partial_{x}(\frac{1}{n_{e}n_{i}})\right)(\partial_{x}^{6}N_{e})^{2}\\
&+\int\partial_{x}\left(\frac{1}{n_{e}}\partial_{x}(\frac{1-u_{i}}{n_{e}n_{i}})\right)\partial_{x}^{5}N_{e}\partial_{x}^{6}N_{e}.
\end{split}
\end{equation*}
Note that
\begin{equation}
\begin{split}\label{use224}
\|\partial_{x}\left(\frac{1}{n_{e}}\partial_{x}(\frac{1-u_{i}}{n_{e}n_{i}})\right)\|_{L^{\infty}}
\leq &C(\epsilon+\epsilon^{3}(\|\partial_{x}N_{e}\|_{L^{\infty}}+\|\partial_{x}N_{i}\|_{L^{\infty}}+\|\partial_{x}U\|_{L^{\infty}})\\
&+\epsilon^{6}(\|\partial_{x}N_{e}\|_{L^{\infty}}^{2}+\|\partial_{x}N_{i}\|_{L^{\infty}}^{2}+\|\partial_{x}U\|_{L^{\infty}}^{2}\\
&+\|\partial_{x}N_{e}\partial_{x}N_{i}\|_{L^{\infty}}
+\|\partial_{x}N_{e}\partial_{x}U\|_{L^{\infty}}+\|\partial_{x}N_{i}\partial_{x}U\|_{L^{\infty}})).
\end{split}
\end{equation}
Thus, by \eqref{use 11}, \eqref{use131}, H\"older inequality, Sobolev embedding theorem and Lemma \ref{L1},
\begin{equation}
\begin{split}\label{s10}
J_{2122}
\leq C(1+\epsilon^{2}|\!|\!|(N_{e},U)|\!|\!|_{\epsilon}^{2})|\!|\!|(N_{e},U)|\!|\!|_{\epsilon}^{2}.
\end{split}
\end{equation}
Adding the estimates \eqref{s9} and \eqref{s10}, we have
\begin{equation*}
\begin{split}
J_{21}\leq C_{1}(1+\epsilon^{2}|\!|\!|(N_{e},U)|\!|\!|_{\epsilon}^{4})|\!|\!|(N_{e},U)|\!|\!|_{\epsilon}^{2}.
\end{split}
\end{equation*}
$J_{16}$ and $J_{20}$ are similar to $J_{21}$ and can be bounded by $C_{1}(1+\epsilon^{2}|\!|\!|(N_{e},U)|\!|\!|_{\epsilon}^{4})|\!|\!|(N_{e},U)|\!|\!|_{\epsilon}^{2}$¡£

\emph{Estimate of $J_{22}$.} By using \eqref{equ28}\ and \ \eqref{equ29} in Lemma \ref{L8}, similarly we have
\begin{equation*}
\begin{split}
J_{22}
\leq C_{1}(1+\epsilon^{2}|\!|\!|(N_{e},U)|\!|\!|_{\epsilon}^{4})|\!|\!|(N_{e},U)|\!|\!|_{\epsilon}^{2}.
\end{split}
\end{equation*}
The proof of Lemma \eqref{Lem-u1} is then complete.
\end{proof}

\subsection{Estimate of $K_{12}$.}\label{2.5}Next, we estimate $K_{12}$ in \eqref{biaohao}.
\begin{lemma}[\textbf{\emph{Estimate of $K_{12}$}}]\label{Lem-u2} Let $(N_{i},N_{e},U)$ be a solution to \eqref{rem}, then there holds
\begin{equation*}
\begin{split}
K_{12}\leq &-\frac\epsilon2\frac{d}{dt}\left[\int\frac{n_{e}}{n_{i}}(\partial_{x}^{3}N_{e})^{2}
+\epsilon\int\frac{n_{e}^{2}}{n_{i}}(\partial_{x}^{4}N_{e})^{2}
+\frac{\epsilon^{2}H^{2}}{4}\int\frac{1}{n_{i}}(\partial_{x}^{5}N_{e})^{2}\right]\\
&-\frac\epsilon2\frac{\epsilon H^{2}}{4}\frac{d}{dt}\left[\int\frac{1}{n_{e}n_{i}}(\partial_{x}^{4}N_{e})^{2}
+\epsilon\int\frac{1}{n_{i}}(\partial_{x}^{5}N_{e})^{2}
+\frac{\epsilon^{2}H^{2}}{4}\int\frac{1}{(n_{e})^{2}n_{i}}(\partial_{x}^{6}N_{e})^{2}\right]\\
&+C_1(1+\epsilon^2|\!|\!|(N_{e},U)|\!|\!|^{6}_{\epsilon})(1+|\!|\!|(N_{e},U)|\!|\!|^2_{\epsilon}),
\end{split}
\end{equation*}
\end{lemma}

\begin{proof}[Proof.]Recall that in \eqref{biaohao}
\begin{equation*}
\begin{split}
K_{12}=-\epsilon\int\left(\frac{n_{e}}{n_{i}}\partial_{x}^{3}N_{e}
-\frac{\epsilon H^{2}}{4}\frac{1}{n_{e}n_{i}}\partial_{x}^{5}N_{e}\right)\partial_{t}\partial_{x}^{3}N_{i}.
\end{split}
\end{equation*}
By the $\eqref{rem-3}$, we have
\begin{equation*}
\begin{split}
\partial_{t}&\partial_{x}^{3}N_{i}=\partial_{t}\partial_{x}^{3}N_{e}-\epsilon\partial_{t}\partial_{x}^{3}\big(n_{e}\partial_{x}^{2}N_{i}\big)
+\frac{\epsilon^{2}H^{2}}{4}\partial_{t}\partial_{x}^{3}\big(\frac{\partial_{x}^{4}N_{e}}{n_{e}}\big)\\
&-2\epsilon^{2}\partial_{t}\partial_{x}^{3}\big(\partial_{x}\tilde{n_{e}}\partial_{x}N_{e}\big)
-\epsilon^{4}\partial_{t}\partial_{x}^{3}\big((\partial_{x}N_{e})^{2}\big)
-\epsilon^{2}\partial_{t}\partial_{x}^{3}\big(\partial_{x}^{2}\tilde{n_{e}}N_{e}\big)\\
&-\epsilon^{2}\partial_{t}\partial_{x}^{3}R_{(3)}^{1}
+\frac{H^{2}}{4}\Big\{-12\epsilon^{5}\partial_{t}\partial_{x}^{3}\big(\frac{(\partial_{x}\tilde{n_{e}})^{3}}{n_{e}^{4}}\partial_{x}N_{e}\big)
+14\epsilon^{4}\partial_{t}\partial_{x}^{3}\big(\frac{\partial_{x}\tilde{n_{e}}\partial_{x}^{2}\tilde{n_{e}}}{n_{e}^{3}}\partial_{x}N_{e}\big)\\
&-3\epsilon^{3}\partial_{t}\partial_{x}^{3}\big(\frac{\partial_{x}^{3}\tilde{n_{e}}}{n_{e}^{2}}\partial_{x}N_{e}\big)
-18\epsilon^{7}\partial_{t}\partial_{x}^{3}\big(\frac{(\partial_{x}\tilde{n_{e}})^{2}}{n_{e}^{4}}(\partial_{x}N_{e})^{2}\big)
+7\epsilon^{4}\partial_{t}\partial_{x}^{3}\big(\frac{(\partial_{x}\tilde{n_{e}})^{2}}{n_{e}^{3}}\partial_{x}^{2}N_{e}\big)\\
&+7\epsilon^{6}\partial_{t}\partial_{x}^{3}\big(\frac{\partial_{x}^{2}\tilde{n_{e}}}{n_{e}^{3}}(\partial_{x}N_{e})^{2}\big)
-4\epsilon^{3}\partial_{t}\partial_{x}^{3}\big(\frac{\partial_{x}^{2}\tilde{n_{e}}}{n_{e}^{2}}\partial_{x}^{2}N_{e}\big)
-12\epsilon^{9}\partial_{t}\partial_{x}^{3}\big(\frac{\partial_{x}\tilde{n_{e}}}{n_{e}^{4}}(\partial_{x}N_{e})^{3}\big)\\
&+14\epsilon^{6}\partial_{t}\partial_{x}^{3}\big(\frac{\partial_{x}\tilde{n_{e}}}{n_{e}^{2}}\partial_{x}N_{e}\partial_{x}^{2}N_{e}\big)
-3\epsilon^{3}\partial_{t}\partial_{x}^{3}\big(\frac{\partial_{x}\tilde{n_{e}}}{n_{e}^{2}}\partial_{x}^{3}N_{e}\big)
-3\epsilon^{11}\partial_{t}\partial_{x}^{3}\big(\frac{1}{n_{e}^{4}}(\partial_{x}N_{e})^{4}\big)\\
&+7\epsilon^{8}\partial_{t}\partial_{x}^{3}\big(\frac{1}{n_{e}^{3}}(\partial_{x}N_{e})^{2}\partial_{x}^{2}N_{e}\big)
-2\epsilon^{5}\partial_{t}\partial_{x}^{3}\big(\frac{1}{n_{e}^{2}}(\partial_{x}^{2}N_{e})^{2}\big)
-3\epsilon^{5}\partial_{t}\partial_{x}^{3}\big(\frac{1}{n_{e}^{2}}\partial_{x}N_{e}\partial_{x}^{3}N_{e}\big)\\
&+\epsilon^{2}\partial_{t}\partial_{x}^{3}\big(\frac{R_{(3)}^{2}+R_{(3)}^{3}}{n_{e}^{4}}\big)\Big\}\\
=:&\sum_{i=1}^{22}D_{i}\ .
\end{split}
\end{equation*}
Accordingly, $K_{12}$ is decomposed into
\begin{equation*}
\begin{split}
K_{12}=-\sum_{i=1}^{22}\epsilon\int(\frac{n_{e}}{n_{i}}\partial_{x}^{3}N_{e}
-\frac{\epsilon H^{2}}{4}\frac{1}{n_{e}n_{i}}\partial_{x}^{5}N_{e})D_{i}
=:\sum_{i=1}^{22}I_{i}\ .
\end{split}
\end{equation*}

\emph{Estimate of $I_{1}$.}
\begin{equation*}
\begin{split}
I_{1}=-\epsilon\int(\frac{n_{e}}{n_{i}}\partial_{x}^{3}N_{e}
-\frac{\epsilon H^{2}}{4}\frac{1}{n_{e}n_{i}}\partial_{x}^{5}N_{e})\partial_{t}\partial_{x}^{3}N_{e}=:I_{11}+I_{12}.
\end{split}
\end{equation*}
By integration by parts, we have
\begin{equation*}
\begin{split}
I_{11}&=-\epsilon\int\frac{n_{e}}{n_{i}}\partial_{x}^{3}N_{e}\partial_{t}\partial_{x}^{3}N_{e}\\
&=-\frac{\epsilon}{2}\frac{d}{dt}\int\frac{n_{e}}{n_{i}}(\partial_{x}^{3}N_{e})^{2}
+\frac{\epsilon}{2}\int(\partial_{t}\frac{n_{e}}{n_{i}})(\partial_{x}^{3}N_{e})^{2}.
\end{split}
\end{equation*}
By direct computation, we have
\begin{equation}
\begin{split}\label{use 23}
\|\partial_{t}\frac{n_{e}}{n_{i}}\|_{L^{\infty}}
\leq C\big(\epsilon+\epsilon^{3}(\|\partial_{t}N_{e}\|_{L^{\infty}}+\|\partial_{t}N_{i}\|_{L^{\infty}})\big).
\end{split}
\end{equation}
Thus by Sobolev embedding $H^1\hookrightarrow L^{\infty}$ and Lemma \ref{L2}-\ref{L3}, we have
\begin{equation*}
\begin{split}
\frac{\epsilon}{2}\int(\partial_{t}\frac{n_{e}}{n_{i}})(\partial_{x}^{3}N_{e})^{2}
&\leq C\|\partial_{t}\frac{n_{e}}{n_{i}}\|_{L^{\infty}}(\epsilon\|\partial_{x}^{3}N_{e}\|^{2})\\
&\leq C_{1}(1+\epsilon^{2}(\|\epsilon\partial_{tx}N_{e}\|^{2}+\|\epsilon\partial_{tx}N_{i}\|^{2}))(\epsilon\|\partial_{x}^{3}N_{e}\|^{2})\\
&\leq C_{1}(1+\epsilon^{2}|\!|\!|(N_{e},U)|\!|\!|_{\epsilon}^{2})|\!|\!|(N_{e},U)|\!|\!|_{\epsilon}^{2}.
\end{split}
\end{equation*}
Applying integration by parts again twice, we have
\begin{equation*}
\begin{split}
I_{12}
=&-\frac{H^{2}}{4}\frac{\epsilon^{2}}{2}\frac{d}{dt}\int\frac{1}{n_{e}n_{i}}(\partial_{x}^{4}N_{e})^{2}
+\frac{H^{2}}{4}\frac{\epsilon^{2}}{2}\int\partial_{t}(\frac{1}{n_{e}n_{i}})(\partial_{x}^{4}N_{e})^{2}\\
&-\frac{\epsilon^{2} H^{2}}{4}\int\partial_{x}(\frac{1}{n_{e}n_{i}})\partial_{x}^{4}N_{e}\partial_{t}\partial_{x}^{3}N_{e}\\
=&:I_{121}+I_{122}+I_{123}\ .
\end{split}
\end{equation*}
Note that the estimate of $\|\partial_{t}(\frac{1}{n_{e}n_{i}})\|_{L^{\infty}}$ is similar to that for \eqref{use 23}, thus similarly $I_{122}$ can be estimated by $C_{1}(1+\epsilon^{2}|\!|\!|(N_{e},U)|\!|\!|_{\epsilon}^{2})|\!|\!|(N_{e},U)|\!|\!|_{\epsilon}^{2}$. By \eqref{use 11}, Sobolev embedding theorem and Cauchy inequality, we have
\begin{equation*}
\begin{split}
I_{123}&\leq C\big(1+\epsilon^{2}(\|\partial_{x}N_{e}\|_{L^{\infty}}^{2}
+\|\partial_{x}N_{i}\|_{L^{\infty}}^{2})\big)(\epsilon^{2}\|\partial_{x}^{4}N_{e}\|^{2}
+\epsilon^{2}\|\epsilon\partial_{t}\partial_{x}^{3}N_{e}\|^{2})\\
&\leq C\big(1+\epsilon^{2}(\|N_{e}\|_{H^{2}}^{2}+\|N_{i}\|_{H^{2}}^{2})\big)
(\epsilon^{2}\|\partial_{x}^{4}N_{e}\|^{2}+\epsilon^{2}\|\epsilon\partial_{t}\partial_{x}^{3}N_{e}\|^{2})\\
&\leq C_{1}(1+\epsilon^{2}|\!|\!|(N_{e},U)|\!|\!|_{\epsilon}^{2})|\!|\!|(N_{e},U)|\!|\!|_{\epsilon}^{2},
\end{split}
\end{equation*}
where we have used Lemma \ref{L1}-\ref{L3}.
Therefore, we obtain
\begin{equation}
\begin{split}\label{w1}
I_{1}\leq&-\frac{\epsilon}{2}\frac{d}{dt}\int\frac{n_{e}}{n_{i}}(\partial_{x}^{3}N_{e})^{2}
-\frac{H^{2}}{4}\frac{\epsilon^{2}}{2}\frac{d}{dt}\int\frac{1}{n_{i}n_{e}}(\partial_{x}^{4}N_{e})^{2}\\
&+C_{1}(1+\epsilon^{2}|\!|\!|(N_{e},U)|\!|\!|_{\epsilon}^{2})|\!|\!|(N_{e},U)|\!|\!|_{\epsilon}^{2}.
\end{split}
\end{equation}

\emph{Estimate of $I_{2}$.} Recall that
\begin{equation*}
\begin{split}
I_{2}=\epsilon^{2}\int(\frac{n_{e}}{n_{i}}\partial_{x}^{3}N_{e}
-\frac{\epsilon H^{2}}{4}\frac{1}{n_{e}n_{i}}\partial_{x}^{5}N_{e})\partial_{t}\partial_{x}^{3}(n_{e}\partial_{x}^{2}N_{e})
=:I_{21}+I_{22}\ .
\end{split}
\end{equation*}

\emph{Estimate of $I_{21}$.} By integration by parts, we have
\begin{equation*}
\begin{split}
I_{21}&=\epsilon^{2}\int\frac{n_{e}}{n_{i}}\partial_{x}^{3}N_{e}\partial_{t}\partial_{x}^{3}(n_{e}\partial_{x}^{2}N_{e})\\
&=-\epsilon^{2}\int\frac{n_{e}}{n_{i}}\partial_{x}^{4}N_{e}\partial_{t}\partial_{x}^{2}(n_{e}\partial_{x}^{2}N_{e})
-\epsilon^{2}\int(\partial_{x}\frac{n_{e}}{n_{i}})\partial_{x}^{3}N_{e}\partial_{t}\partial_{x}^{2}(n_{e}\partial_{x}^{2}N_{e})\\
&=:I_{211}+I_{212}.
\end{split}
\end{equation*}

\emph{Estimate of $I_{211}$.} By direct computation, we have
\begin{equation*}
\begin{split}
I_{211}=&-\epsilon^{2}\int\frac{n_{e}^{2}}{n_{i}}\partial_{x}^{4}N_{e}\partial_{t}\partial_{x}^{4}N_{e}
-\epsilon^{2}\int\frac{n_{e}}{n_{i}}(\partial_{x}^{4}N_{e})^{2}\partial_{t}n_{e}\\
&-\epsilon^{2}\int\frac{n_{e}}{n_{i}}\partial_{x}^{4}N_{e}\partial_{t}
\big(2\partial_{x}n_{e}\partial_{x}^{3}N_{e}+\partial_{x}^{2}n_{e}\partial_{x}^{2}N_{e}\big)\\
=&:I_{2111}+I_{2112}+I_{2113}.
\end{split}
\end{equation*}
Note that the estimate of $\|\partial_{t}({n_{e}^{2}}/{n_{i}})\|_{L^{\infty}}$ is similar to \eqref{use 23}, thus by integration by parts,
\begin{equation*}
\begin{split}
I_{2111}&=-\frac{\epsilon^{2}}{2}\frac{d}{dt}\int\frac{n_{e}^{2}}{n_{i}}(\partial_{x}^{4}N_{e})^{2}
+\frac{\epsilon^{2}}{2}\int\partial_{t}(\frac{n_{e}^{2}}{n_{i}})(\partial_{x}^{4}N_{e})^{2}\\
&\leq-\frac{\epsilon^{2}}{2}\frac{d}{dt}\int\frac{n_{e}^{2}}{n_{i}}(\partial_{x}^{4}N_{e})^{2}
+C_{1}(1+\epsilon^{2}|\!|\!|(N_{e},U)|\!|\!|_{\epsilon}^{2})|\!|\!|(N_{e},U)|\!|\!|_{\epsilon}^{2}.
\end{split}
\end{equation*}
By H\"older inequality, Cauchy inequality, Sobolev embedding theorem and Lemma \ref{L1}-\ref{L3}, we have
\begin{equation*}
\begin{split}
I_{2112}+I_{2113}
\leq C_{1}(1+\epsilon^{2}|\!|\!|(N_{e},U)|\!|\!|_{\epsilon}^{2})|\!|\!|(N_{e},U)|\!|\!|_{\epsilon}^{2}.
\end{split}
\end{equation*}
By \eqref{use 11} and direct computation, we have
\begin{equation*}
\begin{split}
I_{212}\leq C_{1}(1+\epsilon^{2}|\!|\!|(N_{e}, U)|\!|\!|_{\epsilon}^{2})(1+|\!|\!|(N_{e}, U)|\!|\!|_{\epsilon}^{2}).
\end{split}
\end{equation*}

\emph{Estimate of $I_{22}$.} By direct computation, we have
\begin{equation*}
\begin{split}
I_{22}=&-\frac{\epsilon^{3}H^{2}}{4}\int\frac{1}{n_{e}n_{i}}\partial_{x}^{5}N_{e}\partial_{t}\partial_{x}^{3}(n_{e}\partial_{x}^{2}N_{e})\\
=&-\frac{\epsilon^{3}H^{2}}{4}\int\frac{1}{n_{e}n_{i}}\partial_{x}^{5}N_{e}\partial_{t}
\big(n_{e}\partial_{x}^{5}N_{e}+\sum_{\beta=1}^{3}C_{3}^{\beta}\partial_{x}^{\beta}n_{e}\partial_{x}^{5-\beta}N_{e}\big)\\
=&-\frac{\epsilon^{3}H^{2}}{4}\int\frac{1}{n_{i}}\partial_{x}^{5}N_{e}\partial_{t}\partial_{x}^{5}N_{e}
-\frac{\epsilon^{3}H^{2}}{4}\int\frac{1}{n_{e}n_{i}}\partial_{t}n_{e}(\partial_{x}^{5}N_{e})^{2}\\
&-\frac{\epsilon^{3}H^{2}}{4}\int\frac{\partial_{x}^{5}N_{e}}{n_{e}n_{i}}\partial_{t}
\big(\sum_{\beta=1}^{3}C_{3}^{\beta}\partial_{x}^{\beta}n_{e}\partial_{x}^{5-\beta}N_{e}\big)\\
=&:I_{221}+I_{222}+I_{223}.
\end{split}
\end{equation*}
By integration by parts in $t$, we have
\begin{equation*}
\begin{split}
I_{221}=-\frac{H^{2}}{4}\frac{\epsilon^{2}}{2}\frac{d}{dt}\int\frac{1}{n_{i}}(\partial_{x}^{5}N_{e})^{2}
+\frac{H^{2}}{4}\frac{\epsilon^{2}}{2}\int(\partial_{x}^{5}N_{e})^{2}\partial_{t}\frac{1}{n_{i}}.
\end{split}
\end{equation*}
Note that
\begin{equation}
\begin{split}\label{use 5}
\|\partial_{t}\frac{1}{n_{i}}\|_{L^{\infty}}\leq C(\epsilon+\epsilon^{3}\|\partial_{t}{N_{i}}\|_{L^{\infty}}).
\end{split}
\end{equation}
Therefore, by Sobolev embedding theorem and Lemma \ref{L2}, we have
\begin{equation*}
\begin{split}
\frac{H^{2}}{4}\frac{\epsilon^{2}}{2}\int(\partial_{x}^{5}N_{e})^{2}\partial_{t}\frac{1}{n_{i}}
&\leq C_{1}(1+\epsilon^{2}\|\epsilon\partial_{t}N_{i}\|_{H^{1}}^{2})(\epsilon^{3}\|\partial_{x}^{5}N_{e}\|^{2})\\
&\leq C_{1}(1+\epsilon^{2}|\!|\!|(N_{e},U)|\!|\!|_{\epsilon}^{2})|\!|\!|(N_{e},U)|\!|\!|_{\epsilon}^{2}.
\end{split}
\end{equation*}
Similarly, by Sobolev embedding theorem and Lemma \ref{L2}-\ref{L3}, we have
\begin{equation*}
\begin{split}
I_{222}\leq C_{1}(1+\epsilon^{2}\|\epsilon\partial_{t}N_{e}\|_{H^{1}}^{2})(\epsilon^{3}\|\partial_{x}^{5}N_{e}\|^{2})
\leq C_{1}(1+\epsilon^{2}|\!|\!|(N_{e},U)|\!|\!|_{\epsilon}^{2})|\!|\!|(N_{e},U)|\!|\!|_{\epsilon}^{2}.
\end{split}
\end{equation*}
By Cauchy inequality, Sobolev embedding theorem and Lemma \ref{L1}-\ref{L3}, we have
\begin{equation*}
\begin{split}
I_{223}\leq C_{1}(1+\epsilon^{2}|\!|\!|(N_{e},U)|\!|\!|_{\epsilon}^{2})|\!|\!|(N_{e},U)|\!|\!|_{\epsilon}^{2}.
\end{split}
\end{equation*}
Therefore, we have
\begin{equation}
\begin{split}\label{w2}
I_{2}\leq&-\frac{\epsilon^{2}}{2}\frac{d}{dt}\int\frac{n_{e}^{2}}{n_{i}}(\partial_{x}^{4}N_{e})^{2}
-\frac{H^{2}}{4}\frac{\epsilon^{2}}{2}\frac{d}{dt}\int\frac{1}{n_{i}}(\partial_{x}^{5}N_{e})^{2}\\
&+C_{1}(1+\epsilon^{2}|\!|\!|(N_{e},U)|\!|\!|_{\epsilon}^{2})(1+|\!|\!|(N_{e},U)|\!|\!|_{\epsilon}^{2}).
\end{split}
\end{equation}

\emph{Estimate of $I_{3}$.}
\begin{equation*}
\begin{split}
I_{3}=-\frac{\epsilon^{3}H^{2}}{4}\int(\frac{n_{e}}{n_{i}}\partial_{x}^{3}N_{e}
-\frac{\epsilon H^{2}}{4}\frac{1}{n_{e}n_{i}}\partial_{x}^{5}N_{e})\partial_{t}\partial_{x}^{3}(\frac{\partial_{x}^{4}N_{e}}{n_{e}})
=:I_{31}+I_{32}.
\end{split}
\end{equation*}

\emph{Estimate of $I_{31}$.}
By integration by parts twice,
\begin{equation*}
\begin{split}
I_{31}=&-\frac{\epsilon^{3}H^{2}}{4}\int\frac{n_{e}}{n_{i}}\partial_{x}^{5}N_{e}\partial_{tx}(\frac{\partial_{x}^{4}N_{e}}{n_{e}})
-\frac{\epsilon^{3}H^{2}}{2}\int\partial_{x}^{4}N_{e}\partial_{x}(\frac{n_{e}}{n_{i}})\partial_{tx}(\frac{\partial_{x}^{4}N_{e}}{n_{e}})\\
&-\frac{\epsilon^{3}H^{2}}{4}\int\partial_{x}^{3}N_{e}\partial_{x}^{2}(\frac{n_{e}}{n_{i}})\partial_{tx}(\frac{\partial_{x}^{4}N_{e}}{n_{e}})\\
=&:I_{311}+I_{312}+I_{313}.
\end{split}
\end{equation*}
By direct computation, we have
\begin{equation*}
\begin{split}
I_{311}=&-\frac{\epsilon^{3}H^{2}}{4}\int\frac{1}{n_{i}}\partial_{x}^{5}N_{e}\partial_{t}\partial_{x}^{5}N_{e}\\
&-\frac{\epsilon^{3}H^{2}}{4}\int\frac{n_{e}}{n_{i}}\partial_{x}^{5}N_{e}\big((\partial_{t}\frac{1}{n_{e}})\partial_{x}^{5}N_{e}
+(\partial_{t}\frac{1}{n_{e}})\partial_{t}\partial_{x}^{4}N_{e}+(\partial_{tx}\frac{1}{n_{e}})\partial_{x}^{4}N_{e}\big)\\
=&:I_{3111}+I_{3112}.
\end{split}
\end{equation*}
By integration by parts in $t$, Sobolev embedding, \eqref{use 5} and Lemma \ref{L2}, we have
\begin{equation*}
\begin{split}
I_{3111}=&-\frac{1}{2}\frac{\epsilon^{3}H^{2}}{4}\frac{d}{dt}\int\frac{1}{n_{i}}(\partial_{x}^{5}N_{e})^{2}
+\frac{\epsilon^{3}H^{2}}{4}\int\partial_{t}(\frac{1}{n_{i}})(\partial_{x}^{5}N_{e})^{2}\\
\leq &-\frac{1}{2}\frac{\epsilon^{3}H^{2}}{4}\frac{d}{dt}\int\frac{1}{n_{i}}(\partial_{x}^{5}N_{e})^{2}
+C_{1}(1+\epsilon^{2}|\!|\!|(N_{e},U)|\!|\!|_{\epsilon}^{2})|\!|\!|(N_{e},U)|\!|\!|_{\epsilon}^{2}.
\end{split}
\end{equation*}
Note that
\begin{equation}
\begin{split}\label{use 4}
\|\partial_{t}\frac{1}{n_{e}}\|_{L^{\infty}}
\leq C(\epsilon+\epsilon^{3}\|\partial_{t}N_{e}\|_{L^{\infty}}),
\end{split}
\end{equation}
and
\begin{equation}
\begin{split}\label{use 27}
\|\partial_{tx}\frac{1}{n_{e}}\|_{L^{\infty}}
\leq C\big(\epsilon+\epsilon^{3}(\|\partial_{t}N_{e}\|_{L^{\infty}}+\|\partial_{x}N_{e}\|_{L^{\infty}})
+\epsilon^{6}\|\partial_{t}N_{e}\|_{L^{\infty}}\|\partial_{x}N_{e}\|_{L^{\infty}}\big).
\end{split}
\end{equation}
By Sobolev embedding $H^1\hookrightarrow L^{\infty}$, Cauchy inequality and Lemma \ref{L2}-\ref{L3}, we have
\begin{equation*}
\begin{split}
I_{3112}\leq
C_{1}(1+\epsilon^{2}|\!|\!|(N_{e},U)|\!|\!|_{\epsilon}^{2})(1+|\!|\!|(N_{e},U)|\!|\!|_{\epsilon}^{2}).
\end{split}
\end{equation*}
Therefore, we obtain
\begin{equation*}
\begin{split}
I_{311}\leq-\frac{1}{2}\frac{\epsilon^{3}H^{2}}{4}\frac{d}{dt}\int\frac{1}{n_{i}}(\partial_{x}^{5}N_{e})^{2}
+C_{1}(1+\epsilon^{2}|\!|\!|(N_{e},U)|\!|\!|_{\epsilon}^{2})(1+|\!|\!|(N_{e},U)|\!|\!|_{\epsilon}^{2}).
\end{split}
\end{equation*}
By direct computation, we have
\begin{equation*}
\begin{split}
I_{312}&=-\frac{\epsilon^{3}H^{2}}{2}\int\partial_{x}^{4}N_{e}\partial_{x}(\frac{n_{e}}{n_{i}})\partial_{tx}(\frac{\partial_{x}^{4}N_{e}}{n_{e}})\\
&=-\frac{\epsilon^{3}H^{2}}{2}\int\partial_{x}^{4}N_{e}\partial_{x}(\frac{n_{e}}{n_{i}})
\big(\frac{1}{n_{e}}\partial_{t}\partial_{x}^{5}N_{e}+\partial_{t}\frac{1}{n_{e}}\partial_{x}^{5}N_{e}
+\partial_{x}\frac{1}{n_{e}}\partial_{t}\partial_{x}^{4}N_{e}
+\partial_{tx}\frac{1}{n_{e}}\partial_{x}^{4}N_{e}\big).
\end{split}
\end{equation*}
By Sobolev embedding, Cauchy inequality and Lemma \ref{L2}-\ref{L3}, we have
\begin{equation*}
\begin{split}
I_{312}\leq C_{1}(1+\epsilon^{2}|\!|\!|(N_{e},U)|\!|\!|_{\epsilon}^{2})(1+|\!|\!|(N_{e},U)|\!|\!|_{\epsilon}^{2}),
\end{split}
\end{equation*}
where we have used \eqref{use 11},\ \eqref{use 4}\ and\ \eqref{use 27}.
$I_{313}$ is similar to $I_{312}$, thus we have
\begin{equation}
\begin{split}\label{a2}
I_{31}\leq-\frac{1}{2}\frac{\epsilon^{3}H^{2}}{4}\frac{d}{dt}\int\frac{1}{n_{i}}(\partial_{x}^{5}N_{e})^{2}
+C_{1}(1+\epsilon^{2}|\!|\!|(N_{e},U)|\!|\!|_{\epsilon}^{2})(1+|\!|\!|(N_{e},U)|\!|\!|_{\epsilon}^{2}).
\end{split}
\end{equation}

\emph{Estimate of $I_{32}$.} By integration by parts, we have
\begin{equation*}
\begin{split}
I_{32}&=\frac{\epsilon^{3}H^{2}}{4}\frac{\epsilon H^{2}}{4}\int\frac{1}{n_{e}n_{i}}
\partial_{x}^{5}N_{e}\partial_{t}\partial_{x}^{3}(\frac{\partial_{x}^{4}N_{e}}{n_{e}})\\
&=-\frac{\epsilon^{3}H^{2}}{4}\frac{\epsilon H^{2}}{4}\int\frac{1}{n_{e}n_{i}}\partial_{x}^{6}N_{e}
\partial_{t}\partial_{x}^{2}(\frac{\partial_{x}^{4}N_{e}}{n_{e}})
-\frac{\epsilon^{3}H^{2}}{4}\frac{\epsilon H^{2}}{4}\int(\partial_{x}\frac{1}{n_{e}n_{i}})\partial_{x}^{5}N_{e}
\partial_{t}\partial_{x}^{2}(\frac{\partial_{x}^{4}N_{e}}{n_{e}})\\
&=:I_{321}+I_{322}.
\end{split}
\end{equation*}
By direct computation, we have
\begin{equation*}
\begin{split}
I_{321}=&-\frac{\epsilon^{3}H^{2}}{4}\frac{\epsilon H^{2}}{4}\int\frac{1}{n_{e}n_{i}}\partial_{x}^{6}N_{e}\partial_{t}\partial_{x}^{6}N_{e}\\
&-\frac{\epsilon^{3}H^{2}}{4}\frac{\epsilon H^{2}}{4}\int\frac{1}{n_{e}n_{i}}\partial_{x}^{6}N_{e}
\big((\partial_{t}\frac{1}{n_{e}})\partial_{x}^{6}N_{e}+2(\partial_{tx}\frac{1}{n_{e}})\partial_{x}^{5}N_{e}\\
&\ \ \ \ \ \ \ \ \ \ \ \ \ \ \ \ \ \ \ \ +2(\partial_{x}\frac{1}{n_{e}})\partial_{t}\partial_{x}^{5}N_{e}+(\partial_{x}^{2}\frac{1}{n_{e}})\partial_{t}\partial_{x}^{4}N_{e}
+(\partial_{t}\partial_{x}^{2}\frac{1}{n_{e}})\partial_{x}^{4}N_{e}\big)\\
=&:I_{3211}+I_{3212}.
\end{split}
\end{equation*}
By integration by parts in $t$, we have
\begin{equation*}
\begin{split}
I_{3211}=-\frac{1}{2}\frac{\epsilon^{3}H^{2}}{4}\frac{\epsilon H^{2}}{4}\frac{d}{dt}\int\frac{1}{n_{e}^{2}n_{i}}(\partial_{x}^{6}N_{e})^{2}
+\frac{1}{2}\frac{\epsilon^{3}H^{2}}{4}\frac{\epsilon H^{2}}{4}\int\partial_{t}(\frac{1}{n_{e}^{2}n_{i}})(\partial_{x}^{6}N_{e})^{2}.
\end{split}
\end{equation*}
Similar to \eqref{use 23}, $\|\partial_{t}({1}/{n_{e}^{2}n_{i}})\|_{L^{\infty}}$ has same bound with $\|\partial_{t}({n_{e}}/{n_{i}})\|_{L^{\infty}}$. Therefore, by Sobolev embedding $H^1\hookrightarrow L^{\infty}$ and Lemma \ref{L2}-\ref{L3}, we have
\begin{equation*}
\begin{split}
\frac{1}{2}\frac{\epsilon^{3}H^{2}}{4}\frac{\epsilon H^{2}}{4}\int\partial_{t}(\frac{1}{n_{e}^{2}n_{i}})(\partial_{x}^{6}N_{e})^{2}
&\leq C(1+\epsilon^{2}(\|\epsilon\partial_{t}N_{e}\|_{L^{\infty}}^{2}
+\|\epsilon\partial_{t}N_{i}\|_{L^{\infty}}^{2}))\epsilon^{4}\|\partial_{x}^{6}N_{e}\|^{2}\\
&\leq C(1+\epsilon^{2}|\!|\!|(N_{e},U)|\!|\!|_{\epsilon}^{2})|\!|\!|(N_{e},U)|\!|\!|_{\epsilon}^{2}.
\end{split}
\end{equation*}
By direct computation, we have
\begin{equation}
\begin{split}\label{use 235}
\|\partial_{t}\partial_{x}^{2}\frac{1}{n_{e}}\|_{L^{\infty}}
\leq &C\big(\epsilon+\epsilon^{3}(\|\partial_{x}N_{e}\|_{L^{\infty}}+\|\partial_{t}N_{e}\|_{L^{\infty}}
+\|\partial_{t}\partial_{x}^{2}N_{e}\|_{L^{\infty}})\\
&+\epsilon^{6}(\|\partial_{x}N_{e}\|_{L^{\infty}}\|\partial_{t}N_{e}\|_{L^{\infty}}
+\|\partial_{x}N_{e}\|_{L^{\infty}}^{2}+\|\partial_{t}N_{e}\|_{L^{\infty}}\|\partial_{x}^{2}N_{e}\|_{L^{\infty}})\\
&+\epsilon^{9}\|\partial_{t}N_{e}\|_{L^{\infty}}\|\partial_{x}N_{e}\|_{L^{\infty}}^{2}\big).
\end{split}
\end{equation}
By \eqref{use 1},\ \eqref{use 4},\ \eqref{use 27}\ and\ \eqref{use 235}, we have
\begin{equation*}
\begin{split}
I_{3212}\leq C_{1}(1+\epsilon^{2}|\!|\!|(N_{e},U)|\!|\!|_{\epsilon}^{4})(1+|\!|\!|(N_{e},U)|\!|\!|_{\epsilon}^{2}),
\end{split}
\end{equation*}
where we have used Sobolev embedding theorem, Lemma \ref{L2}\ and\ \ref{L3}. Thus, we have
\begin{equation}
\begin{split}\label{two}
I_{321}\leq-\frac{1}{2}\frac{\epsilon^{3}H^{2}}{4}\frac{\epsilon H^{2}}{4}\frac{d}{dt}\int\frac{1}{n_{e}^{2}n_{i}}(\partial_{x}^{6}N_{e})^{2}
+C_{1}(1+\epsilon^{2}|\!|\!|(N_{e},U)|\!|\!|_{\epsilon}^{4})(1+|\!|\!|(N_{e},U)|\!|\!|_{\epsilon}^{2}).
\end{split}
\end{equation}
By integration by parts, we have
\begin{equation*}
\begin{split}
I_{322}=&-\frac{\epsilon^{3}H^{2}}{4}\frac{\epsilon H^{2}}{4}\int\partial_{x}^{5}N_{e}(\partial_{x}\frac{1}{n_{e}n_{i}})
\partial_{t}\partial_{x}^{2}(\frac{\partial_{x}^{4}N_{e}}{n_{e}})\\
=&\frac{\epsilon^{3}H^{2}}{4}\frac{\epsilon H^{2}}{4}\int\partial_{x}^{6}N_{e}(\partial_{x}\frac{1}{n_{e}n_{i}})
\partial_{t}\partial_{x}(\frac{\partial_{x}^{4}N_{e}}{n_{e}})\\
&+\frac{\epsilon^{3}H^{2}}{4}\frac{\epsilon H^{2}}{4}\int\partial_{x}^{5}N_{e}(\partial_{x}^{2}\frac{1}{n_{e}n_{i}})
\partial_{t}\partial_{x}(\frac{\partial_{x}^{4}N_{e}}{n_{e}})\\
=&:I_{3221}+I_{3222}.
\end{split}
\end{equation*}
By direct computation, we have
\begin{equation*}
\begin{split}
I_{3221}=\frac{\epsilon^{3}H^{2}}{4}\frac{\epsilon H^{2}}{4}\int&\partial_{x}^{6}N_{e}(\partial_{x}\frac{1}{n_{e}n_{i}})
\Big(\frac{1}{n_{e}}\partial_{t}\partial_{x}^{5}N_{e}+(\partial_{t}\frac{1}{n_{e}})\partial_{x}^{5}N_{e}\\
&+(\partial_{x}\frac{1}{n_{e}})\partial_{t}\partial_{x}^{4}N_{e}+(\partial_{tx}\frac{1}{n_{e}})\partial_{x}^{4}N_{e}\Big).
\end{split}
\end{equation*}
By \eqref{one order},\ \eqref{use 11},\ \eqref{use 4}\ and\  \eqref{use 27}, Sobolev embedding theorem and Lemma \ref{L2}-\ref{L3}, we have
\begin{equation*}
\begin{split}
I_{3221}\leq C_{1}(1+\epsilon^{2}|\!|\!|(N_{e},U)|\!|\!|_{\epsilon}^{4})(1+|\!|\!|(N_{e},U)|\!|\!|_{\epsilon}^{2}).
\end{split}
\end{equation*}
By direct computation, we have
\begin{equation}
\begin{split}\label{use233}
\|\partial_{x}^{2}(\frac{1}{n_{e}n_{i}})\|_{L^{\infty}}
\leq &C\big(\epsilon+\epsilon^{3}(\|\partial_{x}N_{e}\|_{L^{\infty}}+\|\partial_{x}N_{i}\|_{L^{\infty}})
+\epsilon^{6}(\|\partial_{x}N_{e}\|_{L^{\infty}}\|\partial_{x}N_{i}\|_{L^{\infty}}\\
&+\|\partial_{x}N_{e}\|_{L^{\infty}}^{2}+\|\partial_{x}N_{i}\|_{L^{\infty}}^{2})\big).
\end{split}
\end{equation}
Thus by \eqref{one order},\ \eqref{use 4},\  \eqref{use 27}\  and\ \eqref{use233},  Sobolev embedding theorem and Lemma \ref{L1}-\ref{L3}, we have
\begin{equation*}
\begin{split}
I_{3222}\leq C_{1}(1+\epsilon^{2}|\!|\!|(N_{e},U)|\!|\!|_{\epsilon}^{4})(1+|\!|\!|(N_{e},U)|\!|\!|_{\epsilon}^{2}).
\end{split}
\end{equation*}
Therefore, we have
\begin{equation}
\begin{split}\label{a1}
I_{322}\leq C_{1}(1+\epsilon^{2}|\!|\!|(N_{e},U)|\!|\!|_{\epsilon}^{4})(1+|\!|\!|(N_{e},U)|\!|\!|_{\epsilon}^{2}).
\end{split}
\end{equation}
Adding \eqref{two} and \eqref{a1}, we have
\begin{equation}
\begin{split}\label{three}
I_{32}\leq-\frac{1}{2}\frac{\epsilon^{3}H^{2}}{4}\frac{\epsilon H^{2}}{4}\frac{d}{dt}\int\frac{1}{n_{e}^{2}n_{i}}(\partial_{x}^{6}N_{e})^{2}
+C_{1}(1+\epsilon^{2}|\!|\!|(N_{e},U)|\!|\!|_{\epsilon}^{4})(1+|\!|\!|(N_{e},U)|\!|\!|_{\epsilon}^{2}).
\end{split}
\end{equation}
Combining to \eqref{a2} and \eqref{three}, we have
\begin{equation}
\begin{split}\label{w3}
I_{3}\leq&-\frac{1}{2}\frac{\epsilon^{3}H^{2}}{4}\frac{d}{dt}\int\frac{1}{n_{i}}(\partial_{x}^{5}N_{e})^{2}
-\frac{1}{2}\frac{\epsilon^{3}H^{2}}{4}\frac{\epsilon H^{2}}{4}\frac{d}{dt}\int\frac{1}{n_{e}^{2}n_{i}}(\partial_{x}^{6}N_{e})^{2}\\
&+C_{1}(1+\epsilon^{2}|\!|\!|(N_{e},U)|\!|\!|_{\epsilon}^{4})(1+|\!|\!|(N_{e},U)|\!|\!|_{\epsilon}^{2}).
\end{split}
\end{equation}
Thus combining \eqref{w1},\ \eqref{w2}\ and\ \eqref{w3}, we obtain
\begin{equation*}
\begin{split}
\sum_{i=1}^{3}I_{i}
&\leq
-\frac{1}{2}\frac{d}{dt}\left\{\int\frac{n_{e}}{n_{i}}(\partial_{x}^{2}N_{e})^{2}+\epsilon\int\frac{n_{e}^{2}}{n_{i}}(\partial_{x}^{3}N_{e})^{2}
+\frac{\epsilon^{2}H^{2}}{4}\int\frac{1}{n_{i}}(\partial_{x}^{4}N_{e})^{2}\right\}\\
&-\frac{1}{2}\frac{\epsilon H^{2}}{4}\frac{d}{dt}\left\{\int\frac{1}{n_{e}n_{i}}(\partial_{x}^{3}N_{e})^{2}+\epsilon\int\frac{1}{n_{i}}(\partial_{x}^{4}N_{e})^{2}
+\frac{\epsilon^{2} H^{2}}{4}\int\frac{1}{n_{e}^{2}n_{i}}(\partial_{x}^{5}N_{e})^{2}\right\}\\
&+C_{1}(1+\epsilon^{2}|\!|\!|(N_{e},U)|\!|\!|_{\epsilon}^{4})(1+|\!|\!|(N_{e},U)|\!|\!|_{\epsilon}^{2}).
\end{split}
\end{equation*}
By Lemma \ref{L2}-\ref{L3}, we have
\begin{equation*}
\begin{split}
I_{4,5,6,7}\leq C_{1}(1+\epsilon^{2}|\!|\!|(N_{e},U)|\!|\!|_{\epsilon}^{2})(1+|\!|\!|(N_{e},U)|\!|\!|_{\epsilon}^{2}).
\end{split}
\end{equation*}

\emph{Estimate of $I_{18}$.} By direct computation, we have
\begin{equation*}
\begin{split}
I_{18}&=\frac{3 H^{2}}{4}\epsilon^{12}\int(\frac{n_{e}}{n_{i}}\partial_{x}^{3}N_{e}
-\frac{\epsilon H^{2}}{4}\frac{1}{n_{e}n_{i}}\partial_{x}^{5}N_{e})
\partial_{x}^{3}\big\{(\partial_{t}\frac{1}{n_{e}^{4}})(\partial_{x}N_{e})^{4}
+4\frac{1}{n_{e}^{4}}(\partial_{x}N_{e})^{3}\partial_{tx}N_{e}\big\}\\
&=:I_{181}+I_{182}\ .
\end{split}
\end{equation*}

\emph{Estimate of $I_{181}$.} Using commutator notation \eqref{w}, we have
\begin{equation*}
\begin{split}
\partial_{x}^{3}(\partial_{t}(\frac{1}{n_{e}^{4}})(\partial_{x}N_{e})^{4}
=[\partial_{x}^{3},\partial_{t}(\frac{1}{n_{e}^{4}})](\partial_{x}N_{e})^{4}
+\partial_{t}(\frac{1}{n_{e}^{4}})\partial_{x}^{3}\big((\partial_{x}^{4}N_{e})^{4}\big).
\end{split}
\end{equation*}
By commutator estimate \eqref{w11}, we have
\begin{equation*}
\begin{split}
\|[\partial_{x}^{3},\partial_{t}(\frac{1}{n_{e}^{4}})](\partial_{x}N_{e})^{4}\|
\leq\|\partial_{tx}(\frac{1}{n_{e}^{4}})\|_{L^{\infty}}\|\partial_{x}^{2}(\partial_{x}N_{e})^{4}\|
+\|\partial_{t}\partial_{x}^{3}(\frac{1}{n_{e}^{4}})\|\|\partial_{x}N_{e}\|^{4}_{L^{\infty}}.
\end{split}
\end{equation*}
Note that the estimate of $\|\partial_{tx}(\frac{1}{n_{e}^{4}})\|_{L^{\infty}}$ is similar to that for \eqref{use 27}. We note that
\begin{equation}
\begin{split}\label{h1}
\|\partial_{x}^{2}(\partial_{x}N_{e})^{4}\|\leq C(\|\partial_{x}N_{e}\|^{2}_{L^{\infty}}\|\partial_{x}^{2}N_{e}\|_{L^{\infty}}\|\partial_{x}^{2}N_{e}\|
+\|\partial_{x}N_{e}\|^{3}_{L^{\infty}}\|\partial_{x}^{3}N_{e}\|),
\end{split}
\end{equation}
and
\begin{equation}
\begin{split}\label{h2}
\|\partial_{t}&\partial_{x}^{3}(\frac{1}{n_{e}^{4}})\| \leq C\big(\epsilon+\epsilon^{3}(\|\partial_{x}N_{e}\|+\|\partial_{t}N_{e}\|+\|\partial_{x}^{2}N_{e}\|
+\|\partial_{tx}N_{e}\|\\&+\|\partial_{t}\partial_{x}^{2}N_{e}\|+\|\partial_{t}\partial_{x}^{2}N_{e}\|)
+\epsilon^{6}(\|\partial_{x}N_{e}\|^{2}+\|\partial_{t}N_{e}\|\|\partial_{x}N_{e}\|_{L^{\infty}}\\&+\|\partial_{x}N_{e}\|_{L^{\infty}}\|\partial_{tx}N_{e}\|
+\|\partial_{t}N_{e}\|_{L^{\infty}}\|\partial_{x}^{2}N_{e}\|+\|\partial_{t}N_{e}\partial_{x}^{3}N_{e}\|
+\|\partial_{tx}N_{e}\|_{L^{\infty}}\|\partial_{x}^{2}N_{e}\|\\&+\|\partial_{x}N_{e}\|_{L^{\infty}}\|\partial_{t}\partial_{x}^{2}N_{e}\|)
+\epsilon^{9}(\|\partial_{x}N_{e}\|_{L^{\infty}}^{2}\|\partial_{x}N_{e}\|
+\|\partial_{t}N_{e}\|\|\partial_{x}N_{e}\|_{L^{\infty}}\\
&+\|\partial_{t}N_{e}\|_{L^{\infty}}\|\partial_{x}N_{e}\|_{L^{\infty}}\|\partial_{x}^{2}N_{e}\|
+\|\partial_{tx}N_{e}\|\|\partial_{x}N_{e}\|_{L^{\infty}}^{2}\\
&+\epsilon^{12}\|\partial_{t}N_{e}\|_{L^{\infty}}\|\partial_{x}N_{e}\|_{L^{\infty}}^{2}\|\partial_{x}N_{e}\|\big).
\end{split}
\end{equation}
Thus by \eqref{use 27},\ \eqref{use 4},\ \eqref{h1},\  \eqref{h2}\  and\ \eqref{use 17},  Sobolev embedding theorem and Lemma \ref{L2}-\ref{L3}, we have
\begin{equation*}
\begin{split}
I_{181}\leq C_{1}(1+\epsilon^{2}|\!|\!|(N_{e},U)|\!|\!|_{\epsilon}^{4})
(1+|\!|\!|(N_{e},U)|\!|\!|_{\epsilon}^{2}).
\end{split}
\end{equation*}
\emph{Estimate of $I_{182}$.} Using commutator notation \eqref{w}, we have
\begin{equation*}
\begin{split}
\partial_{x}^{3}(\frac{1}{n_{e}^{4}}(\partial_{x}N_{e})^{3}\partial_{tx}N_{e})
=[\partial_{x}^{3},\frac{1}{n_{e}^{4}}](\partial_{x}N_{e})^{3}\partial_{tx}N_{e})
+\frac{1}{n_{e}^{4}}\partial_{x}^{3}((\partial_{x}N_{e})^{3}\partial_{tx}N_{e}).
\end{split}
\end{equation*}
By commutator estimate \eqref{w11} of Lemma \ref{L9}, we have
\begin{equation*}
\begin{split}
\|[\partial_{x}^{3},\frac{1}{n_{e}^{4}}](\partial_{x}N_{e})^{3}\partial_{tx}N_{e})\|
\leq&\|\partial_{x}(\frac{1}{n_{e}^{4}})\|_{L^{\infty}}\|\partial_{x}^{2}((\partial_{x}N_{e})^{3}\partial_{tx}N_{e})\|\\
&+\|\partial_{x}^{3}(\frac{1}{n_{e}^{4}})\|\|\partial_{x}N_{e}\|_{L^{\infty}}^{3}\|\partial_{tx}N_{e}\|_{L^{\infty}}.
\end{split}
\end{equation*}
By direct computation, we have
\begin{equation}
\begin{split}\label{m3}
\|\partial_{x}^{2}&((\partial_{x}N_{e})^{3}\partial_{tx}N_{e})\|
\leq C(\|\partial_{x}N_{e}\|_{L^{\infty}}\|\partial_{x}^{2}N_{e}\|_{L^{\infty}}^{2}\|\partial_{tx}N_{e}\|\\
&+\|\partial_{x}N_{e}\|_{L^{\infty}}^{2}\|\partial_{x}^{3}N_{e}\|_{L^{\infty}}\|\partial_{tx}N_{e}\|
+\|\partial_{x}N_{e}\|_{L^{\infty}}^{2}\|\partial_{x}^{2}N_{e}\|_{L^{\infty}}\|\partial_{t}\partial_{x}^{2}N_{e}\|\\
&+\|\partial_{x}N_{e}\|_{L^{\infty}}^{3}\|\partial_{t}\partial_{x}^{3}N_{e}\|),
\end{split}
\end{equation}
and
\begin{equation}
\begin{split}\label{m4}
\|\partial_{x}^{3}&((\partial_{x}N_{e})^{3}\partial_{tx}N_{e})\|
\leq C(\|\partial_{x}N_{e}\|_{L^{\infty}}\|\partial_{x}^{2}N_{e}\|_{L^{\infty}}\|\partial_{x}^{3}N_{e}\|_{L^{\infty}}\|\partial_{tx}N_{e}\|\\
&+\|\partial_{x}^{2}N_{e}\|^{3}_{L^{\infty}}\|\partial_{tx}N_{e}\|
+\|\partial_{x}N_{e}\|_{L^{\infty}}\|\partial_{x}^{2}N_{e}\|^{2}_{L^{\infty}}\|\partial_{t}\partial_{x}^{2}N_{e}\|\\
&+\|\partial_{x}N_{e}\|^{2}_{L^{\infty}}\|\partial_{x}^{4}N_{e}\|_{L^{\infty}}\|\partial_{tx}N_{e}\|
+\|\partial_{x}N_{e}\|^{2}_{L^{\infty}}\|\partial_{x}^{3}N_{e}\|_{L^{\infty}}\|\partial_{t}\partial_{x}^{2}N_{e}\|\\
&+\|\partial_{x}N_{e}\|^{2}_{L^{\infty}}\|\partial_{x}^{2}N_{e}\|_{L^{\infty}}\|\partial_{t}\partial_{x}^{3}N_{e}\|
+\|\partial_{x}N_{e}\|^{3}_{L^{\infty}}\|\partial_{t}\partial_{x}^{4}N_{e}\|).
\end{split}
\end{equation}
Thus by \eqref{use 1},\ \eqref{use 21},\  \eqref{m3}\  and\ \eqref{m4},  Sobolev embedding theorem and Lemma \ref{L2}-\ref{L3}, we have
\begin{equation*}
\begin{split}
I_{182}\leq C_{1}(1+\epsilon^{2}|\!|\!|(N_{e},U)|\!|\!|_{\epsilon}^{6})
(1+|\!|\!|(N_{e},U)|\!|\!|_{\epsilon}^{2}).
\end{split}
\end{equation*}
Therefore, we have
\begin{equation*}
\begin{split}
I_{18}\leq C_{1}(1+\epsilon^{2}|\!|\!|(N_{e},U)|\!|\!|_{\epsilon}^{6})
(1+|\!|\!|(N_{e},U)|\!|\!|_{\epsilon}^{2}).
\end{split}
\end{equation*}
The estimates of $I_{8}\sim I_{11},\ I_{13}\ and\ I_{15}$ are similar to that for $I_{18}$.

\emph{Estimate of $I_{21}$.} By direct computation, we have
\begin{equation*}
\begin{split}
I_{21}&=\frac{3 H^{2}}{4}\epsilon^{6}\int(\frac{n_{e}}{n_{i}}\partial_{x}^{3}N_{e}
-\frac{\epsilon H^{2}}{4}\frac{1}{n_{e}n_{i}}\partial_{x}^{5}N_{e})
\partial_{x}^{3}\big\{(\partial_{t}\frac{1}{n_{e}^{2}})\partial_{x}N_{e}\partial_{x}^{3}N_{e}
+\frac{1}{n_{e}^{2}}\partial_{t}(\partial_{x}N_{e}\partial_{x}^{3}N_{e})\big\}\\
&=:I_{211}+I_{212}.
\end{split}
\end{equation*}

\emph{Estimate of $I_{211}$.} Using commutator notation \eqref{w}, we have
\begin{equation*}
\begin{split}
\partial_{x}^{3}((\partial_{t}\frac{1}{n_{e}^{2}})\partial_{x}N_{e}\partial_{x}^{3}N_{e}
=[\partial_{x}^{3},\partial_{t}\frac{1}{n_{e}^{2}}]\partial_{x}N_{e}\partial_{x}^{3}N_{e}
+(\partial_{t}\frac{1}{n_{e}^{2}})\partial_{x}^{3}\big(\partial_{x}N_{e}\partial_{x}^{3}N_{e}\big).
\end{split}
\end{equation*}
By commutator estimate \eqref{w11} in Lemma \ref{L9}, we have
\begin{equation*}
\begin{split}
\|[\partial_{x}^{3},\partial_{t}(\frac{1}{n_{e}^{2}})]\partial_{x}N_{e}\partial_{x}^{3}N_{e}\|
\leq&\|\partial_{tx}(\frac{1}{n_{e}^{2}})\|_{L^{\infty}}\|\partial_{x}^{2}(\partial_{x}N_{e}\partial_{x}^{3}N_{e})\|\\
&+\|\partial_{t}\partial_{x}^{3}(\frac{1}{n_{e}^{2}})\|\|\partial_{x}N_{e}\|_{L^{\infty}}\|\partial_{x}^{3}N_{e}\|_{L^{\infty}}.
\end{split}
\end{equation*}
Note that the estimate of $\|\partial_{tx}(\frac{1}{n_{e}^{2}})\|_{L^{\infty}}$ is similar to that for \eqref{use 27}. By direct computation, we note that
\begin{equation}
\begin{split}\label{m7}
\|\partial_{x}^{2}(\partial_{x}N_{e}\partial_{x}^{3}N_{e})\|
\leq C(\|\partial_{x}^{3}N_{e}\|_{L^{\infty}}\|\partial_{x}^{3}N_{e}\|
+\|\partial_{x}N_{e}\|_{L^{\infty}}\|\partial_{x}^{5}N_{e}\|
+\|\partial_{x}^{2}N_{e}\|_{L^{\infty}}\|\partial_{x}^{4}N_{e}\|),
\end{split}
\end{equation}
and
\begin{equation}
\begin{split}\label{m8}
\|\partial_{x}^{3}(\partial_{x}N_{e}\partial_{x}^{3}N_{e})\|
\leq C(\|\partial_{x}^{3}N_{e}\|_{L^{\infty}}\|\partial_{x}^{4}N_{e}\|
+\|\partial_{x^{2}}N_{e}\|_{L^{\infty}}\|\partial_{x}^{5}N_{e}\|
+\|\partial_{x}N_{e}\|_{L^{\infty}}\|\partial_{x}^{6}N_{e}\|).
\end{split}
\end{equation}
Thus by \eqref{use 27}, \eqref{h2}, \eqref{m7}, \eqref{m8} and \eqref{use 4}, Sobolev embedding theorem and Lemma \ref{L2}-\ref{L3}, we have
\begin{equation*}
\begin{split}
I_{211}\leq C_{1}(1+\epsilon^{2}|\!|\!|(N_{e},U)|\!|\!|_{\epsilon}^{6})
(1+|\!|\!|(N_{e},U)|\!|\!|_{\epsilon}^{2}).
\end{split}
\end{equation*}

\emph{Estimate of $I_{212}$.} Using commutator notation \eqref{w}, we have
\begin{equation*}
\begin{split}
I_{212}=&\frac{3 H^{2}}{4}\epsilon^{6}\int(\frac{n_{e}}{n_{i}}\partial_{x}^{3}N_{e}
-\frac{\epsilon H^{2}}{4}\frac{1}{n_{e}n_{i}}\partial_{x}^{5}N_{e})
\partial_{x}^{3}\big\{
\frac{1}{n_{e}^{2}}\partial_{t}(\partial_{x}N_{e}\partial_{x}^{3}N_{e})\big\}\\
=&\frac{3 H^{2}}{4}\epsilon^{6}\int(\frac{n_{e}}{n_{i}}\partial_{x}^{3}N_{e}
-\frac{\epsilon H^{2}}{4}\frac{1}{n_{e}n_{i}}\partial_{x}^{5}N_{e})
\big\{[\partial_{x}^{3},\frac{1}{n_{e}^{2}}]\partial_{t}(\partial_{x}N_{e}\partial_{x}^{3}N_{e})\\
&\ \ \ \ \ \ \ \ \ \ \ \ \ +\frac{1}{n_{e}^{2}}\partial_{t}\partial_{x}^{3}(\partial_{x}N_{e}\partial_{x}^{3}N_{e})\big\}\\
=:&I_{2121}+I_{2122}.
\end{split}
\end{equation*}
By commutator estimate \eqref{w11} in Lemma \ref{L9}, we have
\begin{equation*}
\begin{split}
\|[\partial_{x}^{3},\frac{1}{n_{e}^{2}}]\partial_{t}(\partial_{x}N_{e}\partial_{x}^{3}N_{e})\|
\leq&\|\partial_{x}(\frac{1}{n_{e}^{2}})\|_{L^{\infty}}\|\partial_{x}^{2}(\partial_{t}(\partial_{x}N_{e}\partial_{x}^{3}N_{e}))\|\\
&+\|\partial_{x}^{3}(\frac{1}{n_{e}^{2}})\|\|\partial_{t}(\partial_{x}N_{e}\partial_{x}^{3}N_{e})\|_{L^{\infty}}.
\end{split}
\end{equation*}
By direct computation, we have
\begin{equation}
\begin{split}\label{m88}
\|\partial_{t}(\partial_{x}N_{e}\partial_{x}^{3}N_{e})\|_{L^{\infty}}
\leq C(\|\partial_{tx}N_{e}\|_{L^{\infty}}\|\partial_{x}^{3}N_{e}\|_{L^{\infty}}
+\|\partial_{x}N_{e}\|_{L^{\infty}}\|\partial_{t}\partial_{x}^{3}N_{e}\|_{L^{\infty}},
\end{split}
\end{equation}
and
\begin{equation}
\begin{split}\label{m9}
\|\partial_{t}\partial_{x}^{2}(\partial_{x}N_{e}\partial_{x}^{3}N_{e})\|
\leq &C(\|\partial_{x}N_{e}\|_{L^{\infty}}\|\partial_{t}\partial_{x}^{5}N_{e}\|
+\|\partial_{x}^{2}N_{e}\|_{L^{\infty}}\|\partial_{t}\partial_{x}^{4}N_{e}\|\\
&+\|\partial_{x}^{3}N_{e}\|_{L^{\infty}}\|\partial_{t}\partial_{x}^{3}N_{e}\|
+\|\partial_{x}^{4}N_{e}\|_{L^{\infty}}\|\partial_{t}\partial_{x}^{2}N_{e}\|\\
&+\|\partial_{x}^{5}N_{e}\|_{L^{\infty}}\|\partial_{tx}N_{e}\|.
\end{split}
\end{equation}
Thus by \eqref{use 1},\ \eqref{use 17},\  \eqref{m88}\  and\ \eqref{m9},  Sobolev embedding theorem and Lemma \ref{L2}-\ref{L3},
\begin{equation*}
\begin{split}
I_{2121}\leq C_{1}(1+\epsilon^{2}|\!|\!|(N_{e},U)|\!|\!|_{\epsilon}^{4})
(1+|\!|\!|(N_{e},U)|\!|\!|_{\epsilon}^{2}).
\end{split}
\end{equation*}
By integration by parts, we have
\begin{equation*}
\begin{split}
I_{2122}=&-\frac{3 H^{2}}{4}\epsilon^{6}\int\big\{\partial_{x}(\frac{1}{n_{e}n_{i}})\partial_{x}^{3}N_{e}
-\frac{\epsilon H^{2}}{4}\partial_{x}(\frac{1}{n_{e}^{3}n_{i}})\partial_{x}^{5}N_{e}\big\}
\partial_{t}\partial_{x}^{2}(\partial_{x}N_{e}\partial_{x}^{3}N_{e})\\
&-\frac{3 H^{2}}{4}\epsilon^{6}\int\big\{\frac{1}{n_{e}n_{i}}\partial_{x}^{4}N_{e}
-\frac{\epsilon H^{2}}{4}\frac{1}{n_{e}^{3}n_{i}}\partial_{x}^{6}N_{e}\big\}
\partial_{t}\partial_{x}^{2}(\partial_{x}N_{e}\partial_{x}^{3}N_{e}).
\end{split}
\end{equation*}
By direct computation, we note that $\partial_{x}\frac{1}{n_{e}^{3}n_{i}}$ has same estimate with \eqref{use 11}, thus by \eqref{m9}, Sobolev embedding theorem and Lemma \ref{L1}-\ref{L3}, we have
\begin{equation*}
\begin{split}
I_{2122}\leq C_{1}(1+\epsilon^{2}|\!|\!|(N_{e},U)|\!|\!|_{\epsilon}^{2})
(1+|\!|\!|(N_{e},U)|\!|\!|_{\epsilon}^{2}).
\end{split}
\end{equation*}
Therefore, we have
\begin{equation*}
\begin{split}
I_{212}\leq C_{1}(1+\epsilon^{2}|\!|\!|(N_{e},U)|\!|\!|_{\epsilon}^{4})
(1+|\!|\!|(N_{e},U)|\!|\!|_{\epsilon}^{2}),
\end{split}
\end{equation*}
and hence
\begin{equation*}
\begin{split}
I_{21}\leq C_{1}(1+\epsilon^{2}|\!|\!|(N_{e},U)|\!|\!|_{\epsilon}^{6})
(1+|\!|\!|(N_{e},U)|\!|\!|_{\epsilon}^{2}).
\end{split}
\end{equation*}
The estimates of $I_{12},\  I_{14},\ I_{16},\ I_{17},\ I_{19}\ and\ I_{20}$ are similar to that of $I_{18}$. According to the Lemma \ref{L8}, we have
\begin{equation*}
\begin{split}
I_{22}\leq C_{1}(1+\epsilon^{2}|\!|\!|(N_{e},U)|\!|\!|_{\epsilon}^{4})
(1+|\!|\!|(N_{e},U)|\!|\!|_{\epsilon}^{2}).
\end{split}
\end{equation*}
The proof of Lemma \ref{Lem-u2} is then complete.
\end{proof}

\section{Proof of Theorem \ref{thm1}}
\begin{proof}[\textbf{Proof of Theorem \ref{thm1} }]
Adding Propositions \ref{P1-P2} with $\gamma=0,1,2$  and Proposition \ref{P3} together, we obtain
\begin{equation}\label{Gron}
\begin{split}
\frac{1}{2}&\frac{d}{dt}(\|U\|_{H^{2}}^{2}+\epsilon\|\partial_{x}^{3}U\|_{L^{2}}^{2})
+\frac{1}{2}\frac{d}{dt}\Big\{\int\frac{n_{e}}{n_{i}}(\sum_{i=0}^{2}| \partial_{x}^{i}N_{e}|^{2}+\epsilon|\partial_{x}^{3}N_{e}|^{2})\Big\}\\
&+\frac{1}{2}\frac{d}{dt}\Big\{\int\epsilon(\frac{n_{e}^{2}}{n_{i}}+\frac{H^{2}}{8n_{e}n_{i}})(\sum_{i=0}^{3}| \partial_{x}^{i}N_{e}|^{2}+\epsilon|\partial_{x}^{4}N_{e}|^{2})\Big\}\\
&+\frac{3H^{2}}{16}\frac{d}{dt}\Big\{\int\frac{\epsilon^{2}}{n_{i}}(\sum_{i=0}^{4}| \partial_{x}^{i}N_{e}|^{2}+\epsilon|\partial_{x}^{5}N_{e}|^{2})\Big\}\\
&+\frac{H^{2}}{8}\frac{d}{dt}\Big\{\int\frac{\epsilon^{3}}{n_{e}^{2}n_{i}}(\sum_{i=0}^{5}| \partial_{x}^{i}N_{e}|^{2}+\epsilon|\partial_{x}^{6}N_{e}|^{2})\Big\}\\
\leq &C(1+\epsilon^{2}|\!|\!|(N_{e},U)|\!|\!|_{\epsilon}^{6})(1+|\!|\!|(N_{e},U)|\!|\!|_{\epsilon}^{2}).
\end{split}
\end{equation}
Integrating the inequality \eqref{Gron} over $(0,t)$ yields
\begin{equation*}
\begin{split}
|\!|\!|(N_{e},U)(t)|\!|\!|_{\epsilon}^{2}
&\leq C|\!|\!|(N_{e},U)(0)|\!|\!|_{\epsilon}^{2}
+\int_{0}^{t}C_{1}(1+\epsilon^{2}|\!|\!|(N_{e},U)|\!|\!|_{\epsilon}^{6})
(1+|\!|\!|(N_{e},U)|\!|\!|_{\epsilon}^{2})ds\\
&\leq C_{1}|\!|\!|(N_{e},U)(0)|\!|\!|_{\epsilon}^{2}+\int_{0}^{t}C_{1}
(1+\epsilon^{2}\tilde{C})
(1+|\!|\!|(N_{e},U)|\!|\!|_{\epsilon}^{2})ds,
\end{split}
\end{equation*}
where $C$ is an absolute constant.

Recall that $C_1$ depends on $|\!|\!|(N_{e},U)|\!|\!|^2_{\epsilon}$ through $\epsilon|\!|\!|(N_{e},U)|\!|\!|^2_{\epsilon}$ and is nondecreasing. Let $C_1'=C_1(1)$ and $C_2>C\sup_{\epsilon<1}|\!|\!|(u^{\epsilon}_R,\phi^{\epsilon}_R) (0)|\!|\!|^2_{\epsilon}$. For any arbitrarily given $\tau>0$, we choose $\tilde C$ sufficiently large such that $\tilde C>e^{4C_1'\tau}(1+C_2)(1+C_1')$. Then there exists $\epsilon_0>0$ such that ${\epsilon}\tilde C\leq 1$ for all $\epsilon<\epsilon_0$, we have
\begin{equation}\label{e999}
\begin{split}
\sup_{0\leq t\leq\tau}|\!|\!|(N_{e},U)(t)|\!|\!|^2_{\epsilon}\leq e^{4C_1'\tau}(C_2+1)<\tilde C.
\end{split}
\end{equation}
In particular, we have  the uniform bound for $(N_{e},U)$,
\begin{equation}
\begin{split}\label{final}
\sup_{0\leq t\leq\tau}\Big(&\|(N_{e},U)\|_{H^{2}}^{2} +\epsilon\|(\partial_{x}^{3}N_{e}, \partial_{x}^{3}U)\|_{L^{2}}^{2}\\
&+\epsilon^{2}\|\partial_{x}^{4}N_{e}\|_{L^{2}}^{2} +\epsilon^{3}\|\partial_{x}^{5}N_{e}\|_{L^{2}}^{2} +\epsilon^{4}\|\partial_{x}^{6}N_{e}\|_{L^{2}}^{2}\Big) \leq \tilde C.
\end{split}
\end{equation}
On the other hand, by Lemma \ref{L1} and \eqref{final}, we have
\begin{equation*}
\begin{split}
\sup_{0\leq t\leq\tau}\|N_{i}\|_{H^2}^2\leq \tilde C.
\end{split}
\end{equation*}
It is now standard to obtain uniform estimates independent of $\epsilon$ by the continuity method.

\end{proof}

\end{document}